\documentclass[12pt,letterpaper]{amsart}

\usepackage{amsfonts,amssymb,amsmath,gensymb,color,epic,eepic,float,fullpage}
\usepackage[mathscr]{euscript}
\usepackage[T1]{fontenc}
\usepackage{mathtools}
\usepackage{comment}
\usepackage{bm,bbm}
\usepackage{enumerate,enumitem}
\usepackage{caption}
\usepackage{hyperref}
\usepackage{xcolor}
\usepackage{tikz}
\usetikzlibrary{graphs,patterns,decorations.markings,arrows,matrix}
\usetikzlibrary{calc,decorations.pathmorphing,decorations.pathreplacing,shapes}
\tikzstyle{block} = [rectangle,draw,text width=10em,text centered,rounded corners,minimum height=4em]
\tikzstyle{line} = [draw, -latex']
\allowdisplaybreaks

\makeatletter
\DeclareFontFamily{OMX}{MnSymbolE}{}
\DeclareSymbolFont{MnLargeSymbols}{OMX}{MnSymbolE}{m}{n}
\SetSymbolFont{MnLargeSymbols}{bold}{OMX}{MnSymbolE}{b}{n}
\DeclareFontShape{OMX}{MnSymbolE}{m}{n}{
    <-6>  MnSymbolE5
   <6-7>  MnSymbolE6
   <7-8>  MnSymbolE7
   <8-9>  MnSymbolE8
   <9-10> MnSymbolE9
  <10-12> MnSymbolE10
  <12->   MnSymbolE12
}{}
\DeclareFontShape{OMX}{MnSymbolE}{b}{n}{
    <-6>  MnSymbolE-Bold5
   <6-7>  MnSymbolE-Bold6
   <7-8>  MnSymbolE-Bold7
   <8-9>  MnSymbolE-Bold8
   <9-10> MnSymbolE-Bold9
  <10-12> MnSymbolE-Bold10
  <12->   MnSymbolE-Bold12
}{}

\let\llangle\@undefined
\let\rrangle\@undefined
\DeclareMathDelimiter{\llangle}{\mathopen}%
                     {MnLargeSymbols}{'164}{MnLargeSymbols}{'164}
\DeclareMathDelimiter{\rrangle}{\mathclose}%
                     {MnLargeSymbols}{'171}{MnLargeSymbols}{'171}
\makeatother

\newtheorem{defn}{Definition}[section]

\newtheorem{thm}[defn]{Theorem}

\newtheorem{cor}[defn]{Corollary}
\newtheorem{remark}[defn]{Remark}
\newtheorem{prop}[defn]{Proposition}

\def\wh{\widehat}
\def\wt{\widetilde}
\def\bw{\overline{w}}
\def\bW{\overline{W}}
\renewcommand{\ast}[1]{#1^*}
\renewcommand{\b}[1]{#1^*}
\renewcommand{\overline}[1]{#1^*}

\def\leq{\leqslant}
\def\geq{\geqslant}

\def\i{{\sf i}}

\def\ie{{\it i.e.}\/}
\def\cf{{\it cf.}\/}

\def\re{\color{red}}
\colorlet{lgray}{white!70!black}
\colorlet{lred}{white!85!red}

\def\hl{spin Hall--Littlewood\ }
\def\Hl{Spin Hall--Littlewood\ }
\def\wh{spin $q$--Whittaker\ }
\def\Wh{Spin $q$--Whittaker\ }

\def\F{{\sf F}}
\def\G{{\sf G}}
\def\Gc{\G^{*}}
\def\c{{\sf c}}

\newcommand{\bra}[1]{\langle #1|}
\newcommand{\bbra}[1]{\llangle #1|}
\newcommand{\ket}[1]{|#1\rangle}
\newcommand{\kett}[1]{|#1\rrangle}

\renewcommand{\vert}[5]{
\begin{gathered}
\begin{tikzpicture}[scale=#5,baseline=(current bounding box.center)]
\draw[lgray,thick] (-1,0) -- (1,0);
\draw[lgray,line width=5pt] (0,-1) -- (0,1);
\node[left] at (-0.8,0) {\tiny $#2$};\node[right] at (0.8,0) {\tiny $#4$};
\node[below] at (0,-0.8) {\tiny $#1$};\node[above] at (0,0.8) {\tiny $#3$};
\end{tikzpicture}
\end{gathered}
}

\newcommand{\fvert}[5]{
\begin{gathered}
\begin{tikzpicture}[scale=#5,baseline=(current bounding box.center)]
\draw[lgray,line width=5pt] (-1,0) -- (1,0);
\draw[lgray,line width=5pt] (0,-1) -- (0,1);
\node[left] at (-0.8,0) {\tiny $#2$};\node[right] at (0.8,0) {\tiny $#4$};
\node[below] at (0,-0.8) {\tiny $#1$};\node[above] at (0,0.8) {\tiny $#3$};
\end{tikzpicture}
\end{gathered}
}

\tikzset{->-/.style={decoration={markings,
mark=at position #1 with {\arrow{>}}},postaction={decorate}}}

\title[Spin $q$--Whittaker polynomials]{Spin $q$--Whittaker polynomials}

\author{Alexei Borodin}

\address[Alexei Borodin]{ Department of Mathematics, MIT, Cambridge, USA, and
Institute for Information Transmission Problems, Moscow, Russia. E-mail: borodin@math.mit.edu }

\author{Michael Wheeler}

\address[Michael Wheeler]{ School of Mathematics and Statistics, The University of Melbourne, Parkville,
Victoria 3010, Australia. E-mail: wheelerm@unimelb.edu.au}

\begin{document}

\begin{abstract} We introduce and study a one-parameter generalization of the 
$q$--Whittaker symmetric functions. This is a family of multivariate symmetric 
polynomials, whose construction may be viewed as an application of the 
procedure of fusion from integrable lattice models to a vertex model 
interpretation of a one-parameter generalization of Hall--Littlewood polynomials 
from \cite{Borodin, BorodinP1, BorodinP2}.

We prove branching and Pieri rules, standard and dual (skew) Cauchy
summation identities, and an integral representation for the new polynomials.
\end{abstract}

\maketitle

\setcounter{tocdepth}{1}
\tableofcontents

\section{Introduction}

\subsection{Background}

The $q$-deformed class one $\mathfrak{gl}_{n}$ Whittaker functions, or {\it 
$q$--Whittaker functions} for short, were defined in \cite{GLO1} as a special 
class of joint eigenfuctions of the $q$-deformed Toda chain Hamiltonians 
\cite{Etingof,Ruijsenaars} with the support in the positive Weyl chamber. In 
the limit $q \rightarrow 1$ they reduce to classical $\mathfrak{gl}_{n}$ 
Whittaker functions \cite{GLO2}, while for general $q$ they are themselves the 
limiting case of a more general family of symmetric functions, the Macdonald 
polynomials \cite{Macdonald}. In the last decade they have come to play a 
prominent role in integrable probability via the $q$--Whittaker processes 
\cite{BorodinC, BPlectures}. These are a rather general class of stochastic 
processes, which not only degenerate at $q \rightarrow 1$ to the Whittaker 
processes of O'Connell \cite{O'Connell} used to describe random directed 
polymers but, in a different direction, to the continuous time $q$-deformed 
totally asymmetric simple exclusion process ($q$--TASEP). For a survey 
of these properties, and their connection with KPZ (Kardar--Parisi--Zhang) 
universality, we refer the reader to \cite[Chapters 3--5]{BorodinC}, 
\cite{BPlectures}.

A somewhat perpendicular approach to KPZ-type observables has recently been 
proposed in \cite{BorodinP1,BorodinP2}, in the setting of a higher spin version 
of the six-vertex model from statistical mechanics \cite{Baxter}. The chief 
object of study in these works has been the {\it height function} of the vertex 
model, a random (non-negative, integer-valued) variable which lives on the 
vertices of the lattice. Once again, the theory of symmetric functions turns out 
to be indispensable to the calculations: a family of symmetric, rational 
functions and their associated orthogonality relations play a pivotal role in 
writing exact integral expressions for the $q$-moments of the height function. 
These rational functions were originally introduced in \cite{Borodin} as a 
one-parameter deformation of the {\it Hall--Littlewood functions} 
\cite{Macdonald} (which are yet another family encompassed by the more general 
Macdonald polynomials). In this work we refer to the former rational functions 
as {\it spin Hall--Littlewood functions,} in reference to the spin parameter 
$s$ which they carry.

Even more recently, some direct equivalences between expectations of 
observables in Macdonald processes (and their degenerations) and those in higher 
spin vertex models have been remarked \cite{Borodin2}, see also 
\cite{OrrPetrov}. While these equivalences (once guessed) can be readily proved 
by comparing integral formulae, their origin remains fairly mysterious. A more 
conceptual version of the equivalence \cite{Borodin2}, albeit in the limit to 
Hall--Littlewood processes, has now been exposed \cite{BorodinBW}. A natural 
angle towards better understanding these equivalences is to search for a family 
of symmetric functions which unifies all of the classes listed above, and to 
study in detail the properties of such a family. This remains beyond the scope 
of the present work, although we mention that promising steps in this direction 
have been made in \cite{GarbaliGW}, where a mutual generalization of Macdonald 
polynomials and \hl functions was explicitly constructed.

In fact, rather than unifying the known families, in this paper we will add one more family to the list: these we term {\it spin $q$--Whittaker polynomials.} The name is chosen to reflect the fact that they are a one-parameter, $s$-deformation of $q$--Whittaker polynomials; they degenerate to the latter at $s=0$. This means that the \wh polynomials have a similar footing in the general theory as do the \hl functions, but now on the $t=0$ side of the Macdonald ``coin'' (whereas the \hl functions lie on the $q=0$ side). We refer the reader to Figure \ref{fig:landscape} for a survey of the symmetric function landscape.

\begin{figure}
\resizebox{14cm}{6cm}{
\begin{tikzpicture}[node distance = 4.5cm,auto,>=stealth]
\node [block] (M) {{\bf Macdonald} \\ $(q,t)$};
\node [block,below left of=M] (H) {{\bf Hall--Littlewood} \\ $(t)$};
\node [block,below right of=M] (W) {{\bf $q$--Whittaker} \\ $(q)$};
\node[block,above left of=H] (sH) {{\bf Spin} \\ {\bf Hall--Littlewood} \\ $(s,t)$};
\node[block,above right of=W] (sW) {{\bf Spin} \\ {\bf $q$--Whittaker} \\ $(s,q)$};
\path [line] (M) -- node[left] {$q=0$} (H);
\path [line] (M) -- node[right] {\ $t=0$} (W);
\path [line] (sH) -- node[left] {$s=0$} (H);
\path [line] (sW) -- node[right] {$s=0$} (W);
\draw [bend left,->] (sH) to node[above] {fusion + $(t \mapsto q)$} (sW);
\end{tikzpicture}
}
\caption{A table of symmetric functions, indicating their parameter-dependence.}
\label{fig:landscape}
\end{figure}
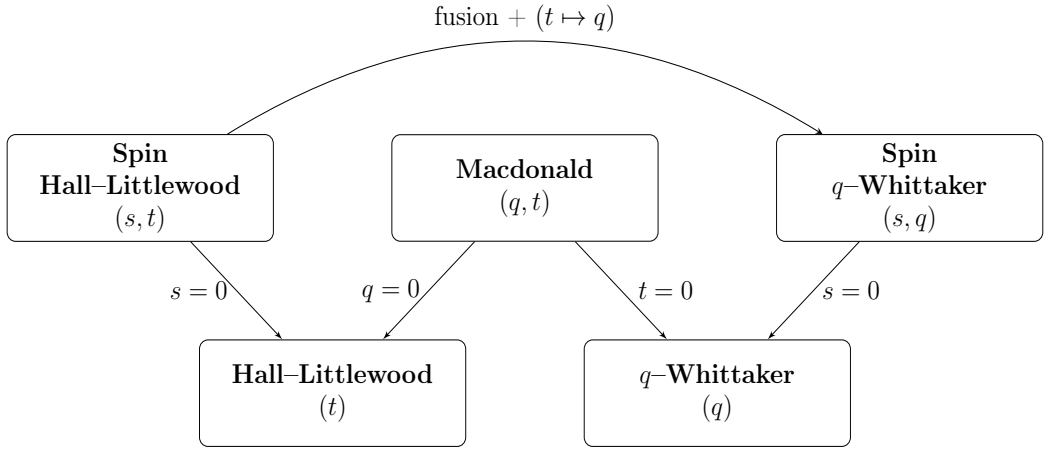

This analogy between the \hl functions and \wh polynomials is not accidental: one can construct the latter directly from the former, using the technique of {\it fusion} in integrable lattice models \cite{KulishRS,KirillovR}. To see how this can transpire, consider the right hand side of the Cauchy identity for Hall--Littlewood polynomials (here we use a parameter $q$, rather than the conventional $t$),
\begin{align*}
\Pi(u_1,u_2,\dots,u_m;v_1,v_2,\dots,v_n) := \prod_{i=1}^{m} \prod_{j=1}^{n}  
\frac{1-q u_i v_j}{1-u_i v_j}\,,
\end{align*}
and specialize one set of these variables to a geometric progression with ratio $q$, letting $(u_1,u_2,\dots,u_m) = 
(u,qu,\dots,q^{m-1}u)$. We obtain
\begin{align*}
\Pi(u,qu,\dots,q^{m-1}u;v_1,v_2,\dots,v_n) = \prod_{j=1}^{n}  \frac{1-q^m 
u v_j}{1-u v_j}\,.
\end{align*}
After analytic continuation in $q^m$ effected by sending $q^m u \mapsto -x$, and setting $u=0$, the product turns into 
$\prod_{j=1}^{n} (1+x v_j)$. This we recognize as the right hand side of the dual Cauchy identity between a Hall--Littlewood polynomial and a $q$--Whittaker polynomial (in one variable). On the other hand, such geometric specializations are known to be a key ingredient in fusion. It is therefore reasonable to suggest that $q$--Whittaker polynomials might be recovered by applying the fusion technique to the lattice model construction of Hall--Littlewood polynomials. This indeed turns out to be the case; surprisingly, we were unable to find a precise statement of this fact in the literature. 

In light of this observation, one can play the same game taking the \hl 
functions (and their lattice model construction, developed in 
\cite{Borodin,BorodinP1,BorodinP2}) as the starting point: doing so ultimately 
leads us to our definition of the \wh polynomials.

As is common in the theory of symmetric functions \cite{Macdonald}, we define 
skew \wh polynomials $\mathbb{F}_{\lambda/\nu}(x_1,\dots,x_m)$ parametrized by 
pairs of partitions $\lambda,\nu$, as well as non-skew ones 
$\mathbb{F}_{\lambda}(x_1,\dots,x_m)$ parametrized by single partitions 
$\lambda$; the latter correspond to taking $\nu=\varnothing$ in the former. 
Let us list a few of their properties that we prove below. 
\begin{itemize}[leftmargin=*]
\item $\mathbb{F}_{\lambda/\nu}(x_1,\dots,x_m)$ is a symmetric, {\it 
inhomogeneous} polynomial in $(x_1,\dots,x_m)$.

\medskip

\item For any pair of partitions $\lambda,\nu$ the following stability relation holds:
\begin{align*}
\mathbb{F}_{\lambda/\nu}(x_1,\dots,x_{m-1},-s)
=
\mathbb{F}_{\lambda/\nu}(x_1,\dots,x_{m-1}).
\end{align*}
At $s=0$ this recovers the usual stability property of $q$--Whittaker polynomials.

\medskip

\item 
The \wh polynomials satisfy a standard branching rule
$$
\mathbb{F}_{\lambda / \mu}(x_1,\dots,x_{m+n})
=
\sum_{\nu}
\mathbb{F}_{\nu / \mu}(x_1,\dots,x_m)
\mathbb{F}_{\lambda / \nu}(x_{m+1},\dots,x_{m+n}).
$$

\item The one-variable skew \wh polynomial is given by
\begin{align*}
\mathbb{F}_{\lambda/\nu}(x)
=
\left\{
\begin{array}{ll}
x^{|\lambda|-|\nu|}
\displaystyle{
\prod_{i \geq 1}
\frac{(-s/x;q)_{\lambda_i-\nu_i} (-sx;q)_{\nu_i-\lambda_{i+1}} (q;q)_{\lambda_i-\lambda_{i+1}}}
{(q;q)_{\lambda_i-\nu_i} (q;q)_{\nu_i-\lambda_{i+1}} (s^2;q)_{\lambda_i-\lambda_{i+1}}}
},
\qquad
&
\lambda \succ \nu,
\\ \\
0,
\qquad
&
{\rm otherwise.}
\end{array}
\right.
\end{align*}
Together with the branching rule, this yields a simple ``interlacing'' 
construction of the \wh polynomials (or equivalently, one in terms of 
Gelfand--Tsetlin patterns).

\medskip

\item Define the \emph{dual} \wh polynomials by 
$$
\mathbb{F}_{\lambda/\nu}^*(x_1,\dots,x_m)=\prod_{i\ge 1} 
\frac{(s^2;q)_{\lambda_i-\lambda_{i+1}}}{(q;q)_{\lambda_i-\lambda_{i+1}}}\frac{
(q;q)_{\nu_i-\nu_{i+1}}}{(s^2;q)_{\nu_i-\nu_{i+1}}}\cdot\mathbb{F}_{\lambda/\nu}
(x_1,\dots,x_m).
$$
Then the following skew Cauchy-type summation identities hold: For arbitrary 
partitions $\mu$ and $\nu$, one has
\begin{multline*}
\sum_{\lambda}
\mathbb{F}_{\lambda/\mu}(x_1,\dots,x_m)
\mathbb{F}^{*}_{\lambda/\nu}(y_1,\dots,y_n)
\\
=\prod_{i=1}^{m}
\prod_{j=1}^{n}
\frac{(-sx_i;q)_{\infty} (-sy_j;q)_{\infty}}
{(s^2;q)_{\infty} (x_i y_j;q)_{\infty}}
\sum_{\kappa}
\mathbb{F}_{\nu/\kappa}(x_1,\dots,x_m)
\mathbb{F}^{*}_{\mu/\kappa}(y_1,\dots,y_n).
\end{multline*}
For $\mu=\nu=\varnothing$, this turns into a Cauchy-type identity
\begin{align*}
\sum_{\lambda}
\mathbb{F}_{\lambda}(x_1,\dots,x_m)
\mathbb{F}_{\lambda}(y_1,\dots,y_n)
\prod_{i \geq 1}
\frac{(s^2;q)_{\lambda_i-\lambda_{i+1}}}
{(q;q)_{\lambda_i-\lambda_{i+1}}}
=
\prod_{i=1}^{m}
\prod_{j=1}^{n}
\frac{(-sx_i;q)_{\infty} (-sy_j;q)_{\infty}}
{(s^2;q)_{\infty} (x_i y_j;q)_{\infty}},
\end{align*}
which is a multivariate generalization of the $q$--Gauss summation theorem.

There are also dual (skew) Cauchy identities that combine \wh polynomials and 
stable \hl functions, see Section \ref{sec:dual-cauchy} below. 

\medskip 

\item The following integral representation holds:
\begin{multline*}
\mathbb{F}_{\lambda}(x_1,\dots,x_m)
\\=
\oint_{\mathcal{C}} \frac{du_1}{2\pi\i u_1}
\cdots
\oint_{\mathcal{C}} \frac{du_L}{2\pi\i u_L}
\prod_{1 \leq i<j \leq L}
\left(
\frac{u_i-u_j}{u_i-q u_j}
\right)
\prod_{i=1}^{L}
\left( 
\frac{1-su_i}{u_i-s}
\right)^{\lambda'_i}
\left(
\frac{\prod_{j=1}^{m}(1+u_i x_j)}{(1-s u_i)^{m+1}}
\right),
\end{multline*}
where $L=\lambda_1$ denotes the largest part of $\lambda$ and the contour 
$\mathcal{C}$ is as in Figure \ref{fig:contour} below. 
\end{itemize}
 
The vertex model interpretation of the \hl functions makes it natural to add 
countably many inhomogeneity parameters to their definition; most of their 
properties remain intact \cite{BorodinP1}. One can do the same for the \wh 
polynomials; the branching rule remains the same, and in the formula for the 
one-variable specialization above one needs to add index $i$ to the spin
parameter $s$. This preserves the (skew) dual Cauchy identities (that need to 
involve stable \hl functions with the same inhomogeneities), but appears to 
destroy the ordinary (skew) Cauchy identities. There is also an integral 
representation that generalizes the one above. We decided to leave the 
inhomogeneous case out of this work in order not to cloud the arguments
with more involved notation.

\subsection{Layout of paper}

In Section \ref{sec:vertex_model} we recall the basic features of the 
integrable higher spin vertex model studied in 
\cite{Borodin,BorodinP1,BorodinP2}, and then in Section \ref{sec:spin_hl_funct} 
we use it to define the \hl rational functions as lattice model partition 
functions. In Section 
\ref{sec:fusion} we gather a number of results on the fusion procedure as 
applied to the vertex model of Section \ref{sec:vertex_model}, leading to the 
construction of integrable Boltzmann weights in equation \eqref{fused_wt} that 
form the foundation of the remainder of the paper. We remark that, at $s=0$, 
these weights degenerate into those used by Korff in Sections 3 and 6 of \cite{Korff}. 
This is expected, since \cite{Korff} contains an integrable lattice construction 
of the $q$--Whittaker polynomials which is exactly the $s=0$ specialization of 
our results here. 

In Section \ref{sec:wh_poly} we apply the fusion procedure to the \hl functions, and use the resulting partition functions to define \wh polynomials. Along the way, we also define a {\it stable} version of the \hl functions, which are later paired with the \wh polynomials in dual Cauchy summation identities. We formulate the \wh polynomials algebraically using certain monodromy matrix operators in the model \eqref{fused_wt}, and derive key commutation relations between these operators. In Section \ref{sec:combin} we use the lattice construction of the \wh polynomials to derive their branching rules, and show that they reduce to (ordinary) $q$--Whittaker polynomials by specializing $s=0$. Section \ref{sec:cauchy_pieri} contains a list of various Cauchy, dual Cauchy and Pieri identities which the \wh polynomials and the stable \hl functions satisfy; these are all proven by means of the commutation relations derived in Section \ref{sec:wh_poly}. We conclude, in Section \ref{sec:integral}, with a multiple 
integral expression for the the \wh polynomials.

\subsection{Notation}

Throughout the paper we use a number of partition-related terminologies, which we summarize below. These are all standard in the combinatorics literature, with one exception: we include parts of size zero in our partitions, rather than the usual practice of truncating a partition after its last positive part.

A {\it partition} $\lambda$ is a finite non-increasing sequence of non-negative integers $\lambda_1 \geq \cdots \geq \lambda_{\ell} \geq 0$, where $\lambda_i$ is called a {\it part} of $\lambda$ for all $1 \leq i \leq \ell$. The empty partition $\varnothing$ is the trivial sequence consisting of no parts. The {\it length} of a partition $\lambda$ is the number of parts which comprise it, and denoted by $\ell(\lambda)$. We let ${\rm Part}_{\ell}$ denote the set $\{\lambda_1 \geq \cdots \geq \lambda_{\ell} \geq 0\}$ of all partitions of length $\ell$. Similarly, ${\rm Part}^{+}_{\ell}$ will denote the set $\{\lambda_1 \geq \cdots \geq \lambda_m \geq 1\}_{0 \leq m \leq \ell}$ of all partitions with purely positive parts, whose length is bounded by $\ell$. It is sometimes convenient to write a partition in terms of its {\it part-multiplicities} $m_i(\lambda) = \#\{j:\lambda_j = i\}$, \ie\ by writing $\lambda = 0^{m_0} 1^{m_1} 2^{m_2} \dots$ where $m_i \equiv m_i(\lambda)$. The {\it conjugate} of a partition $\lambda$, denoted $\lambda'$, is the partition with parts $\lambda'_i = \#\{j:\lambda_j \geq i\}$. 

For two positive partitions $\lambda, \mu$ we write $\lambda \supset \mu$ if $\ell(\lambda) \geq \ell(\mu)$ and the inequality $\lambda_i \geq \mu_i$ holds for all $1 \leq i \leq \ell(\mu)$. Similarly, we write $\lambda \succ \mu$ and say that $\lambda$ {\it interlaces} $\mu$ if $0 \leq \ell(\lambda)-\ell(\mu) \leq 1$ and $\lambda_i \geq \mu_i \geq \lambda_{i+1}$ for all $1 \leq i \leq \ell(\mu)$ (where the final inequality $\mu_{\ell(\mu)} \geq \lambda_{\ell(\mu)+1}$ is omitted in the case $\ell(\lambda) = \ell(\mu)$).

We make frequent use of $q$--Pochhammer symbols, which are defined as follows:
\begin{align*}
(a;q)_m
=
\left\{
\begin{array}{ll}
\prod_{1 \leq i \leq m}
(1-aq^{i-1}),
& \quad
m > 0,
\\
\\
1,
& \quad
m = 0,
\\
\\
\prod_{1 \leq i \leq -m}
(1-aq^{-i})^{-1},
& \quad
m < 0.
\end{array}
\right.
\end{align*}
We will also tacitly assume that $|q|<1$ so that, in particular, $(a;q)_{\infty}$ makes sense. Most of our equations depend on another parameter $s$; we will also assume that $|s|<1$, since this prevents $s$ from taking values which would lead to divergences because of vanishing denominators.

\subsection{Acknowledgments}

A.~B. is supported by the National Science Foundation grant DMS-1607901 and by 
Fellowships of the Radcliffe Institute for Advanced Study and the Simons 
Foundation.  
M.~W. is supported by the Australian Research Council grant DE160100958.

\section{Higher spin vertex model}
\label{sec:vertex_model}

\subsection{Vertex weights and Yang--Baxter equation}

Following Section 2 of \cite{BorodinP2}, we define a vertex model consisting of SW $\rightarrow$ NE oriented paths on a square grid. Horizontal edges of the grid can be occupied by at most one lattice path, but no restriction is placed on the number of paths which traverse a vertical edge. Every intersection of horizontal and vertical gridlines constitutes a vertex, and each vertex is assigned a Boltzmann weight that depends on the local configuration of lattice paths about that intersection. Assuming conservation of lattice paths through a vertex, four types of vertex are possible. We indicate these vertices and their explicit weights below:
\begin{align}
\label{vertices}
\begin{array}{cccc}
\begin{tikzpicture}[scale=0.6,>=stealth]
\draw[lgray,thick] (-1,0) node[left,black] {$0$} -- (1,0) node[right,black] {$0$};
\draw[lgray,line width=5pt] (0,-1) -- (0,1);
\node[below] at (0,-1) {$g$};
\draw[thick,->] (-0.075,-1) -- (-0.075,1);
\draw[thick,->] (0.075,-1) -- (0.075,1);
\node[above] at (0,1) {$g$};
\end{tikzpicture}
\quad\quad\quad
&
\begin{tikzpicture}[scale=0.6,>=stealth]
\draw[lgray,thick] (-1,0) node[left,black] {$0$} -- (1,0) node[right,black] {$1$};
\draw[lgray,line width=5pt] (0,-1) -- (0,1);
\node[below] at (0,-1) {$g+1$};
\draw[thick,->] (-0.15,-1) -- (-0.15,1);
\draw[thick,->] (0,-1) -- (0,1);
\draw[thick,->] (0.15,-1) -- (0.15,0) -- (1,0);
\node[above] at (0,1) {$g$};
\end{tikzpicture}
\quad\quad\quad
&
\begin{tikzpicture}[scale=0.6,>=stealth]
\draw[lgray,thick] (-1,0) node[left,black] {$1$} -- (1,0) node[right,black] {$0$};
\draw[lgray,line width=5pt] (0,-1) -- (0,1);
\node[below] at (0,-1) {$g$};
\draw[thick,->] (-1,0) -- (-0.15,0) -- (-0.15,1);
\draw[thick,->] (0,-1) -- (0,1);
\draw[thick,->] (0.15,-1) -- (0.15,1);
\node[above] at (0,1) {$g+1$};
\end{tikzpicture}
\quad\quad\quad
&
\begin{tikzpicture}[scale=0.6,>=stealth]
\draw[lgray,thick] (-1,0) node[left,black] {$1$} -- (1,0) node[right,black] {$1$};
\draw[lgray,line width=5pt] (0,-1) -- (0,1);
\node[below] at (0,-1) {$g$};
\draw[thick,->] (-1,0) -- (-0.15,0) -- (-0.15,1);
\draw[thick,->] (0,-1) -- (0,1);
\draw[thick,->] (0.15,-1) -- (0.15,0) -- (1,0);
\node[above] at (0,1) {$g$};
\end{tikzpicture}
\\
w_u(g,0;g,0)
\quad\quad\quad
&
w_u(g+1,0;g,1)
\quad\quad\quad
&
w_u(g,1;g+1,0)
\quad\quad\quad
&
w_u(g,1;g,1)
\\ \\
\dfrac{1-s q^g u}{1-su}
\quad\quad\quad
&
\dfrac{(1-s^2 q^g) u}{1-su}
\quad\quad\quad
&
\dfrac{1-q^{g+1}}{1-su}
\quad\quad\quad
&
\dfrac{u-sq^g}{1-su}
\end{array}
\end{align}
where $g$ is any non-negative integer (representing the number of paths which sit at a vertical edge), $s$ and $q$ are fixed global parameters of the model, and $u$ is a local variable called the {\it spectral parameter.} We denote the Boltzmann weight of a vertex in two equivalent ways, and interchange between the two according to convenience:
\begin{align}
\label{boltz}
w_u
\left(
\begin{gathered}
\begin{tikzpicture}[scale=0.4,baseline=(current bounding box.center)]
\draw[lgray,thick] (-1,0) -- (1,0);
\draw[lgray,line width=5pt] (0,-1) -- (0,1);
\node[left] at (-0.8,0) {\tiny $j$};\node[right] at (0.8,0) {\tiny $\ell$};
\node[below] at (0,-0.8) {\tiny $i$};\node[above] at (0,0.8) {\tiny $k$};
\end{tikzpicture}
\end{gathered}
\right)
\equiv
w_u(i,j;k,\ell),
\quad
i,k \in \mathbb{Z}_{\geq 0},
\quad
0\leq j,\ell \leq 1.
\end{align}
It is conventional to relax the constraint of lattice path conservation through each vertex, which can be done by extending the Boltzmann weights \eqref{boltz} to all values of $i,j,k,\ell$ and assuming that $w_u(i,j;k,\ell) = 0$ for all $i+j \not= k + \ell$. The common factor $1-su$ in the denominator of each vertex \eqref{vertices} is to ensure that the empty vertex has weight $1$, \ie\ $w_u\left(0,0;0,0\right)=1$.  

Define an $n$-vertex by concatenating $n$ vertices vertically, with summation assumed over all internal vertical edges:
\begin{align}
\label{n-vert}
w_{\{u_1,\dots,u_n\}}
\left(
\begin{gathered}
\begin{tikzpicture}[scale=0.6,baseline=(current bounding box.center)]
\draw[lgray,line width=5pt] (0,-1) -- (0,3);
\foreach\y in {0,...,2}{
\draw[lgray,thick] (-1,\y) -- (1,\y);
}
\node[left] at (-1,0) {$j_1$};
\node[left] at (-1.3,1.2) {$\vdots$};
\node[left] at (-1,2) {$j_n$};
\node[right] at (1,0) {$\ell_1$};
\node[right] at (1.3,1.2) {$\vdots$};
\node[right] at (1,2) {$\ell_n$};
\node[below] at (0,-1) {$i$};
\node[above] at (0,3) {$k$};
\end{tikzpicture}
\end{gathered}
\right)
\equiv
w_{\{u_1,\dots,u_n\}}
\Big(i,\{j_1,\dots,j_n\} ; k,\{\ell_1,\dots,\ell_n\}\Big)
\end{align}
It is easily seen that the weight of an $n$-vertex \eqref{n-vert} is always factorized into weights of individual vertices \eqref{vertices}: knowing three of the edge states surrounding a vertex determines the fourth by lattice path conservation, and this constrains each of the internal edges in \eqref{n-vert} to assume a unique value (or it causes the whole weight to vanish if $i + j_1 + \cdots + j_n \not= k + \ell_1 + \cdots + \ell_n$, when conservation is impossible).

\begin{defn}{\rm
Let $W = {\rm Span}\{\ket{j}\}_{0 \leq j \leq 1} \cong \mathbb{C}^2$ be a two-dimensional vector space, and for all $1 \leq i \leq n$ let $W_i$ denote a copy of $W$. The {\it $n$-vertex operator} $\mathcal{W}_{\{u_1,\dots,u_n\}}(i;k)$ acts linearly on $W_1\otimes \cdots \otimes W_n$ as follows:
\begin{multline*}
\mathcal{W}_{\{u_1,\dots,u_n\}}(i;k)
:
\ket{\ell_1}_1 \otimes \cdots \otimes \ket{\ell_n}_n
\\
\mapsto
\sum_{0 \leq j_1,\dots,j_n \leq 1}
w_{\{u_1,\dots,u_n\}}
\Big(i,\{j_1,\dots,j_n\};k,\{\ell_1,\dots,\ell_n\}\Big)
\ket{j_1}_1 \otimes \cdots \otimes \ket{j_n}_n.
\end{multline*}
This in fact defines an infinite family of operators, since $i$ and $k$ can be any 
non-negative integers.
}
\end{defn}

\begin{prop}{\rm
Let $\mathcal{W}_{\{u_1,u_2\}}(i;k) \in {\rm End}(W_1 \otimes W_2)$ be a $2$-vertex operator as defined above, with $i,k \in \mathbb{Z}_{\geq0}$. The {\it Yang--Baxter equation} holds:
\begin{align}
\label{yb-eqn}
\mathcal{P} \circ
\mathcal{R}(u_2/u_1) 
\circ
\mathcal{W}_{\{u_1,u_2\}}(i;k)
=
\mathcal{W}_{\{u_2,u_1\}}(i;k) \circ \mathcal{P} \circ
\mathcal{R}(u_2/u_1),
\end{align}
where the $\mathcal{R}$-matrix is given by
\begin{align}
\label{R-mat}
\mathcal{R}(u)
=
\begin{pmatrix}
1-qu & 0 & 0 & 0
\\
0 & q(1-u) & 1-q & 0
\\
0 & (1-q)u & 1-u & 0
\\
0 & 0 & 0 & 1-qu
\end{pmatrix}
\in 
{\rm End}(W_1 \otimes W_2),
\end{align}
and $\mathcal{P} \in {\rm End}(W_1 \otimes W_2)$ is the permutation operator, with action $\mathcal{P}: \ket{a}_1 \otimes \ket{b}_2 \mapsto \ket{b}_1 \otimes \ket{a}_2$ for all vectors $\ket{a},\ket{b} \in W$.
}
\end{prop}

\begin{proof}
By direct computation, using the vertex weights \eqref{vertices} and the explicit realization of the permutation operator,
\begin{align*}
\mathcal{P}
=
\begin{pmatrix}
1 & 0 & 0 & 0
\\
0 & 0 & 1 & 0
\\
0 & 1 & 0 & 0
\\
0 & 0 & 0 & 1
\end{pmatrix}
\in 
{\rm End}(W_1 \otimes W_2).
\end{align*}
See also Section 2 of \cite{Borodin} for the derivation of \eqref{yb-eqn} from the $U_q(\widehat{\mathfrak{sl}_2})$ Yang--Baxter equation in the tensor product $W_1 \otimes W_2 \otimes V_{I}$, where $V_I$ denotes a highest weight representation with weight $I$.
 
\end{proof}

\subsection{Dual vertex weights}
\label{sec:dual_wt}

We will at times adopt an alternative convention for the vertex weights \eqref{vertices}. The change in convention is brought about inverting the spectral parameter $u$ in \eqref{vertices}, then multiplying all vertices by $(u-s)/(1-s u)$. Since the Yang--Baxter equation \eqref{yb-eqn} is preserved under this transformation (up to inversion of $u_1$ and $u_2$), this does not damage the integrability of the model, although it will be essential to ensure that certain infinite partition functions which we study are well-defined. The alternative weights are as shown below:
\begin{align}
\label{bar_vertices}
\begin{array}{cccc}
\begin{tikzpicture}[scale=0.6,>=stealth]
\draw[lred,thick] (-1,0) node[left,black] {$0$} -- (1,0) node[right,black] {$0$};
\draw[lred,line width=5pt] (0,-1) -- (0,1);
\node[below] at (0,-1) {$g$};
\draw[thick,->] (-0.075,-1) -- (-0.075,1);
\draw[thick,->] (0.075,-1) -- (0.075,1);
\node[above] at (0,1) {$g$};
\end{tikzpicture}
\quad\quad\quad
&
\begin{tikzpicture}[scale=0.6,>=stealth]
\draw[lred,thick] (-1,0) node[left,black] {$0$} -- (1,0) node[right,black] {$1$};
\draw[lred,line width=5pt] (0,-1) -- (0,1);
\node[below] at (0,-1) {$g+1$};
\draw[thick,->] (-0.15,-1) -- (-0.15,1);
\draw[thick,->] (0,-1) -- (0,1);
\draw[thick,->] (0.15,-1) -- (0.15,0) -- (1,0);
\node[above] at (0,1) {$g$};
\end{tikzpicture}
\quad\quad\quad
&
\begin{tikzpicture}[scale=0.6,>=stealth]
\draw[lred,thick] (-1,0) node[left,black] {$1$} -- (1,0) node[right,black] {$0$};
\draw[lred,line width=5pt] (0,-1) -- (0,1);
\node[below] at (0,-1) {$g$};
\draw[thick,->] (-1,0) -- (-0.15,0) -- (-0.15,1);
\draw[thick,->] (0,-1) -- (0,1);
\draw[thick,->] (0.15,-1) -- (0.15,1);
\node[above] at (0,1) {$g+1$};
\end{tikzpicture}
\quad\quad\quad
&
\begin{tikzpicture}[scale=0.6,>=stealth]
\draw[lred,thick] (-1,0) node[left,black] {$1$} -- (1,0) node[right,black] {$1$};
\draw[lred,line width=5pt] (0,-1) -- (0,1);
\node[below] at (0,-1) {$g$};
\draw[thick,->] (-1,0) -- (-0.15,0) -- (-0.15,1);
\draw[thick,->] (0,-1) -- (0,1);
\draw[thick,->] (0.15,-1) -- (0.15,0) -- (1,0);
\node[above] at (0,1) {$g$};
\end{tikzpicture}
\\
\dfrac{u-s q^g}{1-su}
\quad\quad\quad
&
\dfrac{1-s^2 q^{g}}{1-su}
\quad\quad\quad
&
\dfrac{(1-q^{g+1})u}{1-su}
\quad\quad\quad
&
\dfrac{1-sq^g u}{1-su}
\end{array}
\end{align}
Graphically, we distinguish such vertices from their counterparts \eqref{vertices} by using a coloured background. 
If we complement the states that live on horizontal edges (\ie\ draw a path if the edge is unoccupied, delete a path if the edge is occupied), we find that the vertices \eqref{bar_vertices} are converted to 
\begin{align}
\label{red_vertices}
\begin{array}{cccc}
\begin{tikzpicture}[scale=0.6,>=stealth]
\draw[lred,thick] (-1,0) node[left,red] {$1$} -- (1,0) node[right,red] {$1$};
\draw[lred,line width=5pt] (0,-1) -- (0,1);
\node[below] at (0,-1) {$g$};
\draw[thick,->,red] (-1,0) -- (-0.15,0) -- (-0.15,-1);
\draw[thick,<-,red] (0,-1) -- (0,1);
\draw[thick,->,red] (0.15,1) -- (0.15,0) -- (1,0);
\node[above] at (0,1) {$g$};
\end{tikzpicture}
\quad\quad\quad
&
\begin{tikzpicture}[scale=0.6,>=stealth]
\draw[lred,thick] (-1,0) node[left,red] {$1$} -- (1,0) node[right,red] {$0$};
\draw[lred,line width=5pt] (0,-1) -- (0,1);
\node[below] at (0,-1) {$g+1$};
\draw[thick,<-,red] (0.15,-1) -- (0.15,1);
\draw[thick,<-,red] (0,-1) -- (0,1);
\draw[thick,<-,red] (-0.15,-1) -- (-0.15,0) -- (-1,0);
\node[above] at (0,1) {$g$};
\end{tikzpicture}
\quad\quad\quad
&
\begin{tikzpicture}[scale=0.6,>=stealth]
\draw[lred,thick] (-1,0) node[left,red] {$0$} -- (1,0) node[right,red] {$1$};
\draw[lred,line width=5pt] (0,-1) -- (0,1);
\node[below] at (0,-1) {$g$};
\draw[thick,<-,red] (1,0) -- (0.15,0) -- (0.15,1);
\draw[thick,<-,red] (0,-1) -- (0,1);
\draw[thick,<-,red] (-0.15,-1) -- (-0.15,1);
\node[above] at (0,1) {$g+1$};
\end{tikzpicture}
\quad\quad\quad
&
\begin{tikzpicture}[scale=0.6,>=stealth]
\draw[lred,thick] (-1,0) node[left,red] {$0$} -- (1,0) node[right,red] {$0$};
\draw[lred,line width=5pt] (0,-1) -- (0,1);
\node[below] at (0,-1) {$g$};
\draw[thick,<-,red] (-0.075,-1) -- (-0.075,1);
\draw[thick,<-,red] (0.075,-1) -- (0.075,1);
\node[above] at (0,1) {$g$};
\end{tikzpicture}
\\
\bw_{u}(g,1;g,1)
\quad\quad\quad
&
\bw_{u}(g+1,1;g,0)
\quad\quad\quad
&
\bw_{u}(g,0;g+1,1)
\quad\quad\quad
&
\bw_{u}(g,0;g,0)
\\ \\
\dfrac{u-s q^g}{1-su}
\quad\quad\quad
&
\dfrac{1-s^2 q^{g}}{1-su}
\quad\quad\quad
&
\dfrac{(1-q^{g+1})u}{1-su}
\quad\quad\quad
&
\dfrac{1-sq^g u}{1-su}
\end{array}
\end{align}
where we have also reversed the orientation of all paths, so that they propagate NW $\rightarrow$ SE. In this way we define a set of {\it dual vertex weights} $w_u^{*}(i,j;k,\ell)$, whose non-zero values are indicated in \eqref{red_vertices}. There is a strong similarity between the vertices \eqref{red_vertices} and those of the starting vertex model \eqref{vertices}. Indeed, by reflecting the vertices \eqref{red_vertices} about their central horizontal axis, we almost recover those of \eqref{vertices} -- the only difference is that the weights of the two middle vertices have changed slightly. It turns out that this discrepancy can be cured by a simple gauge transformation of the weights. More precisely, find that
\begin{align}
\label{wt_transform}
w_u(i,j;k,\ell)
=
\frac{(q;q)_k}{(s^2;q)_k}
\bw_{u}(k,j;i,\ell)
\frac{(s^2;q)_i}{(q;q)_i},
\quad
i,k \in \mathbb{Z}_{\geq 0},
\quad
0 \leq j,\ell \leq 1.
\end{align}

\subsection{Partition states}

Let $V = {\rm Span}\left\{\ket{m} \right\}_{m \geq 0}$ be an infinite-dimensional vector space, and for all $i \geq 0$ let $V_i$ denote a copy of $V$. Further, let $\mathbb{V} = \otimes_{i \geq 0} V_i$, the global vector space obtained by tensoring the local spaces $V_i$ for all $i \geq 0$. It can be viewed as the span of pure tensors $\otimes_{i \geq 0} \ket{m_i}_i$ with $m_i \in \mathbb{Z}_{\geq 0}$, all but finitely many of which are $0$.

Consider a partition $\lambda = 0^{m_0} 1^{m_1} 2^{m_2} \dots $ expressed in terms of its of part multiplicities $m_i(\lambda)$. The part multiplicities can be obtained as the difference between adjacent parts in the conjugate partition $\lambda'$, as follows:
\begin{align*}
m_i(\lambda)
=
\lambda'_i-\lambda'_{i+1},
\end{align*}
since $\lambda'_i = \sum_{j \geq i} m_j(\lambda)$ for all $i \geq 0$. To every partition $\lambda$ we associate a unique state in $\mathbb{V}$:
\begin{align}
\label{partition}
\ket{\lambda}
=
\bigotimes_{i \geq 0}
\ket{m_i}_i
=
\bigotimes_{i \geq 0}
\ket{\lambda'_i-\lambda'_{i+1}}_i
\in 
\mathbb{V}.
\end{align}
Similarly, one can define dual partition states $\bra{\lambda} = \otimes_{i \geq 0} \bra{m_i}_i \in \mathbb{V}^{*}$, with the orthonormal action $\langle \lambda | \mu \rangle = \delta_{\lambda,\mu}$ for all partitions $\lambda,\mu$.

In the coming sections, we will study the vertex model \eqref{vertices} on a square lattice with infinitely many columns, labelled from left to right by non-negative integers. In that situation, we shall identify the vector space $V_i$ with the $i^{\rm th}$ column of the lattice, with $\ket{m}_i$ encoding a vertical edge in that column which is occupied by $m$ paths. 

In Section \ref{sec:wh_poly} we will usually require infinitely many paths on the vertical edges of the $0^{\rm th}$ column. In this setting, the natural objects are partitions $\lambda$ with strictly positive parts (\ie\ the traditional notion of a partition), since the number of zero parts in $\lambda$ becomes irrelevant. This leads us to define, for all (positive) partitions $\lambda = 1^{m_1} 2^{m_2} \dots$, the state
\begin{align}
\label{pos_partition}
\kett{\lambda}
=
\bigotimes_{i \geq 1}
\ket{m_i}_i
\in 
\wt{\mathbb{V}},
\quad
\text{where}\ \
\wt{\mathbb{V}} = \bigotimes_{i \geq 1} V_i,
\end{align}
and dual state $\bbra{\lambda} = \otimes_{i \geq 1} \bra{m_i}_i \in \wt{\mathbb{V}}^{*}$.

\section{\Hl functions}
\label{sec:spin_hl_funct}

\subsection{Setup of the lattice for $\F_{\lambda}(u_1,\dots,u_\ell)$}
\label{sec:hl-F}

Following \cite{BorodinP2}, we study the vertex model \eqref{vertices} on the quadrant $\mathbb{Z}_{\geq 0} \times \mathbb{Z}_{\geq 1}$. Horizontal lines are oriented from left to right and numbered from bottom to top, while vertical lines are oriented from bottom to top and numbered from left to right. The $i^{\rm th}$ horizontal line is assigned spectral parameter $u_i$, where $i \in \mathbb{Z}_{\geq 1}$. The boundary conditions are fixed as follows:
\begin{enumerate}[label=\bf\arabic*.]
\item There is an incoming lattice path at the external left edge of every horizontal line.
\item The external bottom edge of every vertical line is unoccupied.
\end{enumerate}
\begin{defn}
{\rm
Let $\lambda = (\lambda_1 \geq \cdots \geq \lambda_\ell \geq 0) \in {\rm Part}_{\ell}$ be a partition. The \hl function $\F_{\lambda}(u_1,\dots,u_\ell)$ is defined as the partition function of $\mathbb{Z}_{\geq 0} \times \{1,\dots,\ell\}$ in the model \eqref{vertices}, whose left and bottom edge boundary conditions are given by {\bf 1} and {\bf 2} as above, and whose $i^{\rm th}$ external top edge is occupied by exactly $m_i(\lambda)$ paths for all $i \geq 0$. See Figure \ref{fig:F}, left panel.
}
\end{defn}
\begin{remark}{\rm
The lattice used in the definition of $\F_{\lambda}(u_1,\dots,u_\ell)$ has infinitely many vertical lines, but it remains a meaningful definition. Indeed, for a given $\lambda$, only the vertical lines numbered $0$ to $\lambda_1$ will contribute non-trivially to the partition function. It is obvious that all vertical lines beyond this will be unoccupied, giving rise to infinitely many copies of the vertex $\vert{0}{0}{0}{0}{0.3}$, which has weight $1$. A similar remark applies to all infinite partition functions studied in this section.
}
\end{remark}

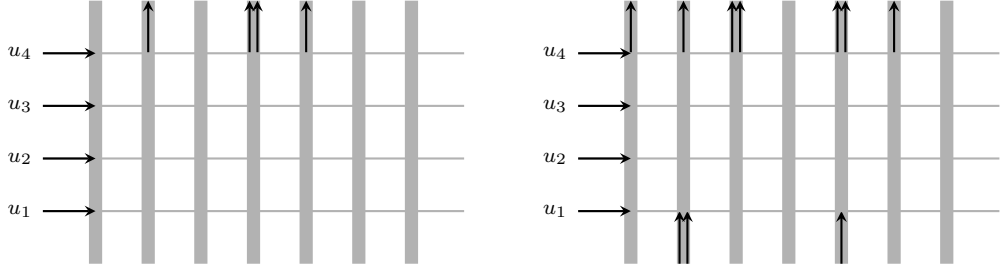
\begin{figure}
\begin{tabular}{cc}
\begin{tikzpicture}[scale=0.7,baseline=(current bounding box.center),>=stealth]
\foreach\x in {0,...,6}{
\draw[lgray,line width=5pt] (\x,0) -- (\x,5);
}
\draw[thick,->] (1,4) -- (1,5);
\draw[thick,->] (2.925,4) -- (2.925,5);
\draw[thick,->] (3.075,4) -- (3.075,5);
\draw[thick,->] (4,4) -- (4,5);
\foreach\y in {1,...,4}{
\draw[lgray,thick] (-1,\y) -- (7,\y);
\draw[thick,->] (-1,\y) -- (0,\y);
\node[left] at (-1,\y) {\tiny $u_\y$};
}
\end{tikzpicture}
\quad
&
\quad
\begin{tikzpicture}[scale=0.7,baseline=(current bounding box.center),>=stealth]
\foreach\x in {0,...,6}{
\draw[lgray,line width=5pt] (\x,0) -- (\x,5);
}
\draw[thick,->] (0,4) -- (0,5);
\draw[thick,->] (1,4) -- (1,5);
\draw[thick,->] (1.925,4) -- (1.925,5);
\draw[thick,->] (2.075,4) -- (2.075,5);
\draw[thick,->] (3.925,4) -- (3.925,5);
\draw[thick,->] (4.075,4) -- (4.075,5);
\draw[thick,->] (5,4) -- (5,5);
\draw[thick,->] (0.925,0) -- (0.925,1);
\draw[thick,->] (1.075,0) -- (1.075,1);
\draw[thick,->] (4,0) -- (4,1);
\foreach\y in {1,...,4}{
\draw[lgray,thick] (-1,\y) -- (7,\y);
\draw[thick,->] (-1,\y) -- (0,\y);
\node[left] at (-1,\y) {\tiny $u_\y$};
}
\end{tikzpicture}
\end{tabular}
\caption{Left panel: boundary conditions used to calculate $\F_{\lambda}(u_1,\dots,u_4)$ for $\lambda = (4,3,3,1)$. Right panel: boundary conditions used to calculate $\F_{\lambda/\mu}(u_1,\dots,u_4)$ for $\lambda=(5,4,4,2,2,1,0)$, $\mu=(4,1,1)$.}
\label{fig:F}
\end{figure}

By allowing more general boundary conditions at the base of the lattice, we can extend the definition to skew Young diagrams:
\begin{defn}{\rm 
Let $\lambda=(\lambda_1 \geq \cdots \geq \lambda_{\ell+n} \geq 0) \in {\rm Part}_{\ell+n}$ and $\mu=(\mu_1 \geq \cdots \geq \mu_n \geq 0) \in {\rm Part}_n$ be two partitions. The skew \hl function $\F_{\lambda/\mu}(u_1,\dots,u_\ell)$ is defined as the partition function of $\mathbb{Z}_{\geq 0} \times \{1,\dots,\ell\}$ in the model \eqref{vertices}, whose left edge boundary satisfies {\bf 1} as above, whose $i^{\rm th}$ external top edge is occupied by exactly $m_i(\lambda)$ paths, and whose $i^{\rm th}$ external bottom edge is occupied by exactly $m_i(\mu)$ paths for all $i \geq 0$. The (ordinary) \hl function $\F_{\lambda}$ is recovered in the special case $\mu = \varnothing$. See Figure \ref{fig:F}, right panel.
}
\end{defn} 

\subsection{Setup of the lattice for $\G_{\lambda}(v_1,\dots,v_n)$}
\label{sec:hl-G}

Here we use an alternative set of boundary conditions:
\begin{enumerate}[label=\bf\arabic*.]
\item The external left edge of every horizontal line is unoccupied.
\item The external bottom edge of each vertical line is unoccupied, with the exception of the $0^{\rm th}$, which is occupied by $\ell$ incoming paths, for some $\ell \geq 0$.
\end{enumerate}
\begin{defn}
\label{def:G}
{\rm
Let $\lambda = (\lambda_1 \geq \cdots \geq \lambda_\ell \geq 0) \in {\rm Part}_{\ell}$ be a partition. The dual \hl function $\G_{\lambda}(v_1,\dots,v_n)$ is defined as the partition function of $\mathbb{Z}_{\geq 0} \times \{1,\dots,n\}$ in the model \eqref{vertices}, whose left and bottom edge boundary conditions are given by {\bf 1} and {\bf 2} as above, and whose $i^{\rm th}$ external top edge is occupied by exactly $m_i(\lambda)$ paths for all $i \geq 0$. Notice that we impose no relation between the values of $\ell$ and $n$. See Figure \ref{fig:G}, left panel.
}
\end{defn}

\begin{figure}
\begin{tabular}{cc}
\begin{tikzpicture}[scale=0.7,baseline=(current bounding box.center),>=stealth]
\foreach\x in {0,...,6}{
\draw[lgray,line width=5pt] (\x,0) -- (\x,6);
}
\draw[thick,->] (0,5) -- (0,6);
\draw[thick,->] (1,5) -- (1,6);
\draw[thick,->] (4,5) -- (4,6);
\draw[thick,->] (5,5) -- (5,6);
\draw[thick,->] (0.775-1,0) -- (0.775-1,1);
\draw[thick,->] (0.925-1,0) -- (0.925-1,1);
\draw[thick,->] (1.075-1,0) -- (1.075-1,1);
\draw[thick,->] (1.225-1,0) -- (1.225-1,1);
\foreach\y in {1,...,5}{
\draw[lgray,thick] (-1,\y) -- (7,\y);
\node[left] at (-1,\y) {\tiny $v_\y$};
}
\end{tikzpicture}
\quad
&
\quad
\begin{tikzpicture}[scale=0.7,baseline=(current bounding box.center),>=stealth]
\foreach\x in {0,...,6}{
\draw[lgray,line width=5pt] (\x,0) -- (\x,6);
}
\draw[thick,->] (1,5) -- (1,6);
\draw[thick,->] (2.925,5) -- (2.925,6);
\draw[thick,->] (3.075,5) -- (3.075,6);
\draw[thick,->] (4,5) -- (4,6);
\draw[thick,->] (0.925-1,0) -- (0.925-1,1);
\draw[thick,->] (0.075,0) -- (0.075,1);
\draw[thick,->] (2,0) -- (2,1);
\draw[thick,->] (3,0) -- (3,1);
\foreach\y in {1,...,5}{
\draw[lgray,thick] (-1,\y) -- (7,\y);
\node[left] at (-1,\y) {\tiny $v_\y$};
}
\end{tikzpicture}
\end{tabular}
\caption{Left panel: boundary conditions used to calculate $\G_{\lambda}(v_1,\dots,v_5)$ for $\lambda = (5,4,1,0)$. Right panel: boundary conditions used to calculate $\G_{\lambda/\mu}(v_1,\dots,v_5)$ for $\lambda = (4,3,3,1)$, $\mu=(3,2,0,0)$.}
\label{fig:G}
\end{figure}
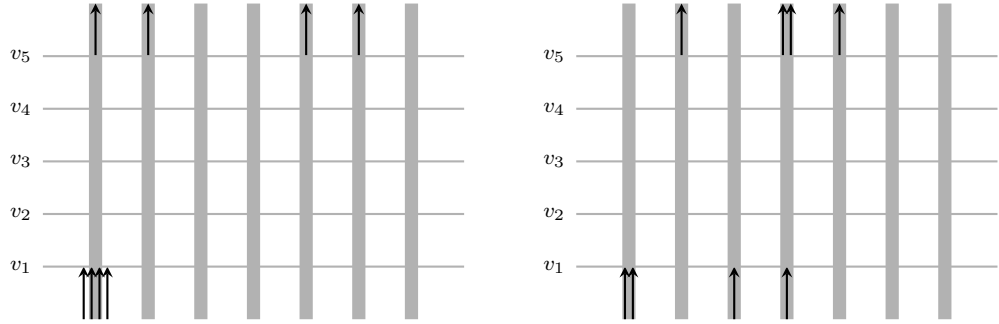
Once again, this definition can be extended to skew Young diagrams by allowing for more general boundary conditions at the base of the lattice:
\begin{defn}
\label{def:skewG}
{\rm
Let $\lambda = (\lambda_1 \geq \cdots \geq \lambda_\ell \geq 0) \in {\rm Part}_{\ell}$ and $\mu = (\mu_1 \geq \cdots \geq \mu_\ell \geq 0) \in {\rm Part}_{\ell}$ be two partitions of equal length. The dual skew \hl function $\G_{\lambda/\mu}(v_1,\dots,v_n)$ is defined as the partition function of $\mathbb{Z}_{\geq 0} \times \{1,\dots,n\}$ in the model \eqref{vertices}, whose left edge boundary conditions are given by {\bf 1} as above, whose $i^{\rm th}$ external top edge is occupied by exactly $m_i(\lambda)$ paths, and whose $i^{\rm th}$ external bottom edge is occupied by exactly $m_i(\mu)$ paths for all $i \geq 0$. The (ordinary) dual \hl function $\G_{\lambda}$ is recovered in the special case $\mu = 0^{\ell}$. See Figure \ref{fig:G}, right panel.
}
\end{defn}

\subsection{The function $\Gc_{\lambda/\mu}(v_1,\dots,v_n)$}
\label{sec:hl-Gc}

There is a slight modification of the dual polynomials in Section \ref{sec:hl-G} that turns out to be important for correctly stating the Cauchy identity between $\F_{\lambda}$ and $\G_{\lambda}$ \cite{Borodin,BorodinP2}. Let $\lambda$ and $\mu$ be two partitions in the set ${\rm Part}_{\ell}$. We define
\begin{align}
\label{normaliz}
\Gc_{\lambda/\mu}(v_1,\dots,v_n)
:=
\frac{\c_{\lambda}(q,s)}{\c_{\mu}(q,s)}
\G_{\lambda/\mu}(v_1,\dots,v_n),
\qquad
\c_{\lambda}(q,s)
:=
\frac{(q;q)_{\ell}}{(s^2;q)_{\ell}}
\prod_{i \geq 0} \frac{(s^2;q)_{m_i(\lambda)}}{(q;q)_{m_i(\lambda)}}.
\end{align}
In the special case $\mu = 0^{\ell}$, one has $m_0(\mu) = \ell$ and $m_i(\mu) = 0$ for all 
$i \geq 1$, so that $\c_{\mu}(q,s) = 1$. Then \eqref{normaliz} reduces to 
\begin{align*}
\Gc_{\lambda}(v_1,\dots,v_n)
=
\c_{\lambda}(q,s)
\G_{\lambda}(v_1,\dots,v_n).
\end{align*}

\begin{prop}
\label{prop:Gc}
{\rm
The function $\Gc_{\lambda/\mu}(v_1,\dots,v_n)$ is equal to the partition function of $\mathbb{Z}_{\geq 0} \times \{1,\dots,n\}$ in the model \eqref{red_vertices}, whose external left edges are all unoccupied, whose $i^{\rm th}$ external top edge is occupied by exactly $m_i(\mu)$ paths, and whose $i^{\rm th}$ external bottom edge is occupied by exactly $m_i(\lambda)$ paths for all $i \geq 0$ (see Figure \ref{fig:redG}).
}
\end{prop}

\begin{figure}
\begin{tikzpicture}[scale=0.7,baseline=(current bounding box.center),>=stealth]
\foreach\x in {0,...,6}{
\draw[lred,line width=5pt] (\x,2) -- (\x,7);
}
\foreach\y in {4,...,1}{
\draw[lred,thick] (-1,7-\y) -- (7,7-\y);
\node[left] at (-1,7-\y) {\tiny$v_\y$};
}
\draw[thick,<-,red] (-0.15,6) -- (-0.15,7); \draw[thick,<-,red] (0,6) -- (0,7); \draw[thick,<-,red] (0.15,6) -- (0.15,7);
\draw[thick,<-,red] (2,6) -- (2,7);
\draw[thick,<-,red] (3,6) -- (3,7);
\draw[thick,<-,red] (0,2) -- (0,3);
\draw[thick,<-,red] (0.925,2) -- (0.925,3); \draw[thick,<-,red] (1.075,2) -- (1.075,3);
\draw[thick,<-,red] (3,2) -- (3,3);
\draw[thick,<-,red] (4,2) -- (4,3);
\end{tikzpicture}
\caption{Lattice used to calculate $\Gc_{\lambda/\mu}(v_1,\dots,v_4)$ for 
$\lambda=(4,3,1,1,0)$, $\mu=(3,2,0,0,0)$.}
\label{fig:redG}
\end{figure}
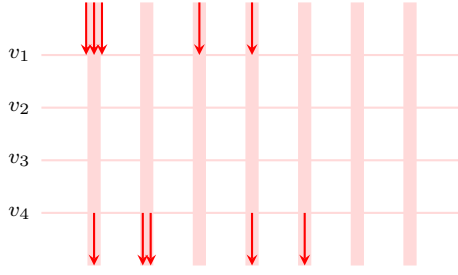

\begin{proof}
This is a straightforward consequence of the relation \eqref{wt_transform}. Starting from the partition function as described in the Proposition, we apply the conjugation \eqref{wt_transform} to every vertex in the lattice. This conjugation can be effected by multiplying the whole partition function by 
\begin{align*}
\prod_{i \geq 0} 
\left(
\frac{(s^2;q)_{m_i(\mu)}}{(q;q)_{m_i(\mu)}}
\right)
\left(
\frac{(q;q)_{m_i(\lambda)}}{(s^2;q)_{m_i(\lambda)}}
\right),
\end{align*}
and at the end of this transformation we obtain the partition function as described in Definition \ref{def:skewG}. We conclude that the partition function that we started with is equal to
\begin{align*}
\prod_{i \geq 0} 
\left(
\frac{(q;q)_{m_i(\mu)}}{(s^2;q)_{m_i(\mu)}}
\right)
\left(
\frac{(s^2;q)_{m_i(\lambda)}}{(q;q)_{m_i(\lambda)}}
\right)
\G_{\lambda/\mu}(v_1,\dots,v_n)
=
\frac{\c_{\lambda}(q,s)}{\c_{\mu}(q,s)}
\G_{\lambda/\mu}(v_1,\dots,v_n),
\end{align*}
completing the proof.
\end{proof}

\subsection{Row operators}
\label{sec:row_hl}

So far we defined the \hl functions as partition functions in the vertex model \eqref{vertices}, but it is also possible to formulate them algebraically. For that, we now introduce finite {\it row operators,} whose action is specified in terms of the partition function of a single row of vertices:
\begin{align*}
w_u
\left(
\begin{gathered}
\begin{tikzpicture}[scale=0.5,baseline=(current bounding box.center)]
\draw[lgray,thick] (-1,0) -- (7,0);
\node[left] at (-0.8,0) {\tiny $j$};\node[right] at (6.8,0) {\tiny $\ell$};
\foreach\x in {0,...,6}{
\draw[lgray,line width=5pt] (\x,-1) -- (\x,1);
}
\node[below] at (0,-0.8) {\tiny $i_0$};\node[above] at (0,0.8) {\tiny $k_0$};
\node[below] at (3,-0.8) {$\cdots$};\node[above] at (3,0.8) {$\cdots$};
\node[below] at (6,-0.8) {\tiny $i_L$};\node[above] at (6,0.8) {\tiny $k_L$};
\end{tikzpicture}
\end{gathered}
\right)
\equiv
w_u\Big(\{i_0,\dots,i_L\}, j ; \{k_0,\dots,k_L\}, \ell \Big),
\end{align*}
where $L$ is any non-negative integer. Introduce the monodromy matrix
\begin{align*}
T(u)
=
\begin{pmatrix}
T_u(0;0) & T_u(0;1)
\\
T_u(1;0) & T_u(1;1)
\end{pmatrix}
\equiv
\begin{pmatrix}
A_L(u) & B_L(u)
\\
C_L(u) & D_L(u)
\end{pmatrix}
\end{align*}
whose entries $T_u(j;\ell)$ are operators acting linearly on $V_0 \otimes \cdots \otimes V_L$ as follows\footnote{We point out that our conventions regarding the operators \eqref{finite_action} are different to those of \cite{BorodinP2}. When viewed as rows in partition functions, the operators \eqref{finite_action} act from top to bottom, whereas those in \cite{BorodinP2} act from bottom to top. Also, the labelling of the operators in \cite{BorodinP2} is obtained from our labelling under $B \leftrightarrow C$.}:
\begin{align}
\label{finite_action}
T_u(j;\ell) : \ket{k_0}_0 \otimes \cdots \otimes \ket{k_L}_L
\mapsto
\sum_{i_0,\dots,i_L \geq 0}
w_u\Big(\{i_0,\dots,i_L\}, j ; \{k_0,\dots,k_L\}, \ell \Big)
\ket{i_0}_0 \otimes \cdots \otimes \ket{i_L}_L.
\end{align}
It is immediate, from $L+1$ applications of the Yang--Baxter equation \eqref{yb-eqn}, that the monodromy matrix satisfies the intertwining relation
\begin{align*}
\mathcal{P}
\circ
\mathcal{R}(u/v)
\circ
\Big(
T(v) \otimes T(u)
\Big)
=
\Big(
T(u) \otimes T(v)
\Big)
\circ
\mathcal{P}
\circ
\mathcal{R}(u/v),
\end{align*}
as an identity in ${\rm End}(W_1 \otimes W_2 \otimes V_1 \otimes \cdots \otimes V_L)$. This encodes sixteen bilinear relations among the entries of the monodromy matrix, which are collectively known as the Yang--Baxter algebra. See, for example, Chapter VII of \cite{KorepinBI} for a complete list of the relations. We will primarily be interested in the following ones:
\begin{align}
\label{commute2}
[A_L(u),A_L(v)] = [B_L(u),B_L(v)] = [C_L(u),C_L(v)] = [D_L(u),D_L(v)] = 0,
\\
\label{exchange}
(1-q u/v) D_L(v) C_L(u) = (1-u/v) C_L(u) D_L(v) + (1-q) D_L(u) C_L(v).
\end{align}
It is convenient to define a ``starred'' version of the $B$- and $D$-operators:
\begin{align}
\label{bar_norm}
\b{B}_L(u) := \left( \frac{u-s}{1-s u} \right)^{L+1} B_L(u^{-1}),
\quad\quad
\b{D}_L(u) := \left( \frac{u-s}{1-s u} \right)^{L+1} D_L(u^{-1}).
\end{align}
This has the effect of modifying these operators, so that they are constructed in precisely the same way as above, but now using the alternative vertices \eqref{bar_vertices}.

\subsection{Infinite volume limit}

It is possible to take the length of the row operators to infinity, by sending $L \rightarrow \infty$ in \eqref{finite_action}. However some care is needed in doing so, since the behaviour of $A_L(u),B_L(u),C_L(u),D_L(u)$ in the limit depends on how they are normalized. One finds that
\begin{align*}
A(u)
:= \lim_{L \rightarrow \infty} A_L(u),
\quad\quad
C(u)
:= \lim_{L \rightarrow \infty} C_L(u)
\end{align*} 
make sense as written; for the remaining two operators it is necessary to use the normalization \eqref{bar_norm}, leading to
\begin{align*}
\b{B}(u)
:= \lim_{L \rightarrow \infty} \b{B}_L(u),
\quad\quad
\b{D}(u)
:= \lim_{L \rightarrow \infty} \b{D}_L(u).
\end{align*}
Now let us examine what happens to the commutation relation \eqref{exchange} in the limit $L \rightarrow \infty$. We firstly invert $v$, then multiply the whole relation by $(v-s)^{L+1}/(1-sv)^{L+1}$: 
\begin{align*}
(1-q u v) \b{D}_L(v) C_L(u) 
=
(1-u v) C_L(u) \b{D}_L(v)
+
(1-q) 
\left(
\frac{(u-s)(v-s)}{(1-su)(1-sv)}
\right)^{L+1}
\b{D}_L(u) C_L(v).
\end{align*}
Assuming that $u,v,s$ are chosen such that $|(u-s)(v-s)| < |(1-su)(1-sv)|$, the second term of the above equation vanishes when $L \rightarrow \infty$, and we obtain
\begin{align}
\label{HL_exchange}
(1-q u v) \b{D}(v) C(u) 
=
(1-u v) C(u) \b{D}(v).
\end{align}

\subsection{Algebraic formulation of \hl functions}

Using the definition of the row operators from the previous section, the following result is immediate:
\begin{prop}
\label{prop:expect}
{\rm
Fix two partitions $\lambda \in {\rm Part}_{\ell+n}$ and $\mu \in {\rm Part}_{n}$. Then
\begin{align*}
\F_{\lambda/\mu}(u_1,\dots,u_\ell)
=
\bra{\mu} C(u_1) \dots C(u_\ell) \ket{\lambda},
\end{align*}
where we have used the notation \eqref{partition} for partition states, namely, $\ket{\lambda} = \otimes_{i \geq 0} \ket{m_i(\lambda)}_i \in \mathbb{V}$ and $\bra{\mu} = \otimes_{i \geq 0} \bra{m_i(\mu)}_i \in \mathbb{V}^{*}$. Similarly, fix two partitions $\lambda,\mu \in {\rm Part}_{\ell}$. Then
\begin{align*}
\G_{\lambda/\mu}(v_1,\dots,v_n)
&=
\bra{\mu} A(v_1) \dots A(v_n) \ket{\lambda},
\\
\Gc_{\lambda/\mu}(v_1,\dots,v_n)
&=
\bra{\lambda} \b{D}(v_n) \dots \b{D}(v_1) \ket{\mu}.
\end{align*}
}
\end{prop}

\begin{proof}
These expectation values can be directly compared with the partition functions written down in Sections \ref{sec:hl-F}--\ref{sec:hl-Gc}. Note that each one is clearly symmetric in its rapidity variables, by virtue of the commutation relations \eqref{commute2}.
\end{proof}

\subsection{Proving the skew Cauchy identity}

Equipped with the algebraic expressions of Proposition \ref{prop:expect}, it is possible to establish many interesting properties of the symmetric functions $\F_{\lambda/\mu}$ and $\G_{\lambda/\mu}$. Here we shall recall the proof of one such property, the {\it skew Cauchy identity.} We refer the reader to \cite{Borodin,BorodinP2} for a more detailed exposition of other properties.

\begin{thm}
\label{thm:hl_cauchy}
{\rm 
Fix two partitions $\mu \in {\rm Part}_n$ and $\nu \in {\rm Part}_{\ell+n}$, for some $\ell,n \geq 1$. For any complex numbers $u_1,\dots,u_{\ell};v_1,\dots,v_m$ such that $|(u_i-s)(v_j-s)| < |(1-s u_i)(1-s v_j)|$ for all $i,j$, the following summation identity holds:
\begin{align}
\label{HL_skew_C}
\sum_{\lambda} \F_{\lambda/\mu}(u_1,\dots,u_{\ell}) \Gc_{\lambda/\nu}(v_1,\dots,v_m)
=
\left(
\prod_{i=1}^{\ell}
\prod_{j=1}^{m}
\frac{1-q u_i v_j}{1-u_i v_j}
\right)
\sum_{\kappa} \F_{\nu/\kappa}(u_1,\dots,u_{\ell}) \Gc_{\mu/\kappa}(v_1,\dots,v_m),
\end{align}
where the left hand side is summed over partitions $\lambda \in {\rm Part}_{\ell+n}$, and the right hand side over partitions $\kappa \in {\rm Part}_{n}$ (observe that the right hand sum is finite, since the summand vanishes if $\kappa \not\subset \nu$ or if $\kappa \not\subset \mu$).
}
\end{thm}

\begin{proof}
Consider the following expectation value:
\begin{align*}
\mathcal{E}_{\mu,\nu}(u_1,\dots,u_{\ell};v_1,\dots,v_m)
:=
\bra{\mu} C(u_1) \dots C(u_{\ell}) \b{D}(v_m) \dots \b{D}(v_1) \ket{\nu}.
\end{align*}
Inserting the identity operator $\mathbbm{1} = \sum_{\lambda} \ket{\lambda} \bra{\lambda}$ between the final $C$-operator and the first $\b{D}$-operator, it is immediate from Proposition \ref{prop:expect} that
\begin{align}
\nonumber
\mathcal{E}_{\mu,\nu}(u_1,\dots,u_{\ell};v_1,\dots,v_m)
&=
\sum_{\lambda}
\bra{\mu} C(u_1) \dots C(u_{\ell}) \ket{\lambda} \bra{\lambda} \b{D}(v_m) \dots \b{D}(v_1) \ket{\nu}
\\
\label{E1}
&=
\sum_{\lambda} \F_{\lambda/\mu}(u_1,\dots,u_{\ell}) \Gc_{\lambda/\nu}(v_1,\dots,v_m).
\end{align}
On the other hand, by repeated use of the exchange relation \eqref{HL_exchange}, it is clear that one also has
\begin{align}
\nonumber
\mathcal{E}_{\mu,\nu}(u_1,\dots,u_{\ell};v_1,\dots,v_m)
&=
\left(
\prod_{i=1}^{\ell}
\prod_{j=1}^{m}
\frac{1-q u_i v_j}{1-u_i v_j}
\right)
\bra{\mu} \b{D}(v_m) \dots \b{D}(v_1) C(u_1) \dots C(u_{\ell}) \ket{\nu}
\\
\nonumber
&=
\left(
\prod_{i=1}^{\ell}
\prod_{j=1}^{m}
\frac{1-q u_i v_j}{1-u_i v_j}
\right)
\sum_{\kappa}
\bra{\mu} \b{D}(v_m) \dots \b{D}(v_1) \ket{\kappa} \bra{\kappa} C(u_1) \dots C(u_{\ell}) \ket{\nu}
\\
\label{E2}
&=
\left(
\prod_{i=1}^{\ell}
\prod_{j=1}^{m}
\frac{1-q u_i v_j}{1-u_i v_j}
\right)
\sum_{\kappa} \Gc_{\mu/\kappa}(v_1,\dots,v_m) \F_{\nu/\kappa}(u_1,\dots,u_{\ell}),
\end{align}
where we have again inserted a complete set of states $\mathbbm{1} = \sum_{\kappa} \ket{\kappa} \bra{\kappa}$ to establish the final equality. Comparing \eqref{E1} and \eqref{E2}, we recover the skew Cauchy identity \eqref{HL_skew_C}.

\end{proof}

\begin{cor}{\rm 
Assuming $|(u_i-s)(v_j-s)| < |(1-s u_i)(1-s v_j)|$ for all $i,j$, the following identity holds: 
\begin{align}
\label{hl-cauchy}
\sum_{\lambda}
\F_{\lambda}(u_1,\dots,u_{\ell})
\Gc_{\lambda}(v_1,\dots,v_{m})
=
\frac{(q;q)_{\ell}}{\prod_{i=1}^{\ell}(1-su_i)}
\prod_{i=1}^{\ell}
\prod_{j=1}^{m}
\left(
\frac{1-qu_i v_j}{1-u_i v_j}
\right),
\end{align}
where the sum is taken over all partitions $\lambda \in {\rm Part}_{\ell}$.
}
\end{cor}

\begin{proof}
We make the choice $\mu = \varnothing$ and $\nu = 0^{\ell}$ in \eqref{HL_skew_C}. This converts the left hand side into that of \eqref{hl-cauchy}, while the sum on the right hand side now consists of a single term, corresponding to 
$\kappa = \varnothing$. One can easily show that\footnote{It is clear from its definition as a partition function that $\F_{0^{\ell}}(u_1,\dots,u_{\ell})$ is equal to the $\ell$-vertex $w_{\{u_1,\dots,u_{\ell}\}}(0,\{1,\dots,1\};\ell,\{0,\dots,0\})$, whose weight is precisely $(q;q)_{\ell} \prod_{i=1}^{\ell}(1-su_i)^{-1}$.} 
\begin{align*}
\F_{0^{\ell}}(u_1,\dots,u_{\ell}) 
= 
\frac{(q;q)_{\ell}}{\prod_{i=1}^{\ell} (1-su_i)},
\quad\quad
\Gc_{\varnothing}(v_1,\dots,v_m) = 1.
\end{align*}
This accounts for the new factors appearing on the right hand side of \eqref{hl-cauchy}, and completes the proof.
\end{proof}

\section{Fusion}
\label{sec:fusion}

\subsection{Specializing variables to geometric progressions}
\label{sec:geom}

Specializing spectral parameters to geometric progressions plays a crucial role in the fusion procedure. Since the \wh polynomials are ultimately to be obtained under such specializations, we develop some rather general notation here which will prove useful in our calculations. 

Let $\{J_0,\dots,J_n\}$ be some arbitrary set of positive integers and further, define their partial sums $\mathcal{J}_{i}:= \sum_{k=0}^{i} J_k$. We will often write $\mathcal{J}_n \equiv \mathcal{J}$ for the sum of all $J_i$. Then for any set of variables $\{u_1,\dots,u_{\mathcal{J}}\}$, we define its $\{J_0,\dots,J_n;z_0,\dots,z_n\}$-specialization to be
\begin{align}
\label{geom_spec}
\Big\{
u_{(\mathcal{J}_{i-1}+1)}, 
u_{(\mathcal{J}_{i-1}+2)}, 
\dots,
u_{(\mathcal{J}_{i-1}+J_i)}
\Big\}
=
\left\{z_i,q z_i,\dots,q^{J_i-1}z_i\right\}
\end{align}
for all $0 \leq i \leq n$, where by agreement $\mathcal{J}_{-1}=0$. We use the notation $\rho_{\{J_0,\dots,J_n;z_0,\dots,z_n\}}(\cdot)$ to indicate that a $\{J_0,\dots,J_n;z_0,\dots,z_n\}$-specialization has been taken.

For the purposes of fusion itself, we will only require the $n=0$ case of the above. In that case we have a single positive integer $J_0 \equiv J$, and a $\{J;u\}$-specialization of $\{u_1,\dots,u_J\}$ is simply given by
\begin{align}
\label{geom_spec0}
\left\{
u_1, 
u_2, 
\dots,
u_J
\right\}
=
\left\{u,qu,\dots,q^{J-1}u\right\}.
\end{align}

\subsection{Fusion}

We recall the basics of the fusion procedure in integrable lattice models \cite{KulishRS,KirillovR}, mainly following the conventions of Section 5 of \cite{BorodinP2} (see also \cite{CorwinP}, for a more probabilistic interpretation of fusion). The central object is the $J$-vertex, as defined in \eqref{n-vert}, whose spectral parameters $\{u_1,u_2,\dots,u_J\}$ have been $\{J;u\}$-specialized; \ie\ specialized as in \eqref{geom_spec0}. One then defines a ``fused'' vertex as follows:
\begin{align}
\label{fused_defn}
w^{(J)}_u
\left(
\begin{gathered}
\begin{tikzpicture}[scale=0.5,baseline=(current bounding box.center)]
\draw[lgray,line width=5pt] (-1,0) -- (1,0);
\draw[lgray,line width=5pt] (0,-1) -- (0,1);
\node[left] at (-0.8,0) {\tiny $j$};\node[right] at (0.8,0) {\tiny $\ell$};
\node[below] at (0,-0.8) {\tiny $i$};\node[above] at (0,0.8) {\tiny $k$};
\end{tikzpicture}
\end{gathered}
\right)
=
\sum_{\substack{0\leq a_1,\dots,a_J \leq 1: |a| = j \\ 0 \leq b_1,\dots,b_J \leq 1: |b|=\ell}}
\frac{q^{\sum_{m=1}^{J} (m-1) a_m}}{Z_j(J)}
\times
w_{\{u,qu,\dots,q^{J-1}u\}}
\left(
\begin{gathered}
\begin{tikzpicture}[scale=0.6,baseline=(current bounding box.center)]
\draw[lgray,line width=5pt] (0,-1) -- (0,3);
\foreach\y in {0,...,2}{
\draw[lgray,thick] (-1,\y) -- (1,\y);
}
\node[left] at (-1,0) {$a_1$};
\node[left] at (-1.3,1.3) {$\vdots$};
\node[left] at (-1,2) {$a_J$};
\node[right] at (1,0) {$b_1$};
\node[right] at (1.3,1.3) {$\vdots$};
\node[right] at (1,2) {$b_J$};
\node[below] at (0,-1) {$i$};
\node[above] at (0,3) {$k$};
\end{tikzpicture}
\end{gathered}
\right),
\end{align}
obtained by fixing the bottom and top of the $J$-vertex to the states $i$ and $k$, respectively, while summing the left and right edges over all possible ways of assigning $j$ and $\ell$ arrows to $J$ sites. Here we have defined the normalization
\begin{align*}
Z_j(J)
=
\sum_{0 \leq c_1,\dots,c_J \leq 1: |c| = j}
q^{\sum_{m=1}^{J} (m-1) c_m}
=
q^{j(j-1)/2}
\frac{(q;q)_J}{(q;q)_j (q;q)_{J-j}}.
\end{align*}
Graphically, we represent fused vertices by the intersection of a thick horizontal and vertical line. This is supposed to indicate that up to $J$ lattice paths can now occupy a horizontal edge, while the occupation numbers along vertical edges continue to be unbounded. As before we interchange freely between the graphical version of vertices, and a purely algebraic notation:
\begin{align}
\label{fused_notation}
w^{(J)}_u
\left(
\begin{gathered}
\begin{tikzpicture}[scale=0.5,baseline=(current bounding box.center)]
\draw[lgray,line width=5pt] (-1,0) -- (1,0);
\draw[lgray,line width=5pt] (0,-1) -- (0,1);
\node[left] at (-0.8,0) {\tiny $j$};\node[right] at (0.8,0) {\tiny $\ell$};
\node[below] at (0,-0.8) {\tiny $i$};\node[above] at (0,0.8) {\tiny $k$};
\end{tikzpicture}
\end{gathered}
\right)
\equiv
w^{(J)}_u(i,j;k,\ell),
\quad
i,k \in \mathbb{Z}_{\geq 0},
\quad
0 \leq j,\ell \leq J.
\end{align}
One of the key features of the fused vertices is that their horizontal concatenation corresponds with horizontal concatenation in the original vertex model:
\begin{thm}
\label{thm:horizontal_join}
{\rm
For any $L \geq 1$, $0 \leq j,\ell \leq J$ and $\{i_0,\dots,i_L\}, \{k_0,\dots,k_L\} \in \mathbb{Z}_{\geq 0}$, one has
\begin{align}
\label{join}
w^{(J)}_u
\left(
\begin{tikzpicture}[scale=0.5,baseline=(current bounding box.center)]
\draw[lgray,line width=5pt] (-1,0) -- (3,0);
\draw[lgray,line width=5pt] (0,-1) -- (0,1);
\draw[lgray,line width=5pt] (1,-1) -- (1,1);
\draw[lgray,line width=5pt] (2,-1) -- (2,1);
\node[left] at (-0.8,0) {\tiny $j$};\node[right] at (2.8,0) {\tiny $\ell$};
\node[below] at (0,-0.8) {\tiny $i_0$};\node[above] at (0,0.8) {\tiny $k_0$};
\node[below] at (1,-0.8) {\tiny $\cdots$};\node[above] at (1,0.8) {\tiny $\cdots$};
\node[below] at (2,-0.8) {\tiny $i_L$};\node[above] at (2,0.8) {\tiny $k_L$};
\end{tikzpicture}
\right)
=
\sum_{\substack{a_1,\dots,a_J: |a| = j \\ b_1,\dots,b_J: |b|=\ell}}
\frac{q^{\sum_{m=1}^{J} (m-1) a_m}}{Z_j(J)}
\times
w_{\{u,qu,\dots,q^{J-1}u\}}
\left(
\begin{tikzpicture}[scale=0.6,baseline=(current bounding box.center)]
\draw[lgray,line width=5pt] (0,-1) -- (0,3);
\draw[lgray,line width=5pt] (1,-1) -- (1,3);
\draw[lgray,line width=5pt] (2,-1) -- (2,3);
\foreach\y in {0,...,2}{
\draw[lgray,thick] (-1,\y) -- (3,\y);
}
\node[left] at (-1,0) {$a_1$};
\node[left] at (-1.3,1.3) {$\vdots$};
\node[left] at (-1,2) {$a_J$};
\node[right] at (3,0) {$b_1$};
\node[right] at (3.3,1.3) {$\vdots$};
\node[right] at (3,2) {$b_J$};
\node[below] at (0,-1) {$i_0$};
\node[above] at (0,3) {$k_0$};
\node[below] at (1,-1) {$\cdots$};
\node[above] at (1,3) {$\cdots$};
\node[below] at (2,-1) {$i_L$};
\node[above] at (2,3) {$k_L$};
\end{tikzpicture}
\right).
\end{align}
}
\end{thm}

\begin{proof}
We restrict our attention to the case $L = 1$, since the result for larger values of $L$ is proved in exactly the same way.
We take the right hand side of \eqref{join} and explicitly perform the summation over the internal horizontal edges. These internal edges are summed over all ways of distributing $n$ arrows, where (by conservation) 
$n= i_0+j-k_0 = k_1+\ell-i_1$. The right hand side of \eqref{join} is thus
\begin{align}
\label{rhs_concat}
\sum_{\substack{a_1,\dots,a_J: |a| = j 
\\ b_1,\dots,b_J: |b|=\ell \\ c_1,\dots,c_J: |c|=n}}
\frac{q^{\sum_{m=1}^{J} (m-1) a_m}}{Z_j(J)}
\times
w_{\{u,qu,\dots,q^{J-1}u\}}
\left(
\begin{tikzpicture}[scale=0.6,baseline=(current bounding box.center)]
\draw[lgray,line width=5pt] (0,-1) -- (0,3);
\foreach\y in {0,...,2}{
\draw[lgray,thick] (-1,\y) -- (1,\y);
}
\node[left] at (-1,0) {$a_1$};
\node[left] at (-1.3,1.3) {$\vdots$};
\node[left] at (-1,2) {$a_J$};
\node[right] at (1,0) {$c_1$};
\node[right] at (1.3,1.3) {$\vdots$};
\node[right] at (1,2) {$c_J$};
\node[below] at (0,-1) {$i_0$};
\node[above] at (0,3) {$k_0$};
\end{tikzpicture}
\right)
w_{\{u,qu,\dots,q^{J-1}u\}}
\left(
\begin{tikzpicture}[scale=0.6,baseline=(current bounding box.center)]
\draw[lgray,line width=5pt] (0,-1) -- (0,3);
\foreach\y in {0,...,2}{
\draw[lgray,thick] (-1,\y) -- (1,\y);
}
\node[left] at (-1,0) {$c_1$};
\node[left] at (-1.3,1.3) {$\vdots$};
\node[left] at (-1,2) {$c_J$};
\node[right] at (1,0) {$b_1$};
\node[right] at (1.3,1.3) {$\vdots$};
\node[right] at (1,2) {$b_J$};
\node[below] at (0,-1) {$i_1$};
\node[above] at (0,3) {$k_1$};
\end{tikzpicture}
\right),
\end{align}
or using algebraic notation,
\begin{multline}
\label{rhs_concat2}
\sum_{\substack{a_1,\dots,a_J: |a| = j 
\\ b_1,\dots,b_J: |b|=\ell \\ c_1,\dots,c_J: |c|=n}}
\frac{q^{\sum_{m=1}^{J} (m-1) a_m}}{Z_j(J)}
w_{\{u,qu,\dots,q^{J-1}u\}}
\Big(
i_0,\{a_1,\dots,a_J\};k_0,\{c_1,\dots,c_J\}
\Big)
\\
\times
w_{\{u,qu,\dots,q^{J-1}u\}}
\Big(
i_1,\{c_1,\dots,c_J\};k_1,\{b_1,\dots,b_J\}
\Big).
\end{multline}
In order to progress further, we note the following property on $2$-vertices, which can be easily checked for all 
$i,k \in \mathbb{Z}_{\geq 0}$ \cite{BorodinP2}:
\begin{align}
\label{2-vert-rel}
\sum_{0 \leq a_1,a_2 \leq 1}
q^{a_2}
w_{\{u,qu\}} \Big( i,\{a_1,a_2\}; k, \{0,1\} \Big)
=
\sum_{0 \leq a_1,a_2 \leq 1}
q^{a_2+1}
w_{\{u,qu\}} \Big( i,\{a_1,a_2\}; k, \{1,0\} \Big).
\end{align}
Because the spectral parameters in \eqref{rhs_concat} form geometric sequences with ratio $q$, we are able to apply the relation \eqref{2-vert-rel} repeatedly to the first $J$-vertex in \eqref{rhs_concat}, treating all indices apart from $a_1,\dots,a_J$ as fixed. We find that
\begin{multline}
\sum_{a_1,\dots,a_J: |a| = j}
\frac{q^{\sum_{m=1}^{J} (m-1) a_m}}{Z_j(J)}
w_{\{u,qu,\dots,q^{J-1}u\}}
\Big(
i_0,\{a_1,\dots,a_J\};k_0,\{c_1,\dots,c_J\}
\Big)
\\
=
q^{\sum_{m=1}^{J} (m-1) c_m}
\sum_{a_1,\dots,a_J: |a| = j}
\frac{q^{\sum_{m=1}^{J} (m-1) a_m}}{Z_j(J)}
w_{\{u,qu,\dots,q^{J-1}u\}}
\Big(
i_0,\{a_1,\dots,a_J\};k_0,\{\underbrace{1,\dots,1}_{n},\underbrace{0,\dots,0}_{J-n}\}
\Big)
\\
=
\frac{q^{\sum_{m=1}^{J} (m-1) c_m}}{Z_n(J)}
\sum_{\substack{a_1,\dots,a_J: |a| = j \\ d_1,\dots,d_J: |d| = n}}
\frac{q^{\sum_{m=1}^{J} (m-1) a_m}}{Z_j(J)}
w_{\{u,qu,\dots,q^{J-1}u\}}
\Big(
i_0,\{a_1,\dots,a_J\};k_0,\{d_1,\dots,d_J\}
\Big),
\end{multline}
where the latter expression is by definition equal to $w_u^{(J)}(i_0,j;k_0,n) (q^{\sum_{m=1}^{J} (m-1) c_m})/Z_n(J)$. Using this result in \eqref{rhs_concat2}, the sums over $c_1,\dots,c_J$ and $b_1,\dots,b_J$ can then be taken. The result of the calculation is $\sum_{n} w_u^{(J)}(i_0,j;k_0,n) w_u^{(J)}(i_1,n;k_1,\ell)$, which is just the left hand side of \eqref{rhs_concat2}.

\end{proof}

\subsection{Recursion relation}

It is easy to see that the fused vertex \eqref{fused_defn} obeys the following recursion relation, obtained by summing over all possible configurations of the lowest vertex in \eqref{fused_defn}:
\begin{align*}
w^{(J)}_u
\left(\fvert{i}{j}{k}{l}{0.5}\right)
&=
\frac{q^j Z_j(J-1)}{Z_j(J)}
\sum_{n=0}^{1}
w_u
\left(\vert{i}{0}{i-n}{n}{0.5}\right)
w^{(J-1)}_{qu}
\left(\fvert{i-n}{j}{k}{\ell-n}{0.5}\right)
\\
&+
\frac{q^{j-1} Z_{j-1}(J-1)}{Z_j(J)}
\sum_{n=0}^{1}
w_u
\left(\vert{i}{1}{i-n+1}{n}{0.5}\right)
w^{(J-1)}_{qu}
\left(\fvert{i-n+1}{j-1}{k}{\ell-n}{0.5}\right).
\end{align*}
Simplifying and converting to algebraic notation, this reads
\begin{align}
\label{fused_rec}
w^{(J)}_u(i,j;k,\ell)
&=
\frac{q^j-q^J}{1-q^J}
\sum_{n=0}^{1}
w_u(i,0;i-n,n) w^{(J-1)}_{qu}(i-n,j;k,\ell-n)
\\
\nonumber
&+
\frac{1-q^j}{1-q^J}
\sum_{n=0}^{1}
w_u(i,1;i-n+1,n) w^{(J-1)}_{qu}(i-n+1,j-1;k,\ell-n).
\end{align}
The recursion relation \eqref{fused_rec} can be solved to yield an explicit formula for the fused vertex weights. The solution is, however, rather complicated and involves $q$-hypergeometric functions \cite{Mangazeev,Borodin,CorwinP}:
%
%
\begin{multline}
\label{fused_weight}
w^{(J)}_u(i,j;k,\ell)
=
\\
\bm{1}_{i+j = k+\ell}
\frac{(-1)^{\ell-i}\ q^{i(i+2j-1)/2}\ s^{j-k}\ u^i\ (u/s;q)_{\ell-i}\ (s^2;q)_i}
{(su;q)_{k+\ell}\ (q^{J-j+1};q)_{j-\ell}\ (q;q)_i\ (s^2;q)_k}\ 
_{4}\bar{\phi}_{3}
\left( 
\begin{array}{c} 
q^{-k}; q^{-i}, q^{J} s u, qs/u \vspace{0.1cm} \\ s^2,q^{\ell-i+1},q^{J-k-\ell+1}
\end{array}
\Bigg|
q,q
\right),
\end{multline}
where
\begin{align*}
_{r+1}\bar{\phi}_{r}
\left(
\begin{array}{c}
q^{-n};a_1,\dots,a_r 
\\
b_1,\dots,b_r
\end{array}
\Big|
q,z
\right)
:=
\sum_{k=0}^{n}
z^k \frac{(q^{-n};q)_k}{(q;q)_k}
\prod_{i=1}^{r}
(a_i;q)_k (b_i q^k;q)_{n-k}.
\end{align*}
A higher-rank version of the Boltzmann weights \eqref{fused_weight} has recently been considered in \cite{BosnjakM}.

One can write a fused version of the Yang--Baxter equation \eqref{yb-eqn}, which now features an $\mathcal{R}$-matrix acting in a tensor product of spin-$J_1/2$ and spin-$J_2/2$ representations \cite{Mangazeev} (whereas the $\mathcal{R}$-matrix \eqref{R-mat} acts in a tensor product of spin-$1/2$ representations). The entries of the $\mathcal{R}$-matrix have essentially the same functional form as \eqref{fused_weight}, however we will bypass writing this Yang--Baxter equation in its full generality, quoting instead a much simpler special case of it below.

\subsection{Simplified vertex weights after setting $u=s$}

The vertex weights \eqref{fused_weight} greatly simplify at $u=s$, as is explained in \cite{Borodin,BorodinP2,BosnjakM}. Quoting Proposition 6.10 from \cite{BorodinP2}, we have 
\begin{align}
\label{weight_u=s}
w^{(J)}_{s}(i,j;k,\ell)
=
(\bm{1}_{i+j = k+\ell})
(\bm{1}_{i \geq \ell})
(-s q^J)^{\ell}
\frac{(q^{-J};q)_{\ell}\ (s^2 q^J;q)_{i-\ell}\ (q;q)_{k}}
{(q;q)_{\ell}\ (q;q)_{i-\ell}\ (s^2;q)_{k}},
\end{align}
for $i,k \in \mathbb{Z}_{\geq 0}$, $0 \leq j,\ell \leq J$. 
Let us also define an $\mathcal{R}$-matrix with components given by
\begin{align}
\label{fused_R}
\mathcal{R}^{(J,I)}(i,j;k,\ell)
=
(\bm{1}_{i+j = k+\ell})
(\bm{1}_{i \geq \ell})
(q^{J-I})^{\ell}
\frac{(q^{-J};q)_{\ell}\ (q^{J-I};q)_{i-\ell}\ (q;q)_{i}}
{(q;q)_{\ell}\ (q;q)_{i-\ell}\ (q^{-I};q)_{i}},
\end{align}
where $0 \leq i,k \leq I$ and $0 \leq j,\ell \leq J$.

\begin{thm}{\rm
Let $J_1$, $J_2$ be two positive integers, and fix two triples $(i_1,i_2,i_3)$, $(j_1,j_2,j_3)$ such that $0 \leq i_1,j_1 \leq J_1$, $0 \leq i_2,j_2 \leq J_2$ and $i_3,j_3 \in \mathbb{Z}_{\geq0}$. The Yang--Baxter equation holds:
\begin{multline}
\label{fused_yb}
\sum_{k_1=0}^{J_1}
\sum_{k_2=0}^{J_2}
\sum_{k_3=0}^{\infty}
\mathcal{R}^{(J_2,J_1)}(i_2,i_1;k_2,k_1)
w^{(J_1)}_{s}(i_3,k_1;k_3,j_1)
w^{(J_2)}_{s}(k_3,k_2;j_3,j_2)
=
\\
\sum_{k_1=0}^{J_1}
\sum_{k_2=0}^{J_2}
\sum_{k_3=0}^{\infty}
w^{(J_2)}_{s}(i_3,i_2;k_3,k_2)
w^{(J_1)}_{s}(k_3,i_1;j_3,k_1)
\mathcal{R}^{(J_2,J_1)}(k_2,k_1;j_2,j_1).
\end{multline}
}
\end{thm}

\begin{proof}
We will not give a detailed proof of \eqref{fused_yb}. It can be derived straightforwardly by writing a $(J_1+J_2)$-vertex version of the Yang--Baxter equation \eqref{yb-eqn}, then taking its $\{J_1,J_2;s,s\}$ specialization to effect the desired fusion.
\end{proof}

\subsection{Analytic continuation}
\label{sec:analytic}

We see that $q^{J_1}$ and $q^{J_2}$ play the role of spectral parameters in the relation \eqref{fused_yb}, even though the $\mathcal{R}$-matrix \eqref{fused_R} does not depend purely on the ratio of these parameters. This suggests that one should analytically continue $q^{J_1}$ and $q^{J_2}$ to arbitrary complex values, when \eqref{fused_yb} continues to hold (after removing the bounds on $i_1,j_1$ and $i_2,j_2$). This will be a key idea in our construction of the \wh polynomials. 

Let us perform the substitution $q^J \mapsto -x/s$ in \eqref{weight_u=s}, where $x \in \mathbb{C}$. We call the resulting weight $W_x(i,j;k,\ell)$; it is given by
\begin{align}
\label{fused_wt}
W_x
\left(
\begin{gathered}
\begin{tikzpicture}[scale=0.5,baseline=(current bounding box.center)]
\draw[lgray,line width=5pt] (-1,0) -- (1,0);
\draw[lgray,line width=5pt] (0,-1) -- (0,1);
\node[left] at (-0.8,0) {\tiny $j$};\node[right] at (0.8,0) {\tiny $\ell$};
\node[below] at (0,-0.8) {\tiny $i$};\node[above] at (0,0.8) {\tiny $k$};
\end{tikzpicture}
\end{gathered}
\right)
\equiv
W_x(i,j;k,\ell)
=
(\bm{1}_{i+j = k+\ell})
(\bm{1}_{i \geq \ell})\ 
x^{\ell}\
\frac{(-s/x;q)_{\ell}\ (-sx;q)_{i-\ell}\ (q;q)_{k}}
{(q;q)_{\ell}\ (q;q)_{i-\ell}\ (s^2;q)_{k}},
\end{align}
for all $i,j,k,\ell \in \mathbb{Z}_{\geq 0}$. Similarly, under $q^J \mapsto -x/s$ and $q^I \mapsto -y/s$, the $\mathcal{R}$-matrix \eqref{fused_R} becomes
\begin{align*}
\mathcal{R}_{x,y}(i,j;k,\ell)
=
(\bm{1}_{i+j = k+\ell})
(\bm{1}_{i \geq \ell})\
(x/y)^{\ell}\
\frac{(-s/x;q)_{\ell}\ (x/y;q)_{i-\ell}\ (q;q)_{i}}
{(q;q)_{\ell}\ (q;q)_{i-\ell}\ (-s/y;q)_{i}},
\end{align*}
for all $i,j,k,\ell \in \mathbb{Z}_{\geq 0}$.

\begin{cor}{\rm
Let $(i_1,i_2,i_3)$ and $(j_1,j_2,j_3)$ be two triples of non-negative integers. The following identity holds for arbitrary $x,y \in \mathbb{C}$:
\begin{multline}
\label{fused_yb_xy}
\sum_{k_1,k_2,k_3=0}^{\infty}
\mathcal{R}_{x,y}(i_2,i_1;k_2,k_1)
W_{y}(i_3,k_1;k_3,j_1)
W_{x}(k_3,k_2;j_3,j_2)
=
\\
\sum_{k_1,k_2,k_3=0}^{\infty}
W_{x}(i_3,i_2;k_3,k_2)
W_{y}(k_3,i_1;j_3,k_1)
\mathcal{R}_{x,y}(k_2,k_1;j_2,j_1).
\end{multline}

}
\end{cor}

\begin{proof}
After fixing finite triples of non-negative integers $(i_1,i_2,i_3)$ and $(j_1,j_2,j_3)$, equation \eqref{fused_yb} holds for infinitely many values of $J_1$ with fixed $J_2$, and infinitely many values of $J_2$ with fixed $J_1$. Furthermore, both sides of \eqref{fused_yb} are rational functions in $q^{J_1}$ and $q^{J_2}$; the functions must then be equal for all $q^{J_1} \in \mathbb{C}$ and $q^{J_2} \in \mathbb{C}$.
\end{proof}

\begin{remark}{\rm
The $s=0$ case of the Boltzmann weights \eqref{fused_wt}, and of the Yang--Baxter equation \eqref{fused_yb_xy}, first appeared in Section 3 of \cite{Korff}.
}
\end{remark}

\subsection{Dual vertex weights}

Up to this point we have examined the fusion procedure as applied to $J$-vertices in the vertex model \eqref{vertices}. Clearly, we could also adapt this approach to $J$-vertices in the dual model \eqref{red_vertices}; in fact the steps of Sections \ref{sec:geom}--\ref{sec:analytic} can be repeated virtually without modification, because of equation \eqref{wt_transform}, which relates the vertices \eqref{vertices} and \eqref{red_vertices} up to a conjugation applied to their vertical edges. We obtain, for generic spectral parameter $u$, a dual set of weights
\begin{align*}
w_{u}^{*(J)}
\left(
\begin{gathered}
\begin{tikzpicture}[scale=0.5,baseline=(current bounding box.center)]
\draw[lred,line width=5pt] (-1,0) -- (1,0);
\draw[lred,line width=5pt] (0,-1) -- (0,1);
\node[left] at (-0.8,0) {\tiny\re $j$};\node[right] at (0.8,0) {\tiny\re $\ell$};
\node[below] at (0,-0.8) {\tiny $i$};\node[above] at (0,0.8) {\tiny $k$};
\end{tikzpicture}
\end{gathered}
\right)
\equiv
w_{u}^{*(J)}(i,j;k,\ell)
\end{align*}
which are determined by the relation
\begin{align*}
w_u^{(J)}(i,j;k,\ell)
=
\frac{(q;q)_k}{(s^2;q)_k}
w_{u}^{*(J)}(k,j;i,\ell)
\frac{(s^2;q)_i}{(q;q)_i},
\quad
i,k \in \mathbb{Z}_{\geq 0},
\quad
0 \leq j,\ell \leq J,
\end{align*}
where $w_u^{(J)}(i,j;k,\ell)$ is given by \eqref{fused_weight}. It then follows from \eqref{weight_u=s} that, at $u=s$, one has
\begin{align}
\label{weight*_u=s}
w_{s}^{*(J)}(i,j;k,\ell)
=
(\bm{1}_{j+k = i+\ell})
(\bm{1}_{k \geq \ell})\
(-sq^J)^{\ell}\ 
\frac{(q^{-J};q)_{\ell}\ (s^2 q^J;q)_{k-\ell}\ (q;q)_{k}}
{(q;q)_{\ell}\ (q;q)_{k-\ell}\ (s^2;q)_{k}}.
\end{align}
Finally, after analytic continuation in $q^J$ and substituting $q^J \mapsto -x/s$, we arrive  at the vertex model  
\begin{align}
\label{fused_dual}
\bW_{x}
\left(
\begin{gathered}
\begin{tikzpicture}[scale=0.5,baseline=(current bounding box.center)]
\draw[lred,line width=5pt] (-1,0) -- (1,0);
\draw[lred,line width=5pt] (0,-1) -- (0,1);
\node[left] at (-0.8,0) {\tiny\re $j$};\node[right] at (0.8,0) {\tiny\re $\ell$};
\node[below] at (0,-0.8) {\tiny $i$};\node[above] at (0,0.8) {\tiny $k$};
\end{tikzpicture}
\end{gathered}
\right)
\equiv
\bW_{x}(i,j;k,\ell)
=
(\bm{1}_{j+k = i+\ell})
(\bm{1}_{k \geq \ell})\
x^{\ell}\ 
\frac{(-s/x;q)_{\ell}\ (-sx;q)_{k-\ell}\ (q;q)_{k}}
{(q;q)_{\ell}\ (q;q)_{k-\ell}\ (s^2;q)_{k}},
\end{align}
related to \eqref{fused_wt} via the transformation
\begin{align}
\label{dual_wt_transform}
W_{x}(i,j;k,\ell)
=
\frac{(q;q)_k}{(s^2;q)_k}
\bW_{x}(k,j;i,\ell)
\frac{(s^2;q)_i}{(q;q)_i},
\quad
i,j,k,\ell \in \mathbb{Z}_{\geq 0}.
\end{align}

\section{\Wh polynomials}
\label{sec:wh_poly}

\subsection{Geometric specialization of $\F_{\lambda}(u_1,\dots,u_{\mathcal{J}})$}
\label{sec:geom_spec}

Fix a set of positive integers $\{J_0,\dots,J_n\}$ and let $\mathcal{J} = \sum_{i=0}^{n} J_i$ denote their sum. Let $\F_{\lambda}(u_1,\dots,u_{\mathcal{J}})$ be a \hl function as defined in Section \ref{sec:hl-F}. 

Following Section \ref{sec:geom}, we take a $\{J_0,\dots,J_n;z_0,\dots,z_n\}$-specialization of $\F_{\lambda}(u_1,\dots,u_{\mathcal{J}})$, as dictated by \eqref{geom_spec}. When applying this specialization to $\F_{\lambda}(u_1,\dots,u_{\mathcal{J}})$, we group its horizontal spectral parameters into $n+1$ successive geometric progressions with ratio $q$, meaning that we are in the position to fuse the horizontal lines $(\mathcal{J}_{i-1}+1,\dots, \mathcal{J}_{i-1}+J_i)$ for all $0 \leq i \leq n$. The fusion is particularly simple: we note that the left and right boundary conditions of the lattice correspond with trivial summations over $a^{(i)}_1,\dots,a^{(i)}_{J_i}$ such that $\sum_{k=1}^{J_i} a^{(i)}_k = J_i$ and $b^{(i)}_1,\dots,b^{(i)}_{J_i}$ such that $\sum_{k=1}^{J_i} b^{(i)}_k = 0$, for all $0 \leq i \leq n$, where $a$'s and $b$'s are as in Theorem \ref{thm:horizontal_join}. We can therefore apply Theorem \ref{thm:horizontal_join} to obtain
\begin{align}
\label{fused_lattice1}
\rho_{\{J_0,\dots,J_n;z_0,\dots,z_n\}}
\Big(
\F_{\lambda}
\Big)
=
\begin{tikzpicture}[scale=0.75,baseline=(current bounding box.center),>=stealth]
\foreach\x in {0,...,4}{
\draw[lgray,line width=5pt] (\x,0) -- (\x,5);
}
\node[above] at (0,5) {\tiny $m_0$};
\node[above] at (1,5) {\tiny $m_1$};
\node[above] at (2,5) {$\cdots$};
\node[below] at (0,0) {\tiny $0$};
\node[below] at (1,0) {\tiny $0$};
\node[below] at (2,0) {$\cdots$};
\foreach\y in {1,...,4}{
\draw[lgray,line width=5pt] (-1,\y) -- (5,\y);
}
\node[left] at (-1.5,1) {$w^{(J_0)}_{z_0}\Big($}; \node[left] at (-0.9,1) {\tiny $J_0$};
\node[left] at (-1.5,2) {$w^{(J_1)}_{z_1}\Big($}; \node[left] at (-0.9,2) {\tiny $J_1$};
\node[left] at (-1.5,3.1) {$\vdots$};
\node[left] at (-1.5,4) {$w^{(J_n)}_{z_n}\Big($}; \node[left] at (-0.9,4) {\tiny $J_n$};
\node[left] at (5.5,1) {\tiny $0$}; \node[left] at (6,1) {$\Big)$};
\node[left] at (5.5,2) {\tiny $0$}; \node[left] at (6,2) {$\Big)$};
\node[left] at (5.5,3.1) {$\vdots$};
\node[left] at (5.5,4) {\tiny $0$}; \node[left] at (6,4) {$\Big)$};
\draw[->] (-1,0) -- (-0.5,0.5);
\end{tikzpicture}
\end{align}
where every vertex within the lattice is of the form \eqref{fused_weight}, and where we have indicated the value of the spin $J_i$ and the spectral parameter $z_i$ of each fused row by writing $w_{z_i}^{(J_i)}(\cdot)$ around it. Note that the partition state $\lambda$ along the top of the lattice is chosen such that $\sum_{i \geq 0} m_i(\lambda) = \mathcal{J}$, otherwise this partition function vanishes trivially.

\subsection{The limit $J_0 \rightarrow \infty$}
\label{sec:j0_infinite}

In the next step of the calculation we set $z_0=s$, which converts all vertices in the $0^{\rm th}$ row of the lattice to the simplified form \eqref{weight_u=s}, and take $J_0$ to infinity. In order to obtain a non-zero result, we also shift the value of $m_0$ at the top of the $0^{\rm th}$ column, by assuming that $m_0 \geq J_0$. 

Consider the vertex at the intersection of the $0^{\rm th}$ row and column, as indicated by the arrow in \eqref{fused_lattice1}. In taking $J_0 \rightarrow \infty$, its incoming data becomes $i = 0$, $j = \infty$, which by the form of \eqref{weight_u=s} constrains the outgoing data to $k = \infty$, $\ell=0$. The resulting Boltzmann weight is $(q;q)_{\infty} / (s^2;q)_{\infty}$. The remainder of the $0^{\rm th}$ row then only gives rise to empty vertices, which have weight $1$. The result of the calculation is thus
\begin{align}
\label{fused_lattice2}
\rho_{\{\infty,J_1,\dots,J_n;s,z_1,\dots,z_n\}}
\Big(
\F_{\{\infty,m_1,m_2,\dots\}}
\Big)
=
\frac{(q;q)_{\infty}}{(s^2;q)_{\infty}}
\times
\begin{tikzpicture}[scale=0.75,baseline=(current bounding box.center)]
\foreach\x in {0,...,4}{
\draw[lgray,line width=5pt] (\x,1) -- (\x,5);
}
\node[above] at (0,5) {\tiny $\infty$};
\node[above] at (1,5) {\tiny $m_1$};
\node[above] at (2,5) {$\cdots$};
\node[below] at (0,1) {\tiny $\infty$};
\node[below] at (1,1) {\tiny $0$};
\node[below] at (2,1) {$\cdots$};
\foreach\y in {2,...,4}{
\draw[lgray,line width=5pt] (-1,\y) -- (5,\y);
}
\node[left] at (-1.5,2) {$w^{(J_1)}_{z_1}\Big($}; \node[left] at (-0.9,2) {\tiny $J_1$};
\node[left] at (-1.5,3.1) {$\vdots$};
\node[left] at (-1.5,4) {$w^{(J_n)}_{z_n}\Big($}; \node[left] at (-0.9,4) {\tiny $J_n$};
\node[left] at (5.5,2) {\tiny $0$}; \node[left] at (6,2) {$\Big)$};
\node[left] at (5.5,3.1) {$\vdots$};
\node[left] at (5.5,4) {\tiny $0$}; \node[left] at (6,4) {$\Big)$};
\end{tikzpicture}
\end{align}
where we have removed the $0^{\rm th}$ row of the lattice entirely. 

\subsection{Stable \hl functions}

Before proceeding further in our study of equation \eqref{fused_lattice2}, we mention a particular case of it that will later be important: namely, the case $J_1=\cdots=J_n = 1$. In this case the $n$ rows of the lattice are not fused at all, and we return to the model of Section \ref{sec:vertex_model}, but now with infinitely many lattice paths entering and leaving the $0^{\rm th}$ column of the lattice. The vertices in the $0^{\rm th}$ column then have Boltzmann weights of the form
\begin{align}
\label{inf_boltz}
w_u
\left(
\begin{gathered}
\begin{tikzpicture}[scale=0.4,baseline=(current bounding box.center)]
\draw[lgray,thick] (-1,0) -- (1,0);
\draw[lgray,line width=5pt] (0,-1) -- (0,1);
\node[left] at (-0.8,0) {\tiny $j$};\node[right] at (0.8,0) {\tiny $\ell$};
\node[below] at (0,-0.8) {\tiny $\infty$};\node[above] at (0,0.8) {\tiny $\infty$};
\end{tikzpicture}
\end{gathered}
\right)
=
\frac{u^{\ell}}{1-su},
\quad
0 \leq j,\ell \leq 1,
\end{align}
and we remark that they are independent of the value of $j$, the left edge state. We use this partition function to define {\it stable} \hl functions:
\begin{align}
\label{def_stable}
\wt{\F}_{\lambda}(u_1,\dots,u_n)
:=
\bbra{\varnothing}
\wt{C}(u_1) \dots \wt{C}(u_n)
\kett{\lambda},
\quad\quad
\wt{C}(u)
:=
(1-su)
\bra{\infty}_0 C(u) \ket{\infty}_0,
\end{align}
where $\lambda = 1^{m_1} 2^{m_2} \dots \in {\rm Part}^{+}_n$ is a positive partition, and where we have employed the notation \eqref{pos_partition}, \ie\ $\bbra{\varnothing} = \bra{0}_1 \otimes \bra{0}_2 \otimes \cdots$ and $\kett{\lambda} = \ket{m_1}_1 \otimes \ket{m_2}_2 \otimes \cdots$. The multiplicative factor $(1-su)$ appearing in the definition of $\wt{C}(u)$ is to cancel the same factor appearing in the denominator of \eqref{inf_boltz}. The operators $\wt{C}(u)$ live in ${\rm End}(\wt{\mathbb{V}})$ and are given explicitly by
\begin{multline*}
\wt{C}_L(u) : 
\ket{k_1}_1 \otimes \cdots \otimes \ket{k_L}_L
\mapsto
\\
\sum_{0 \leq j \leq 1}\
\sum_{i_1,\dots,i_L \geq 0}
u^j\
w_u\Big(\{i_1,\dots,i_L\}, j ; \{k_1,\dots,k_L\}, 0 \Big)
\ket{i_1}_1 \otimes \cdots \otimes \ket{i_L}_L,
\end{multline*}
with $\wt{C}(u) = \lim_{L \rightarrow \infty} \wt{C}_L(u)$.

For any positive partition $\lambda \in {\rm Part}^{+}_{n-1}$, the functions \eqref{def_stable} satisfy the stability equation
\begin{align}
\label{stability}
\wt{\F}_{\lambda}(u_1,\dots,u_{n-1},0)
=
\wt{\F}_{\lambda}(u_1,\dots,u_{n-1}),
\end{align}
which can be easily verified using the lattice interpretation of $\wt{\F}_{\lambda}$ (namely, \eqref{fused_lattice2} with all $J_i=1$). Indeed, analysing the top row of the partition function \eqref{fused_lattice2}, we find that the only configuration which does not have a common factor of $u_n$ (and accordingly, does not vanish when $u_n=0$) is the following one:
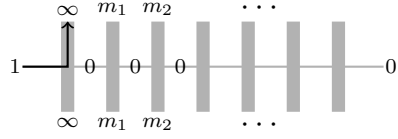
\begin{align*}
\begin{tikzpicture}[scale=0.6,baseline=(current bounding box.center)]
\draw[lgray,thick] (-1,0) -- (7,0);
\node[left] at (-0.8,0) {\tiny $1$};\node[right] at (6.8,0) {\tiny $0$};
\foreach\x in {0,...,6}{
\draw[lgray,line width=5pt] (\x,-1) -- (\x,1);
}
\draw[thick,->] (-1,0) -- (0,0) -- (0,1);
\node[below] at (0,-0.9) {\tiny $\infty$};\node[above] at (0,0.9) {\tiny $\infty$};
\node[below] at (1,-0.9) {\tiny $m_1$};\node[above] at (1,0.9) {\tiny $m_1$};
\node[below] at (2,-0.9) {\tiny $m_2$};\node[above] at (2,0.9) {\tiny $m_2$};
\node[below] at (4.3,-0.9) {$\cdots$};\node[above] at (4.3,0.9) {$\cdots$};
\node at (0.5,0) {\tiny $0$}; \node at (1.5,0) {\tiny $0$}; \node at (2.5,0) {\tiny $0$};
\end{tikzpicture}
\end{align*}
Using the Boltzmann weights \eqref{vertices} with $u=0$, we see that the weight of this configuration is $1$. The property \eqref{stability} is now immediate.

The functions \eqref{def_stable} appeared in essentially the same form in \cite{GarbaliGW}. There an explicit formula for $\wt{\F}_{\lambda}$, as a sum over the symmetric group, was obtained:
\begin{align}
\label{symmetriz}
\wt{\F}_{\lambda}(u_1,\dots,u_n)
=
\frac{(1-q)^n}{(q;q)_{n-\ell(\lambda)}}
\times
\sum_{\sigma \in S_n}
\sigma\left\{
\prod_{1 \leq i<j \leq n}
\left(
\frac{u_i-qu_j}{u_i-u_j}
\right)
\prod_{i=1}^{\ell(\lambda)}
\left(
\frac{u_i}{u_i-s}
\right)
\prod_{i=1}^{n}
\left(
\frac{u_i-s}{1-s u_i}
\right)^{\lambda_i}
\right\}.
\end{align}
This equation is parallel to similar formulae for $\F_{\lambda}$ and $\G_{\lambda}$ obtained in \cite[Section 5]{Borodin} and \cite[Section 4]{BorodinP2}. In fact, one can easily derive \eqref{symmetriz} by starting from the symmetrization formula for $\G_{\lambda}$ in \cite{Borodin} and taking the limit in which the number of lattice paths on the $0^{\rm th}$ column becomes infinite. Doing so leads to the partition function as shown in \eqref{fused_lattice2}, but in which the left edge states are $J_1 = \cdots = J_n = 0$ instead of $J_1 = \cdots = J_n =1$. This discrepancy between the left edge states is actually irrelevant, in view of the fact that the vertices \eqref{inf_boltz} do not depend on the value of $j$.

One can also define a dual version of the stable \hl functions:
\begin{align*}
\widetilde{\F}^{*}_{\lambda}(v_1,\dots,v_n)
:=
\bbra{\lambda} \widetilde{B}^{*}(v_n) \dots \widetilde{B}^{*}(v_1) \kett{\varnothing},
\quad
\widetilde{B}^{*}(v)
:=
(1-sv)
\bra{\infty}_0 D^{*}(v) \ket{\infty}_0,
\end{align*}
for all $\lambda \in {\rm Part}^{+}_n$. This time, the operators $\widetilde{B}^{*}(v) \in {\rm End}(\widetilde{\mathbb{V}})$ have the explicit action
\begin{multline*}
\wt{B}^{*}_L(v) : 
\ket{k_1}_1 \otimes \cdots \otimes \ket{k_L}_L
\mapsto
\\
\sum_{0 \leq j \leq 1}\
\sum_{i_1,\dots,i_L \geq 0}
v^j\
\bw_v\Big(\{i_1,\dots,i_L\}, j ; \{k_1,\dots,k_L\}, 0 \Big)
\ket{i_1}_1 \otimes \cdots \otimes \ket{i_L}_L,
\end{multline*}
with $\wt{B}^{*}(v) = \lim_{L \rightarrow \infty} \wt{B}^{*}_L(v)$. By virtue of the relation \eqref{wt_transform}, it is then evident that
\begin{align}
\label{tildeFc}
\widetilde{\F}^{*}_{\lambda}(v_1,\dots,v_n)
=
\tilde{\c}_{\lambda}(q;s)
\widetilde{\F}_{\lambda}(v_1,\dots,v_n),
\quad
\tilde{\c}_{\lambda}(q,s)
:=
\prod_{i \geq 1} \frac{(s^2;q)_{m_i(\lambda)}}{(q;q)_{m_i(\lambda)}}.
\end{align}

\begin{remark}{\rm
It is easily verified from \eqref{symmetriz} and \eqref{tildeFc} that $\wt{\F}_{\lambda}$ degenerates to the dual Hall--Littlewood polynomial $Q_{\lambda}$ at $s=0$, while $\wt{\F}^{*}_{\lambda}$ becomes equal to the $P_{\lambda}$ Hall--Littlewood polynomial.
}
\end{remark}

\subsection{Lattice construction of $\mathbb{F}_{\lambda}(x_1,\dots,x_n)$}
\label{sec:fused_F}

Returning to \eqref{fused_lattice2}, we set $z_1 = \cdots = z_n = s$, which converts all vertex weights to the form \eqref{weight_u=s}. The weights in the $0^{\rm th}$ column have a particularly simple expression. Indeed, taking the $i \rightarrow \infty$ limit of \eqref{weight_u=s}, we find that
\begin{align}
\label{0th-col}
w^{(J)}_{s}(\infty,j;k,\ell)
=
\bm{1}_{k=\infty}
(-s q^J)^{\ell}
\frac{(q^{-J};q)_{\ell} (s^2 q^J;q)_{\infty}}{(q;q)_{\ell}(s^2;q)_{\infty}},
\end{align}
since the value of $\ell$ must remain finite (otherwise the factor $(-sq^J)^{\ell}$ vanishes). Because of the indicator function $\bm{1}_{k=\infty}$, if infinitely many paths enter a vertex vertically from below, infinitely many paths will also exit vertically from above; this ensures that every vertex in the $0^{\rm th}$ column has a Boltzmann weight of the form \eqref{0th-col}. It is then clear that the $0^{\rm th}$ column of the lattice always produces the factor $\prod_{i=1}^{n} (s^2 q^{J_i};q)_{\infty}/(s^2;q)_{\infty}$ (irrespective of the path configuration along the right edges of the vertices in this column).

The weights \eqref{0th-col} are independent of the value of $j$, meaning that the partition function \eqref{fused_lattice2} does not depend on its left-edge states $\{J_1,\dots,J_n\}$. We may therefore reassign the left-edge states to any values; the most natural choice is $\{0,\dots,0\}$. Furthermore, we can effectively remove the restriction that internal horizontal edge states in the $i^{\rm th}$ row be bounded by $J_i$, since the Boltzmann weights \eqref{weight_u=s} vanish whenever $\ell > J_i$. 

In light of these observations, we see that $\rho_{\{\infty,J_1,\dots,J_n;s,s,\dots,s\}}(\F_{\{\infty,m_1,m_2,\dots\}})$ depends on $J_1,\dots,J_n$ only via $q^{J_1},\dots,q^{J_n}$. After dividing by the common factor $\prod_{i=1}^{n} (s^2 q^{J_i};q)_{\infty}/(s^2;q)_{\infty}$, it is easily shown to be polynomial in each $q^{J_i}$, where the boundedness of degree follows from the fact that $m_i = 0$ for all $i > N$, for sufficiently large $N$. Since we know this polynomial at infinitely many values $q^{J_i} \in \{q,q^2,\dots,\}$, we can analytically continue in $q^{J_i}$, writing $q^{J_i} = -x_i/s$, for all $1 \leq i \leq n$. This takes us to the partition function
\begin{align}
\label{fused_lattice}
\prod_{i=1}^{n} \frac{(s^2;q)_{\infty}}{(-s x_i;q)_{\infty}}
\times
\begin{tikzpicture}[scale=0.7,baseline=(current bounding box.center)]
\foreach\x in {0,...,4}{
\draw[lgray,line width=5pt] (\x,1) -- (\x,5);
}
\node[above] at (0,5) {\tiny $\infty$};
\node[above] at (1,5) {\tiny $m_1$};
\node[above] at (2,5) {\tiny $m_2$};
\node[above] at (3,5) {$\cdots$};
\node[below] at (0,1) {\tiny $\infty$};
\node[below] at (1,1) {\tiny $0$};
\node[below] at (2,1) {\tiny $0$};
\node[below] at (3,1) {$\cdots$};
\foreach\y in {2,...,4}{
\draw[lgray,line width=5pt] (-1,\y) -- (5,\y);
}
\node[left] at (-1.5,2) {$x_1$}; \node[left] at (-0.9,2) {\tiny $0$};
\node[left] at (-1.5,3.1) {$\vdots$};
\node[left] at (-1.5,4) {$x_n$}; \node[left] at (-0.9,4) {\tiny $0$};
\node[left] at (5.5,2) {\tiny $0$}; 
\node[left] at (5.5,3.1) {$\vdots$};
\node[left] at (5.5,4) {\tiny $0$};
\end{tikzpicture}
\end{align}
where each vertex is now of the form \eqref{fused_wt}, and we have dropped the factor $(q;q)_{\infty}/(s^2;q)_{\infty}$ from \eqref{fused_lattice2}, which turns out to be an unnecessary artefact of the calculation.

\begin{defn}{\rm
Fix a positive partition $\lambda$ and let $\lambda' = 1^{m_1(\lambda')} 2^{m_2(\lambda')} \dots$ be its conjugate. The \wh polynomial $\mathbb{F}_{\lambda}(x_1,\dots,x_n)$ is defined to be equal to the partition function shown in \eqref{fused_lattice} with vertex weights \eqref{fused_wt}, where $m_i \equiv m_i(\lambda')$ for all $i \geq 1$. 
}
\end{defn}

\begin{remark}{\rm
A couple of comments are in order regarding the partition function \eqref{fused_lattice}. First, it is a meaningful way to define $\mathbb{F}_{\lambda}$, even though it contains infinitely many columns. This can be argued, once again, using the fact that only the columns numbered $0$ to $\lambda'_1$ will contribute non-trivially to the partition function; the remaining columns will only feature vertices of the form $\fvert{0}{0}{0}{0}{0.3}$, with weight $1$.

Second, it might not be clear at this stage why we have chosen the $m_i$ in \eqref{fused_lattice} to correspond with the part-multiplicities of the conjugate partition $\lambda'$, rather than $\lambda$ itself. Our justification for doing so is the $\mathbb{F}_{\lambda}$, as defined, reduce to the $q$--Whittaker polynomials at $s=0$; see Section \ref{sec:reduction}.
}
\end{remark}

\begin{defn}{\rm
Fix positive partitions $\lambda,\mu$ such that $\lambda \supset \mu$, and let $\lambda' = 1^{m_1(\lambda')} 2^{m_2(\lambda')} \dots$ and $\mu' = 1^{m_1(\mu')} 2^{m_2(\mu')} \dots$ be their conjugates. The skew \wh polynomial $\mathbb{F}_{\lambda/\mu}(x_1,\dots,x_n)$ is defined as the partition function shown in Figure \ref{fig:skew_wh} (left panel) with vertex weights \eqref{fused_wt}, divided by $\prod_{i=1}^{n} (-sx_i;q)_{\infty}/(s^2;q)_{\infty}$. It reduces to the (ordinary) \wh polynomial for $\mu = \varnothing$.
}
\end{defn}

\begin{figure}
\begin{tabular}{cc}
\begin{tikzpicture}[scale=0.9,baseline=(current bounding box.center)]
\foreach\x in {0,...,4}{
\draw[lgray,line width=6pt] (\x,1) -- (\x,5);
}
\node[above] at (0,5) {\tiny $\infty$};
\node[above] at (0.9,5) {\tiny $m_1(\lambda')$};
\node[above] at (2,5) {\tiny $m_2(\lambda')$};
\node[above] at (3,5) {$\cdots$};
\node[below] at (0,1) {\tiny $\infty$};
\node[below] at (0.9,1) {\tiny $m_1(\mu')$};
\node[below] at (2,1) {\tiny $m_2(\mu')$};
\node[below] at (3,1) {$\cdots$};
\foreach\y in {2,...,4}{
\draw[lgray,line width=6pt] (-1,\y) -- (5,\y);
}
\node[left] at (-1.5,2) {$x_1$}; \node[left] at (-0.9,2) {\tiny $0$};
\node[left] at (-1.5,3.1) {$\vdots$};
\node[left] at (-1.5,4) {$x_n$}; \node[left] at (-0.9,4) {\tiny $0$};
\node[left] at (5.5,2) {\tiny $0$}; 
\node[left] at (5.5,3.1) {$\vdots$};
\node[left] at (5.5,4) {\tiny $0$};
\end{tikzpicture}
\quad
&
\quad
\begin{tikzpicture}[scale=0.9,baseline=(current bounding box.center)]
\foreach\x in {0,...,4}{
\draw[lred,line width=6pt] (\x,1) -- (\x,5);
}
\node[above] at (0,5) {\tiny $\infty$};
\node[above] at (0.9,5) {\tiny $m_1(\mu')$};
\node[above] at (2,5) {\tiny $m_2(\mu')$};
\node[above] at (3,5) {$\cdots$};
\node[below] at (0,1) {\tiny $\infty$};
\node[below] at (0.9,1) {\tiny $m_1(\lambda')$};
\node[below] at (2,1) {\tiny $m_2(\lambda')$};
\node[below] at (3,1) {$\cdots$};
\foreach\y in {2,...,4}{
\draw[lred,line width=6pt] (-1,\y) -- (5,\y);
}
\node[left] at (-1.5,2) {$x_n$}; \node[left] at (-0.9,2) {\tiny\re $0$};
\node[left] at (-1.5,3.1) {$\vdots$};
\node[left] at (-1.5,4) {$x_1$}; \node[left] at (-0.9,4) {\tiny\re $0$};
\node[left] at (5.5,2) {\tiny\re $0$}; 
\node[left] at (5.5,3.1) {$\vdots$};
\node[left] at (5.5,4) {\tiny\re $0$};
\end{tikzpicture}
\end{tabular}
\caption{Left panel: lattice construction of $\mathbb{F}_{\lambda/\mu}(x_1,\dots,x_n)$, employing the vertex weights \eqref{fused_wt}. Right panel: lattice construction of $\mathbb{F}^{*}_{\lambda/\mu}(x_1,\dots,x_n)$, using the vertex weights \eqref{fused_dual}.}
\label{fig:skew_wh}
\end{figure}
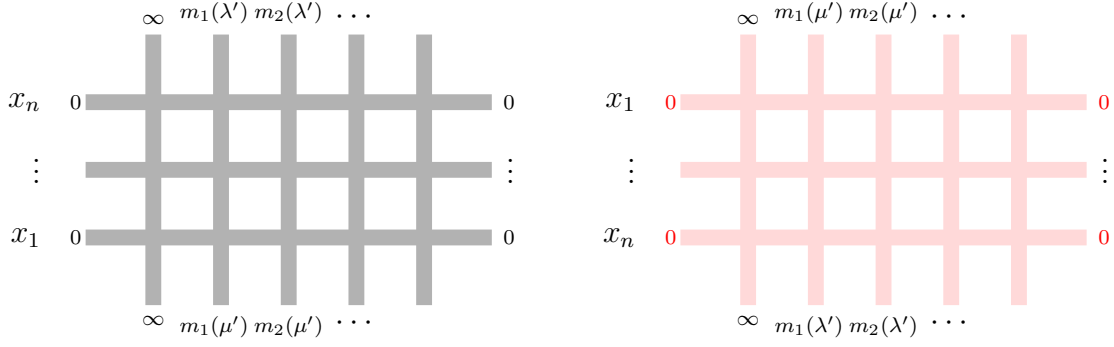

\subsection{Lattice construction of $\mathbb{G}_{\lambda}(x_1,\dots,x_n)$}
\label{sec:G_fusion}

In Sections \ref{sec:geom_spec}, \ref{sec:j0_infinite} and \ref{sec:fused_F} we gave a systematic approach that allows one to start from a \hl function $\F_{\lambda}$ and transform it into a \wh polynomial (up to adjustments of the normalization). One can now repeat this procedure, applying it to the dual \hl function $\G_{\lambda}$. 

This time we begin by fixing a set of positive integers $\{K_1,\dots,K_n\}$, with $\mathcal{K} = \sum_{i=1}^{n} K_i$ denoting their sum, and let 
$\G_{\lambda}(v_1,\dots,v_{\mathcal{K}})$ be the dual \hl function as defined in Section \ref{sec:hl-G}. We then take a $\{K_1,\dots,K_n;s,\dots,s\}$-specialization of 
$\G_{\lambda}(v_1,\dots,v_{\mathcal{K}})$, and send the number of lattice paths entering and exiting the $0^{\rm th}$ column to infinity (this can always be achieved, since one can send both $\ell$ and $m_0(\lambda)$ in Definition \ref{def:G} to infinity, while keeping $\ell-m_0(\lambda)$ finite).

After appropriate analytic continuation in $q^{K_1},\dots,q^{K_n}$, the result of the calculation is direct passage to the partition function \eqref{fused_lattice}, which we already obtained by applying the fusion procedure to $\F_{\lambda}$. We conclude that no new function is obtained in this case; accordingly, we make the identification 
$\mathbb{G}_{\lambda/\mu}(x_1,\dots,x_n) \equiv \mathbb{F}_{\lambda/\mu}(x_1,\dots,x_n)$ for all partitions $\lambda,\mu$, and will not mention $\mathbb{G}_{\lambda/\mu}(x_1,\dots,x_n)$ again in the sequel.

\subsection{The polynomial $\mathbb{F}^{*}_{\lambda/\mu}(x_1,\dots,x_n)$}

Motivated by equation \eqref{tildeFc}, we make the following definition:
\begin{align*}
\mathbb{F}^{*}_{\lambda/\mu}(x_1,\dots,x_n)
:=
\frac{\tilde{\c}_{\lambda'}(q,s)}{\tilde{\c}_{\mu'}(q,s)}
\mathbb{F}_{\lambda/\mu}(x_1,\dots,x_n),
\quad
\tilde{\c}_{\lambda'}(q,s)
=
\prod_{i \geq 1} \frac{(s^2;q)_{m_i(\lambda')}}{(q;q)_{m_i(\lambda')}}
=
\prod_{i \geq 1} \frac{(s^2;q)_{\lambda_i-\lambda_{i+1}}}{(q;q)_{\lambda_i-\lambda_{i+1}}}.
\end{align*}
We refer to these as {\it dual (skew) \wh polynomials.}

\begin{prop}{\rm
Fix positive partitions $\lambda,\mu$ such that $\lambda \supset \mu$, and let $\lambda' = 1^{m_1(\lambda')} 2^{m_2(\lambda')} \dots$ and $\mu' = 1^{m_1(\mu')} 2^{m_2(\mu')} \dots$ be their conjugates. The dual skew \wh polynomial $\mathbb{F}^{*}_{\lambda/\mu}(x_1,\dots,x_n)$ is equal to the partition function shown in Figure \ref{fig:skew_wh} (right panel) with vertex weights \eqref{fused_dual}, divided by $\prod_{i=1}^{n} (-sx_i;q)_{\infty}/(s^2;q)_{\infty}$.
}
\end{prop}

\begin{proof}
This is proved using the relation \eqref{dual_wt_transform}, and following essentially the same steps as in the proof of Proposition \ref{prop:Gc}.
\end{proof}

\subsection{Fused row operators}

In analogy with Section \ref{sec:row_hl}, we now define row operators which can be used to express the \wh polynomials in a purely algebraic way.\footnote{The fused row operators that we study have previously appeared, in the $s=0$ case, in \cite{Korff}, \cite{Korff2} and \cite{DuvalP}; these works also exposed a direct correspondence with Baxter's $Q$-operator.} By horizontally concatenating the vertices \eqref{fused_wt} or \eqref{fused_dual}, and summing over the states on all internal lattice edges, we obtain fused row vertices:
\begin{align*}
W_x
\left(
\begin{gathered}
\begin{tikzpicture}[scale=0.6,baseline=(current bounding box.center)]
\draw[lgray,line width=5pt] (0,0) -- (7,0);
\foreach\x in {1,...,6} {\draw[lgray,line width=5pt] (\x,-1) -- (\x,1);}
\node[left] at (0,0) {\tiny $j$};\node[right] at (7,0) {\tiny $\ell$};
\node[below] at (1,-1) {\tiny $i_1$};\node[above] at (1,1) {\tiny $k_1$};
\node[below] at (3,-0.8) {$\cdots$};\node[above] at (3,0.8) {$\cdots$};
\node[below] at (6,-1) {\tiny $i_L$};\node[above] at (6,1) {\tiny $k_L$};
\end{tikzpicture}
\end{gathered}
\right)
\equiv
W_x\Big(\{i_1,\dots,i_L\}, j ; \{k_1,\dots,k_L\}, \ell \Big)
\end{align*}
\begin{align*}
\bW_{x}
\left(
\begin{gathered}
\begin{tikzpicture}[scale=0.6,baseline=(current bounding box.center)]
\draw[lred,line width=5pt] (0,0) -- (7,0);
\foreach\x in {1,...,6} {\draw[lred,line width=5pt] (\x,-1) -- (\x,1);}
\node[left] at (0,0) {\tiny\re $j$};\node[right] at (7,0) {\tiny\re $\ell$};
\node[below] at (1,-1) {\tiny $i_1$};\node[above] at (1,1) {\tiny $k_1$};
\node[below] at (3,-0.8) {$\cdots$};\node[above] at (3,0.8) {$\cdots$};
\node[below] at (6,-1) {\tiny $i_L$};\node[above] at (6,1) {\tiny $k_L$};
\end{tikzpicture}
\end{gathered}
\right)
\equiv
\bW_{x}\Big(\{i_1,\dots,i_L\},j ; \{k_1,\dots,k_L\}, \ell \Big)
\end{align*}
where all indices $\{i_1,\dots,i_L\}, j, \{k_1,\dots,k_L\},\ell$ take values in $\mathbb{Z}_{\geq 0}$. From this, introduce two (infinite dimensional) monodromy matrices
\begin{align*}
\mathbb{T}(x) 
=
\left(
\begin{array}{ccc}
\mathbb{T}_x(0;0) & \mathbb{T}_x(0;1) & \cdots
\\
\mathbb{T}_x(1;0) & \mathbb{T}_x(1;1) & \cdots
\\
\vdots & \vdots & \ddots
\end{array}
\right),
\quad
\overline{\mathbb{T}}(x) 
=
\left(
\begin{array}{ccc}
\overline{\mathbb{T}}_{x}(0;0) & \overline{\mathbb{T}}_{x}(0;1) & \cdots
\\
\overline{\mathbb{T}}_{x}(1;0) & \overline{\mathbb{T}}_{x}(1;1) & \cdots
\\
\vdots & \vdots & \ddots
\end{array}
\right),
\end{align*}
whose entries act linearly on $V_1 \otimes \cdots \otimes V_L$ as follows:
\begin{align*}
\mathbb{T}_x(j;\ell) : \ket{k_1}_1 \otimes \cdots \otimes \ket{k_L}_L
\mapsto
\sum_{i_1,\dots,i_L \geq 0}
W_x\Big(\{i_1,\dots,i_L\}, j ; \{k_1,\dots,k_L\}, \ell \Big)
\ket{i_1}_1 \otimes \cdots \otimes \ket{i_L}_L,
\end{align*}
\begin{align*}
\overline{\mathbb{T}}_x(j;\ell) : 
\ket{k_1}_1 \otimes \cdots \otimes \ket{k_L}_L
\mapsto
\sum_{i_1,\dots,i_L \geq 0}
\bW_x\Big(\{i_1,\dots,i_L\}, j ; \{k_1,\dots,k_L\}, \ell \Big)
\ket{i_1}_1 \otimes \cdots \otimes \ket{i_L}_L.
\end{align*}
Similarly to the algebraic construction of the stable \hl functions, it turns out that certain linear combinations of the monodromy matrix entries $\mathbb{T}_x(j;\ell)$ and $\overline{\mathbb{T}}_x(j;\ell)$ are the most useful for constructing the \wh polynomials. We define two such linear combinations:
\begin{align}
\label{wh_ops}
\mathbb{C}(x)
:=
\lim_{L \rightarrow \infty}
\left(
\sum_{i=0}^{\infty}
x^{i} \frac{(-s/x;q)_{i}}{(q;q)_{i}}
\mathbb{T}_x(i;0)
\right),
\quad
\quad
\overline{\mathbb{B}}(x)
:=
\lim_{L \rightarrow \infty}
\left(
\sum_{i=0}^{\infty}
x^{i} \frac{(-s/x;q)_{i}}{(q;q)_{i}}
\overline{\mathbb{T}}_x(i;0)
\right),
\end{align}
passing also to the semi-infinite lattice, so that both $\mathbb{C}(x), \overline{\mathbb{B}}(x) \in {\rm End}(\wt{\mathbb{V}})$. The coefficients in the above sums originate from the right hand side of \eqref{0th-col} with $q^J = -s/x$; see the proof of Proposition \ref{prop:wh_exp} below.

\subsection{Algebraic formulation of \wh polynomials} 

Putting together the definitions of the previous subsections, we are now ready to express the \wh polynomials as expectation values of the operators \eqref{wh_ops}.

\begin{prop}
\label{prop:wh_exp}
{\rm
Let $n$ be a positive integer and fix two positive partitions $\lambda,\mu$ such that $\lambda \supset \mu$. Letting $\lambda',\mu'$ be the corresponding conjugate partitions, we have
\begin{align}
\label{skew-s-wh}
\mathbb{F}_{\lambda/\mu}(x_1,\dots,x_n)
&=
\bbra{\mu'} \mathbb{C}(x_1) \dots \mathbb{C}(x_n) \kett{\lambda'},
\\
\label{dual-skew-s-wh}
\mathbb{F}^{*}_{\lambda/\mu}(y_1,\dots,y_n)
&= 
\bbra{\lambda'} \ast{\mathbb{B}}(y_n) \dots \ast{\mathbb{B}}(y_1) \kett{\mu'},
\end{align}
where $\kett{\lambda'} = \otimes_{i \geq 1} \ket{m_i(\lambda')}_i$, and similarly for the remaining state vectors.
}
\end{prop}

\begin{proof}
For the proof of \eqref{skew-s-wh}, we start from the lattice expression for $\mathbb{F}_{\lambda/\mu}(x_1,\dots,x_n)$, shown on the left panel of Figure \ref{fig:skew_wh}. Counting from bottom to top, the $i^{\rm th}$ vertex in the $0^{\rm th}$ column is of the form $\fvert{\infty}{0}{\infty}{\ell_i}{0.3}$, with Boltzmann weight
\begin{align*}
W_{x_i}
\left(
\begin{gathered}
\begin{tikzpicture}[scale=0.5,baseline=(current bounding box.center)]
\draw[lgray,line width=5pt] (-1,0) -- (1,0);
\draw[lgray,line width=5pt] (0,-1) -- (0,1);
\node[left] at (-0.8,0) {\tiny $0$};\node[right] at (0.8,0) {\tiny $\ell_i$};
\node[below] at (0,-0.8) {\tiny $\infty$};\node[above] at (0,0.8) {\tiny $\infty$};
\end{tikzpicture}
\end{gathered}
\right)
= 
x_i^{\ell_i}\
\frac{(-s/x_i;q)_{\ell_i}\ (-sx_i;q)_{\infty}}
{(q;q)_{\ell_i}\ (s^2;q)_{\infty}},
\end{align*}
where $\ell_i \geq 0$ takes any non-negative integer value. Deleting the factor $\prod_{i=1}^{n} (-sx_i;q)_{\infty} / (s^2;q)_{\infty}$ which is common to the $0^{\rm th}$ column of the lattice (as per the definition of $\mathbb{F}_{\lambda/\mu}(x_1,\dots,x_n)$), and summing over all $\ell_1,\dots,\ell_n \geq 0$, the $i^{\rm th}$ row of the lattice then manifestly corresponds with the operator $\mathbb{C}(x_i)$ (\cf\ equation \eqref{wh_ops}, the definition of $\mathbb{C}(x)$). This proves the formula \eqref{skew-s-wh}.

The proof of \eqref{dual-skew-s-wh} is completely analogous; there, one starts from the lattice expression for $\mathbb{F}^{*}_{\lambda/\mu}(y_1,\dots,y_n)$, shown on the right panel of Figure \ref{fig:skew_wh}.

\end{proof}

\subsection{Exchange relations}

In this section we prove some important properties of the operators \eqref{wh_ops}. The first is their self-commutativity, Proposition \ref{prop:commute}. The relations \eqref{commute} allow us to show that the \wh polynomials are symmetric in their variables, which is not obvious from their lattice definition. The remaining properties are the exchange relations \eqref{CB}, \eqref{cB} and \eqref{Cb}, in Proposition \ref{prop:exchange}. These relations play a key role in deriving Cauchy-type summation identities involving the \wh polynomials.

\begin{prop}
\label{prop:commute}
{\rm
The operators \eqref{wh_ops} commute amongst themselves:
\begin{align}
\label{commute}
[\mathbb{C}(x),\mathbb{C}(y)]
=
[\overline{\mathbb{B}}(x),\overline{\mathbb{B}}(y)]
=
0,
\end{align}
for arbitrary complex parameters $x$ and $y$.
}
\end{prop}

\begin{proof}
A particularly simple case of the Yang--Baxter equation \eqref{fused_yb_xy} is recovered by choosing $i_1 = i_2 = j_1 = j_2 =0$:
\begin{multline*}
\sum_{k_1,k_2,k_3=0}^{\infty}
\mathcal{R}_{x,y}(0,0;k_2,k_1)
W_{y}(i_3,k_1;k_3,0)
W_{x}(k_3,k_2;j_3,0)
=
\\
\sum_{k_1,k_2,k_3=0}^{\infty}
W_{x}(i_3,0;k_3,k_2)
W_{y}(k_3,0;j_3,k_1)
\mathcal{R}_{x,y}(k_2,k_1;0,0).
\end{multline*}
The summation over $k_1$ and $k_2$ becomes trivial, since $\mathcal{R}_{x,y}(i,j;k,\ell) = 0$ unless $i+j = k+\ell$. This constrains $k_1$ and $k_2$ to be zero, and we read (after dropping $\mathcal{R}_{x,y}(0,0;0,0) = 1$ from both sides of the equation)
\begin{align}
\label{trivial_eqn}
\sum_{k_3}
W_{y}(i_3,0;k_3,0)
W_{x}(k_3,0;j_3,0)
=
\sum_{k_3}
W_{x}(i_3,0;k_3,0)
W_{y}(k_3,0;j_3,0),
\end{align}
which is true by inspection, since either side of this equation vanishes unless $i_3 = j_3 = k_3$. Moreover, using $L+1$ iterations of \eqref{fused_yb_xy} and applying the same logic as above, we obtain the following non-trivial commutation relation between row operators, generalizing \eqref{trivial_eqn}:
\begin{align}
\label{00exchange}
\begin{gathered}
\begin{tikzpicture}[scale=0.5,baseline=(current bounding box.center),>=stealth]
\foreach\x in {0,...,6}{
\draw[lgray,line width=5pt] (\x,2) -- (\x,4);
}
\draw[lgray,line width=5pt] (-1,3) node[left,black] {\tiny $0$} -- (7,3) node[right,black] {\tiny $0$};
\foreach\x in {0,...,6}{
\draw[lgray,line width=5pt] (\x,0) -- (\x,2);
}
\draw[lgray,line width=5pt] (-1,1) node[left,black] {\tiny $0$} -- (7,1) node[right,black] {\tiny $0$};
\node[above] at (0,4) {\tiny $m_0$};
\node[above] at (1,4) {\tiny $m_1$};
\node[above] at (3,4) {\tiny $\cdots$};
\node[above] at (6,4) {\tiny $m_L$};
\node[below] at (0,0) {\tiny $n_0$};
\node[below] at (1,0) {\tiny $n_1$};
\node[below] at (3,0) {\tiny $\cdots$};
\node[below] at (6,0) {\tiny $n_L$};
\node[left] at (-1.8,3) {$W_x\Big($}; \node[right] at (7.7,3) {$\Big)$};
\node[left] at (-1.8,1) {$W_y\Big($}; \node[right] at (7.7,1) {$\Big)$};
\end{tikzpicture}
\end{gathered}
=
\begin{gathered}
\begin{tikzpicture}[scale=0.5,baseline=(current bounding box.center),>=stealth]
\foreach\x in {0,...,6}{
\draw[lgray,line width=5pt] (\x,2) -- (\x,4);
}
\draw[lgray,line width=5pt] (-1,3) node[left,black] {\tiny $0$} -- (7,3) node[right,black] {\tiny $0$};
\foreach\x in {0,...,6}{
\draw[lgray,line width=5pt] (\x,0) -- (\x,2);
}
\draw[lgray,line width=5pt] (-1,1) node[left,black] {\tiny $0$} -- (7,1) node[right,black] {\tiny $0$};
\node[above] at (0,4) {\tiny $m_0$};
\node[above] at (1,4) {\tiny $m_1$};
\node[above] at (3,4) {\tiny $\cdots$};
\node[above] at (6,4) {\tiny $m_L$};
\node[below] at (0,0) {\tiny $n_0$};
\node[below] at (1,0) {\tiny $n_1$};
\node[below] at (3,0) {\tiny $\cdots$};
\node[below] at (6,0) {\tiny $n_L$};
\node[left] at (-1.8,3) {$W_y\Big($}; \node[right] at (7.7,3) {$\Big)$};
\node[left] at (-1.8,1) {$W_x\Big($}; \node[right] at (7.7,1) {$\Big)$};
\end{tikzpicture}
\end{gathered},
\end{align}
where $m_0,\dots,m_L$ and $n_0,\dots,n_L$ are arbitrary non-negative integers. Sending $m_0,n_0,L \rightarrow \infty$ and deleting the factor $(-sx;q)_{\infty} (-sy;q)_{\infty}  / (s^2;q)^2_{\infty}$ which becomes common to both sides of \eqref{00exchange} (\cf\ the proof of Proposition \ref{prop:wh_exp}), we obtain precisely $\mathbb{C}(y) \mathbb{C}(x) = \mathbb{C}(x) \mathbb{C}(y)$, as required.

The second relation, $[\mathbb{B}^{*}(x),\mathbb{B}^{*}(y)] = 0$, can be deduced from the first using the transformation property \eqref{dual_wt_transform}.
\end{proof}

\begin{cor}{\rm
The skew \wh polynomials are symmetric in their variables, by virtue of the commutativity \eqref{commute} of the operators in \eqref{skew-s-wh} and \eqref{dual-skew-s-wh}.
}
\end{cor}

\begin{prop}
\label{prop:exchange}
{\rm
The following exchange relations hold:
\begin{align}
\label{CB}
\mathbb{C}(x) \overline{\mathbb{B}}(y)
&=
\frac{(-sx;q)_{\infty} (-sy;q)_{\infty}}
{(s^2;q)_{\infty} (xy;q)_{\infty}}\ 
\overline{\mathbb{B}}(y) \mathbb{C}(x),
\quad
\text{for}\ |x|, |y| < 1,
\\
\label{cB}
\widetilde{C}(u) \overline{\mathbb{B}}(x)
&=
\left( \frac{1+ux}{1-su} \right)
\overline{\mathbb{B}}(x) \widetilde{C}(u),
\\
\label{Cb}
\mathbb{C}(x) \ast{\widetilde{B}}(u) 
&=
\left( \frac{1+ux}{1-su} \right)
\ast{\widetilde{B}}(u) \mathbb{C}(x),
\\
\label{cb}
\widetilde{C}(u) \ast{\widetilde{B}}(v) 
&=
\left( \frac{1-quv}{1-uv} \right)
\ast{\widetilde{B}}(v) \widetilde{C}(u),
\quad
\text{for}\ |(u-s)(v-s)| < |(1-su)(1-sv)|.
\end{align}
}
\end{prop}

\begin{remark}{\rm
The $s=0$ case of the relations \eqref{cB} and \eqref{Cb} was previously obtained by Duval and Pasquier in \cite[Section 6]{DuvalP}.
}
\end{remark}

\begin{proof}
We begin by noting the relation
\begin{align*}
\prod_{i=1}^{I} C(u_i)
\prod_{j=1}^{J} \b{D}(v_j)
=
\prod_{i=1}^{I}
\prod_{j=1}^{J}
\frac{1-qu_i v_j}{1-u_i v_j}
\prod_{j=1}^{J} \b{D}(v_j)
\prod_{i=1}^{I} C(u_i)
\end{align*}
which holds in ${\rm End}(\mathbb{V})$, provided $|(u_i-s)(v_j-s)| < |(1-s u_i)(1-s v_j)|$ for all $1 \leq i \leq I$ and $1 \leq j \leq J$. The graphical version of this relation is as follows:
\begin{align}
\label{master}
\begin{gathered}
\begin{tikzpicture}[scale=0.5,baseline=(current bounding box.center),>=stealth]
\foreach\x in {0,...,6}{
\draw[lred,line width=5pt] (\x,2) -- (\x,7);
}
\foreach\y in {3,...,6}{
\draw[lred,thick] (-1,\y) node[left,red] {\tiny $0$} -- (7,\y) node[right,red] {\tiny $0$};
}
\foreach\x in {0,...,6}{
\draw[lgray,line width=5pt] (\x,-2) -- (\x,2);
}
\foreach\y in {-1,...,1}{
\draw[lgray,thick] (-1,\y) node[left,black] {\tiny $1$} -- (7,\y) node[right,black] {\tiny $0$};
\draw[thick,->] (-1,\y) -- (0,\y);
}
\node[above] at (0,7) {\tiny $m_0$};
\node[above] at (1,7) {\tiny $m_1$};
\node[above] at (2,7) {\tiny $m_2$};
\node[above] at (3,7) {\tiny $\cdots$};
\node[below] at (0,-2) {\tiny $n_0$};
\node[below] at (1,-2) {\tiny $n_1$};
\node[below] at (2,-2) {\tiny $n_2$};
\node[below] at (3,-2) {\tiny $\cdots$};
\node[left] at (-1.8,6) {$v_1$};
\node[left] at (-1.8,4.5) {$\vdots$};
\node[left] at (-1.8,3) {$v_J$};
\node[left] at (-1.8,1) {$u_I$};
\node[left] at (-1.8,0) {$\vdots$};
\node[left] at (-1.8,-1) {$u_1$};
\end{tikzpicture}
\end{gathered}
=
\prod_{i=1}^{I}
\prod_{j=1}^{J}
\frac{1-qu_i v_j}{1-u_i v_j}
\times
\begin{gathered}
\begin{tikzpicture}[scale=0.5,baseline=(current bounding box.center),>=stealth]
\foreach\x in {0,...,6}{
\draw[lred,line width=5pt] (\x,-2) -- (\x,3);
}
\foreach\y in {-1,...,2}{
\draw[lred,thick] (-1,\y) node[left,red] {\tiny $0$} -- (7,\y) node[right,red] {\tiny $0$};
}
\foreach\x in {0,...,6}{
\draw[lgray,line width=5pt] (\x,3) -- (\x,7);
}
\foreach\y in {4,...,6}{
\draw[lgray,thick] (-1,\y) node[left,black] {\tiny $1$} -- (7,\y) node[right,black] {\tiny $0$};
\draw[thick,->] (-1,\y) -- (0,\y);
}
\node[above] at (0,7) {\tiny $m_0$};
\node[above] at (1,7) {\tiny $m_1$};
\node[above] at (2,7) {\tiny $m_2$};
\node[above] at (3,7) {\tiny $\cdots$};
\node[below] at (0,-2) {\tiny $n_0$};
\node[below] at (1,-2) {\tiny $n_1$};
\node[below] at (2,-2) {\tiny $n_2$};
\node[below] at (3,-2) {\tiny $\cdots$};
\node[left] at (-1.8,2) {$v_1$};
\node[left] at (-1.8,0.5) {$\vdots$};
\node[left] at (-1.8,-1) {$v_J$};
\node[left] at (-1.8,6) {$u_I$};
\node[left] at (-1.8,5) {$\vdots$};
\node[left] at (-1.8,4) {$u_1$};
\end{tikzpicture}
\end{gathered}
\end{align}
where we have assumed the arbitrary boundary conditions $\ket{m_0}_0 \otimes \ket{m_1}_1 \otimes \cdots \in \mathbb{V}$ and $\bra{n_0}_0 \otimes \bra{n_1}_1 \otimes \cdots \in \mathbb{V}^{*}$ at the top and base of the lattice, respectively. Each relation \eqref{CB}--\eqref{cb} can be deduced by specializing \eqref{master} in a different way, as we now show.

\underline{\it Proof of \eqref{CB}.} Starting from \eqref{master}, we perform an $\{I;s\}$-specialization of the variables $(u_1,\dots,u_I)$ and a $\{J;s\}$-specialization of the variables $(v_1,\dots,v_J)$.\footnote{In making these specializations, one should ensure that $|s^2(1-q^{i-1})(1-q^{j-1})| < |(1-s^2 q^{i-1})(1-s^2 q^{j-1})|$ for all $1 \leq i \leq I$, $1 \leq j \leq J$, since this is necessary for \eqref{master} to remain valid.} These specializations instigate fusion in the lattices on the left and right hand sides of \eqref{master}, and we obtain the relation
\begin{multline}
\label{proof47}
\begin{gathered}
\begin{tikzpicture}[scale=0.5,baseline=(current bounding box.center),>=stealth]
\foreach\x in {0,...,6}{
\draw[lred,line width=5pt] (\x,2) -- (\x,4);
}
\draw[lred,line width=5pt] (-1,3) node[left,red] {\tiny $0$} -- (7,3) node[right,red] {\tiny $0$};
\foreach\x in {0,...,6}{
\draw[lgray,line width=5pt] (\x,0) -- (\x,2);
}
\draw[lgray,line width=5pt] (-1,1) node[left,black] {\tiny $I$} -- (7,1) node[right,black] {\tiny $0$};
\node[above] at (0,4) {\tiny $m_0$};
\node[above] at (1,4) {\tiny $m_1$};
\node[above] at (2,4) {\tiny $m_2$};
\node[above] at (3,4) {\tiny $\cdots$};
\node[below] at (0,0) {\tiny $n_0$};
\node[below] at (1,0) {\tiny $n_1$};
\node[below] at (2,0) {\tiny $n_2$};
\node[below] at (3,0) {\tiny $\cdots$};
\node[left] at (-1.8,3) {$w_s^{*(J)}\Big($}; \node[right] at (7.7,3) {$\Big)$};
\node[left] at (-1.8,1) {$w_s^{(I)}\Big($}; \node[right] at (7.7,1) {$\Big)$};
\end{tikzpicture}
\end{gathered}
=
\\
\rho^u_{\{I;s\}}
\rho^v_{\{J;s\}}
\left(
\prod_{i=1}^{I}
\prod_{j=1}^{J}
\frac{1-qu_i v_j}{1-u_i v_j}
\right)
\times
\begin{gathered}
\begin{tikzpicture}[scale=0.5,baseline=(current bounding box.center),>=stealth]
\foreach\x in {0,...,6}{
\draw[lgray,line width=5pt] (\x,2) -- (\x,4);
}
\draw[lgray,line width=5pt] (-1,3) node[left,black] {\tiny $I$} -- (7,3) node[right,black] {\tiny $0$};
\foreach\x in {0,...,6}{
\draw[lred,line width=5pt] (\x,0) -- (\x,2);
}
\draw[lred,line width=5pt] (-1,1) node[left,red] {\tiny $0$} -- (7,1) node[right,red] {\tiny $0$};
\node[above] at (0,4) {\tiny $m_0$};
\node[above] at (1,4) {\tiny $m_1$};
\node[above] at (2,4) {\tiny $m_2$};
\node[above] at (3,4) {\tiny $\cdots$};
\node[below] at (0,0) {\tiny $n_0$};
\node[below] at (1,0) {\tiny $n_1$};
\node[below] at (2,0) {\tiny $n_2$};
\node[below] at (3,0) {\tiny $\cdots$};
\node[left] at (-1.8,1) {$w_s^{*(J)}\Big($}; \node[right] at (7.7,3) {$\Big)$};
\node[left] at (-1.8,3) {$w_s^{(I)}\Big($}; \node[right] at (7.7,1) {$\Big)$};
\end{tikzpicture}
\end{gathered},
\end{multline}
where the vertices that appear in these partition functions are either of the form \eqref{weight_u=s} or \eqref{weight*_u=s}. The multiplicative factor on the right hand side of \eqref{proof47} may be easily calculated; due to telescopic cancellations, one has
\begin{align*}
\rho^u_{\{I;s\}}
\rho^v_{\{J;s\}}
\left(
\prod_{i=1}^{I}
\prod_{j=1}^{J}
\frac{1-qu_i v_j}{1-u_i v_j}
\right)
=
\rho^u_{\{I;s\}}
\left(
\prod_{i=1}^{I}
\frac{1-s q^J u_i}{1-s u_i}
\right)
=
\frac{(s^2 q^J;q)_I}{(s^2;q)_I}
=
\frac{(s^2 q^I;q)_{\infty} (s^2 q^J;q)_{\infty}}
{(s^2;q)_{\infty} (s^2 q^{I+J};q)_{\infty}}.
\end{align*}
We now exploit the freedom to choose the states at the top and bottom of the lattices in \eqref{proof47}, sending $m_0, n_0 \rightarrow \infty$ simultaneously. Summing explicitly over the possible contributions of the $0^{\rm th}$ column on either side of \eqref{proof47}, and deleting some trivial common factors from the equation, we obtain
\begin{multline}
\label{proof47-2}
\sum_{i=0}^{I}
\sum_{j=0}^{J}
(-sq^I)^i (-sq^J)^j 
\frac{(q^{-I};q)_i}{(q;q)_i}
\frac{(q^{-J};q)_j}{(q;q)_j}
\begin{gathered}
\begin{tikzpicture}[scale=0.5,baseline=(current bounding box.center),>=stealth]
\foreach\x in {1,...,6}{
\draw[lred,line width=5pt] (\x,2) -- (\x,4);
}
\draw[lred,line width=5pt] (0,3) node[left,red] {\tiny $j$} -- (7,3) node[right,red] {\tiny $0$};
\foreach\x in {1,...,6}{
\draw[lgray,line width=5pt] (\x,0) -- (\x,2);
}
\draw[lgray,line width=5pt] (0,1) node[left,black] {\tiny $i$} -- (7,1) node[right,black] {\tiny $0$};
\node[above] at (1,4) {\tiny $m_1$};
\node[above] at (2,4) {\tiny $m_2$};
\node[above] at (3,4) {\tiny $\cdots$};
\node[below] at (1,0) {\tiny $n_1$};
\node[below] at (2,0) {\tiny $n_2$};
\node[below] at (3,0) {\tiny $\cdots$};
\node[left] at (-0.8,3) {$w_s^{*(J)}\Big($}; \node[right] at (7.7,3) {$\Big)$};
\node[left] at (-0.8,1) {$w_s^{(I)}\Big($}; \node[right] at (7.7,1) {$\Big)$};
\end{tikzpicture}
\end{gathered}
\\
=
\frac{(s^2 q^I;q)_{\infty} (s^2 q^J;q)_{\infty}}
{(s^2;q)_{\infty} (s^2 q^{I+J};q)_{\infty}}
\times
\quad\quad\quad\quad\quad\quad\quad\quad\quad\quad\quad\quad\quad\quad\quad\quad\quad\quad\quad\quad\quad
\\
\sum_{i=0}^{I}
\sum_{j=0}^{J}
(-sq^I)^i (-sq^J)^j 
\frac{(q^{-I};q)_i}{(q;q)_i}
\frac{(q^{-J};q)_j}{(q;q)_j}
\begin{gathered}
\begin{tikzpicture}[scale=0.5,baseline=(current bounding box.center),>=stealth]
\foreach\x in {1,...,6}{
\draw[lgray,line width=5pt] (\x,2) -- (\x,4);
}
\draw[lgray,line width=5pt] (0,3) node[left,black] {\tiny $i$} -- (7,3) node[right,black] {\tiny $0$};
\foreach\x in {1,...,6}{
\draw[lred,line width=5pt] (\x,0) -- (\x,2);
}
\draw[lred,line width=5pt] (0,1) node[left,red] {\tiny $j$} -- (7,1) node[right,red] {\tiny $0$};
\node[above] at (1,4) {\tiny $m_1$};
\node[above] at (2,4) {\tiny $m_2$};
\node[above] at (3,4) {\tiny $\cdots$};
\node[below] at (1,0) {\tiny $n_1$};
\node[below] at (2,0) {\tiny $n_2$};
\node[below] at (3,0) {\tiny $\cdots$};
\node[left] at (-0.8,1) {$w_s^{*(J)}\Big($}; \node[right] at (7.7,3) {$\Big)$};
\node[left] at (-0.8,3) {$w_s^{(I)}\Big($}; \node[right] at (7.7,1) {$\Big)$};
\end{tikzpicture}
\end{gathered},
\end{multline}
where $\ket{m_1}_1 \otimes \ket{m_2}_2 \otimes \cdots \in \widetilde{\mathbb{V}}$ and $\bra{n_1}_1 \otimes \bra{n_2}_2 \otimes \cdots \in \widetilde{\mathbb{V}}^{*}$ are two arbitrary states.

We relax the constraint that at most $I$ (resp. $J$) paths pass through each grey (coloured) horizontal edge, and replace both summations $\sum_{i=0}^{I}$ ($\sum_{j=0}^{J}$) in \eqref{proof47-2} by $\sum_{i=0}^{\infty}$ ($\sum_{j=0}^{\infty}$). These modifications of \eqref{proof47-2} are legal, since they only add vanishing terms to the equation: if the summation indices satisfy $i>I$ or $j>J$ on either side of \eqref{proof47-2}, the resulting term must vanish because of the factors $(q^{-I};q)_i$ and $(q^{-J};q)_j$, and a similar comment applies to the internal horizontal edge states. 

Equation \eqref{proof47-2} then becomes an equality between two functions in $q^I$ and $q^J$, which holds for all $I \geq 1$ and $J \geq 1$. The final step is analytic continuation\footnote{This is possible because both sides of \eqref{proof47-2} are clearly jointly analytic in $q^I$ and $q^J$, given that $q$, $sq^I$ and $sq^J$ are all in the unit disc by assumption.} in $q^I$ and $q^J$. This replaces the operators appearing in the equation by those in \eqref{wh_ops}, and converts the factor on the right hand side to $(-sx;q)_{\infty} (-sy;q)_{\infty} (s^2;q)_{\infty}^{-1} (xy;q)_{\infty}^{-1}$, completing the proof of \eqref{CB}.

\underline{\it Proof of \eqref{cB} and \eqref{Cb}.} Let us focus firstly on the proof of \eqref{cB}. Returning to \eqref{master}, we consider the special case $I=1$ (while keeping $J$ generic), and take a $\{J;s\}$-specialization of the variables $(v_1,\dots,v_J)$. This results in the equation
\begin{multline}
\label{proof47-3}
\begin{gathered}
\begin{tikzpicture}[scale=0.5,baseline=(current bounding box.center),>=stealth]
\foreach\x in {0,...,6}{
\draw[lred,line width=5pt] (\x,2) -- (\x,4);
}
\draw[lred,line width=5pt] (-1,3) node[left,red] {\tiny $0$} -- (7,3) node[right,red] {\tiny $0$};
\foreach\x in {0,...,6}{
\draw[lgray,line width=5pt] (\x,0) -- (\x,2);
}
\draw[lgray,thick] (-1,1) node[left,black] {\tiny $1$} -- (7,1) node[right,black] {\tiny $0$};
\draw[thick,->] (-1,1) -- (0,1);
\node[above] at (0,4) {\tiny $m_0$};
\node[above] at (1,4) {\tiny $m_1$};
\node[above] at (2,4) {\tiny $m_2$};
\node[above] at (3,4) {\tiny $\cdots$};
\node[below] at (0,0) {\tiny $n_0$};
\node[below] at (1,0) {\tiny $n_1$};
\node[below] at (2,0) {\tiny $n_2$};
\node[below] at (3,0) {\tiny $\cdots$};
\node[left] at (-1.8,3) {$w_s^{*(J)}\Big($}; \node[right] at (7.7,3) {$\Big)$};
\node[left] at (-1.8,1) {$w_u\Big($}; \node[right] at (7.7,1) {$\Big)$};
\end{tikzpicture}
\end{gathered}
=
\left(
\frac{1-s q^J u}{1-su}
\right)
\times
\begin{gathered}
\begin{tikzpicture}[scale=0.5,baseline=(current bounding box.center),>=stealth]
\foreach\x in {0,...,6}{
\draw[lgray,line width=5pt] (\x,2) -- (\x,4);
}
\draw[lgray,thick] (-1,3) node[left,black] {\tiny $1$} -- (7,3) node[right,black] {\tiny $0$};
\draw[thick,->] (-1,3) -- (0,3);
\foreach\x in {0,...,6}{
\draw[lred,line width=5pt] (\x,0) -- (\x,2);
}
\draw[lred,line width=5pt] (-1,1) node[left,red] {\tiny $0$} -- (7,1) node[right,red] {\tiny $0$};
\node[above] at (0,4) {\tiny $m_0$};
\node[above] at (1,4) {\tiny $m_1$};
\node[above] at (2,4) {\tiny $m_2$};
\node[above] at (3,4) {\tiny $\cdots$};
\node[below] at (0,0) {\tiny $n_0$};
\node[below] at (1,0) {\tiny $n_1$};
\node[below] at (2,0) {\tiny $n_2$};
\node[below] at (3,0) {\tiny $\cdots$};
\node[left] at (-1.8,1) {$w_s^{*(J)}\Big($}; \node[right] at (7.7,3) {$\Big)$};
\node[left] at (-1.8,3) {$w_u\Big($}; \node[right] at (7.7,1) {$\Big)$};
\end{tikzpicture}
\end{gathered},
\end{multline}
in which the coloured rows consist of fused vertices \eqref{weight*_u=s}, while the grey rows are unfused and consist of the vertices \eqref{vertices} of the original model. 

We then repeat the procedure used in the proof of \eqref{CB}, sending $m_0, n_0 \rightarrow \infty$ and summing explicitly over the possible contributions of the $0^{\rm th}$ column on either side of \eqref{proof47-3}. In this case, we find that
\begin{multline*}
\sum_{i=0}^{1}
\sum_{j=0}^{J}
u^i (-sq^J)^j 
\frac{(q^{-J};q)_j}{(q;q)_j}
\begin{gathered}
\begin{tikzpicture}[scale=0.5,baseline=(current bounding box.center),>=stealth]
\foreach\x in {1,...,6}{
\draw[lred,line width=5pt] (\x,2) -- (\x,4);
}
\draw[lred,line width=5pt] (0,3) node[left,red] {\tiny $j$} -- (7,3) node[right,red] {\tiny $0$};
\foreach\x in {1,...,6}{
\draw[lgray,line width=5pt] (\x,0) -- (\x,2);
}
\draw[lgray,thick] (0,1) node[left,black] {\tiny $i$} -- (7,1) node[right,black] {\tiny $0$};
\node[above] at (1,4) {\tiny $m_1$};
\node[above] at (2,4) {\tiny $m_2$};
\node[above] at (3,4) {\tiny $\cdots$};
\node[below] at (1,0) {\tiny $n_1$};
\node[below] at (2,0) {\tiny $n_2$};
\node[below] at (3,0) {\tiny $\cdots$};
\node[left] at (-0.8,3) {$w_s^{*(J)}\Big($}; \node[right] at (7.7,3) {$\Big)$};
\node[left] at (-0.8,1) {$w_u\Big($}; \node[right] at (7.7,1) {$\Big)$};
\end{tikzpicture}
\end{gathered}
\\
=
\left(
\frac{1-s q^J u}{1-su}
\right)
\times
\sum_{i=0}^{1}
\sum_{j=0}^{J}
u^i (-sq^J)^j 
\frac{(q^{-J};q)_j}{(q;q)_j}
\begin{gathered}
\begin{tikzpicture}[scale=0.5,baseline=(current bounding box.center),>=stealth]
\foreach\x in {1,...,6}{
\draw[lgray,line width=5pt] (\x,2) -- (\x,4);
}
\draw[lgray,thick] (0,3) node[left,black] {\tiny $i$} -- (7,3) node[right,black] {\tiny $0$};
\foreach\x in {1,...,6}{
\draw[lred,line width=5pt] (\x,0) -- (\x,2);
}
\draw[lred,line width=5pt] (0,1) node[left,red] {\tiny $j$} -- (7,1) node[right,red] {\tiny $0$};
\node[above] at (1,4) {\tiny $m_1$};
\node[above] at (2,4) {\tiny $m_2$};
\node[above] at (3,4) {\tiny $\cdots$};
\node[below] at (1,0) {\tiny $n_1$};
\node[below] at (2,0) {\tiny $n_2$};
\node[below] at (3,0) {\tiny $\cdots$};
\node[left] at (-0.8,1) {$w_s^{*(J)}\Big($}; \node[right] at (7.7,3) {$\Big)$};
\node[left] at (-0.8,3) {$w_u\Big($}; \node[right] at (7.7,1) {$\Big)$};
\end{tikzpicture}
\end{gathered}.
\end{multline*}
After analytically continuing in $q^J$ (letting $q^J \mapsto -x/s$) the factor on the right hand side of the commutation relation is converted to $(1+ux)/(1-su)$, and \eqref{cB} is immediate.

The proof of \eqref{Cb} is very similar, so we shall not present it in detail. For this proof, one considers the special case $J=1$ of \eqref{master} (leaving $I$ generic) and takes an $\{I;s\}$-specialization of $(u_1,\dots,u_I)$. Sending $m_0,n_0 \rightarrow \infty$ and analytically continuing in $q^I$ then produces \eqref{Cb}.

\underline{\it Proof of \eqref{cb}.} This relation is the simplest of all: it corresponds to the special case $I = J = 1$ of \eqref{master}. After taking $m_0, n_0 \rightarrow \infty$ and expanding over all possible contributions from the $0^{\rm th}$ lattice column, we obtain
\begin{multline*}
\sum_{i=0}^{1}
\sum_{j=0}^{1}
u^i v^j
\begin{gathered}
\begin{tikzpicture}[scale=0.5,baseline=(current bounding box.center),>=stealth]
\foreach\x in {1,...,6}{
\draw[lred,line width=5pt] (\x,2) -- (\x,4);
}
\draw[lred,thick] (0,3) node[left,red] {\tiny $j$} -- (7,3) node[right,red] {\tiny $0$};
\foreach\x in {1,...,6}{
\draw[lgray,line width=5pt] (\x,0) -- (\x,2);
}
\draw[lgray,thick] (0,1) node[left,black] {\tiny $i$} -- (7,1) node[right,black] {\tiny $0$};
\node[above] at (1,4) {\tiny $m_1$};
\node[above] at (2,4) {\tiny $m_2$};
\node[above] at (3,4) {\tiny $\cdots$};
\node[below] at (1,0) {\tiny $n_1$};
\node[below] at (2,0) {\tiny $n_2$};
\node[below] at (3,0) {\tiny $\cdots$};
\node[left] at (-0.8,3) {$w_v^{*}\Big($}; \node[right] at (7.7,3) {$\Big)$};
\node[left] at (-0.8,1) {$w_u\Big($}; \node[right] at (7.7,1) {$\Big)$};
\end{tikzpicture}
\end{gathered}
\\
=
\left(
\frac{1-q uv}{1-uv}
\right)
\times
\sum_{i=0}^{1}
\sum_{j=0}^{1}
u^i v^j
\begin{gathered}
\begin{tikzpicture}[scale=0.5,baseline=(current bounding box.center),>=stealth]
\foreach\x in {1,...,6}{
\draw[lgray,line width=5pt] (\x,2) -- (\x,4);
}
\draw[lgray,thick] (0,3) node[left,black] {\tiny $i$} -- (7,3) node[right,black] {\tiny $0$};
\foreach\x in {1,...,6}{
\draw[lred,line width=5pt] (\x,0) -- (\x,2);
}
\draw[lred,thick] (0,1) node[left,red] {\tiny $j$} -- (7,1) node[right,red] {\tiny $0$};
\node[above] at (1,4) {\tiny $m_1$};
\node[above] at (2,4) {\tiny $m_2$};
\node[above] at (3,4) {\tiny $\cdots$};
\node[below] at (1,0) {\tiny $n_1$};
\node[below] at (2,0) {\tiny $n_2$};
\node[below] at (3,0) {\tiny $\cdots$};
\node[left] at (-0.8,1) {$w_v^{*}\Big($}; \node[right] at (7.7,3) {$\Big)$};
\node[left] at (-0.8,3) {$w_u\Big($}; \node[right] at (7.7,1) {$\Big)$};
\end{tikzpicture}
\end{gathered},
\end{multline*}
which establishes the claim \eqref{cb}.
 
\end{proof}

\section{Combinatorial formulae}
\label{sec:combin}

In this section we examine some of the combinatorial properties of the \wh polynomials, arising from their definition as partition functions. In Sections \ref{sec:branch} and \ref{sec:skew-1v} they are shown to satisfy a simple branching rule, with factorized coefficients when branching off a single variable. Furthermore, the one-variable skew \wh polynomials have the so-called interlacing property: they vanish unless their two participating partitions interlace. In Section \ref{sec:reduction} we study the one-variable skew \wh polynomials at $s=0$, and find agreement with the standard $q$--Whittaker polynomials.

\subsection{Branching rules}
\label{sec:branch}

\begin{prop}{\rm
Let $m,n$ be two positive integers and fix two positive partitions $\lambda,\mu$ such that $\lambda \supset \mu$. The skew \wh polynomials satisfy the branching rules
\begin{align}
\label{branch3}
\mathbb{F}_{\lambda / \mu}(x_1,\dots,x_{m+n})
&=
\sum_{\nu}
\mathbb{F}_{\nu / \mu}(x_1,\dots,x_m)
\mathbb{F}_{\lambda / \nu}(x_{m+1},\dots,x_{m+n}),
\\
\label{branch2}
\mathbb{F}^{*}_{\lambda / \mu}(x_1,\dots,x_{m+n})
&=
\sum_{\nu}
\mathbb{F}^{*}_{\nu / \mu}(x_1,\dots,x_m)
\mathbb{F}^{*}_{\lambda / \nu}(x_{m+1},\dots,x_{m+n}),
\end{align} 
where both summations are taken over all partitions $\nu$ such that $\lambda \supset \nu \supset \mu$.
}
\end{prop}

\begin{proof}
This is an easy consequence of the algebraic expressions for the \wh polynomials. We start from \eqref{skew-s-wh} in the case of $m+n$ variables $(x_1,\dots,x_{m+n})$, inserting the identity $\sum_{\nu} \kett{\nu'} \bbra{\nu'}$ after the $m^{\rm th}$ $\mathbb{C}$-operator. We find that
\begin{align*}
\mathbb{F}_{\lambda/\mu}(x_1,\dots,x_{m+n})
=
\sum_{\nu}
\bbra{\mu'} \mathbb{C}(x_1) \dots \mathbb{C}(x_m) \kett{\nu'}
\bbra{\nu'} \mathbb{C}(x_{m+1}) \dots \mathbb{C}(x_{m+n}) \kett{\lambda'},
\end{align*}
and \eqref{branch3} follows immediately by reapplying \eqref{skew-s-wh}.

Note that the second branching rule \eqref{branch2} follows trivially from the first, by multiplying through by $\tilde{\c}_{\lambda'}(q,s)/ \tilde{\c}_{\mu'}(q,s)$.
\end{proof}

\subsection{One-variable skew \wh polynomials}
\label{sec:skew-1v}

\begin{thm}{\rm 
The one-variable skew \wh polynomials are given explicitly by
\begin{align}
\label{one-var1}
\mathbb{F}_{\mu/\nu}(x)
=
\left\{
\begin{array}{ll}
x^{|\mu|-|\nu|}
\displaystyle{
\prod_{i \geq 1}
\frac{(-s/x;q)_{\mu_i-\nu_i} (-sx;q)_{\nu_i-\mu_{i+1}} (q;q)_{\mu_i-\mu_{i+1}}}
{(q;q)_{\mu_i-\nu_i} (q;q)_{\nu_i-\mu_{i+1}} (s^2;q)_{\mu_i-\mu_{i+1}}}
},
\qquad
&
\mu \succ \nu,
\\ \\
0,
\qquad
&
{\rm otherwise,}
\end{array}
\right.
\end{align}

\begin{align}
\label{one-var2}
\mathbb{F}^{*}_{\mu/\nu}(x)
=
\left\{
\begin{array}{ll}
x^{|\mu|-|\nu|}
\displaystyle{
\prod_{i \geq 1}
\frac{(-s/x;q)_{\mu_i-\nu_i} (-sx;q)_{\nu_i-\mu_{i+1}} (q;q)_{\nu_i-\nu_{i+1}}}
{(q;q)_{\mu_i-\nu_i} (q;q)_{\nu_i-\mu_{i+1}} (s^2;q)_{\nu_i-\nu_{i+1}}}
},
\qquad
&
\mu \succ \nu,
\\ \\
0,
\qquad
&
{\rm otherwise.}
\end{array}
\right.
\end{align}

}
\end{thm}

\begin{proof}
We begin by writing down $\mathbb{F}_{\mu/\nu}(x)$ as a sum of single-row partition functions:
\begin{align}
\label{branch1}
\mathbb{F}_{\mu/\nu}(x)
=
\bbra{\nu'}
\mathbb{C}(x)
\kett{\mu'}
=
\sum_{j=0}^{\infty}
x^{j} \frac{(-s/x;q)_{j}}{(q;q)_{j}}
\times
W_x
\left(
\begin{gathered}
\begin{tikzpicture}[scale=0.5,baseline=(current bounding box.center)]
\draw[lgray,line width=5pt] (0,0) -- (7,0);
\foreach\x in {1,...,6} {\draw[lgray,line width=5pt] (\x,-1) -- (\x,1);}
\node[left] at (0.2,0) {\tiny $j$};\node[right] at (6.8,0) {\tiny $0$};
\node[below] at (1,-0.8) {\tiny $n'_1$};\node[above] at (1,0.8) {\tiny $m'_1$};
\node[below] at (6,-0.8) {\tiny $n'_M$};\node[above] at (6,0.8) {\tiny $m'_M$};
\end{tikzpicture}
\end{gathered}
\right)
\end{align}
where we use the abbreviations $m'_i = m_i(\mu') = \mu_i - \mu_{i+1}$ and $n'_i = m_i(\nu') = \nu_i - \nu_{i+1}$, and $M = \max\{\ell(\mu),\ell(\nu)\}$ (all vertices beyond the $M^{\rm th}$ column will have weight equal to $1$, so we may suppress them). We now read off the Boltzmann weights one by one. We start from the rightmost vertex in the product. Since $m'_M=\mu_M$ and $n'_M=\nu_M$ by the very definition of $M$, the weight of this vertex is given by
\begin{align*}
W_x
\left(
\begin{gathered}
\begin{tikzpicture}[scale=0.4,baseline=(current bounding box.center)]
\draw[lgray,line width=5pt] (-1,0) -- (1,0);
\draw[lgray,line width=5pt] (0,-1) -- (0,1);
\node[left] at (-0.8,0) {\tiny $j$};\node[right] at (0.8,0) {\tiny $0$};
\node[below] at (0,-0.8) {\tiny $\nu_M$};\node[above] at (0,0.8) {\tiny $\mu_M$};
\end{tikzpicture}
\end{gathered}
\right)
=
\left( \bm{1}_{j = \mu_M-\nu_M} \right)
\frac{ (-sx;q)_{\nu_M} (q;q)_{\mu_M}}
{(q;q)_{\nu_M} (s^2;q)_{\mu_M}},
\end{align*}
constraining the number of paths passing through the left edge to $\mu_M-\nu_M$ (and in particular, vanishing if $\nu_M > \mu_M$). Now observe that the vertex in the $i^{\rm th}$ column has a Boltzmann weight of the form 
\begin{multline}
\label{branch_induct}
W_x
\left(
\begin{gathered}
\begin{tikzpicture}[scale=0.4,baseline=(current bounding box.center)]
\draw[lgray,line width=5pt] (-1,0) -- (1,0);
\draw[lgray,line width=5pt] (0,-1) -- (0,1);
\node[left] at (-0.8,0) {\tiny $j$};\node[right] at (0.8,0) {\tiny $\mu_{i+1}-\nu_{i+1}$};
\node[below,text centered] at (0,-0.8) {\tiny $\nu_{i}-\nu_{i+1}$};
\node[above] at (0,0.8) {\tiny $\mu_{i}-\mu_{i+1}$};
\end{tikzpicture}
\end{gathered}
\right)
=
\\
\left( \bm{1}_{j = \mu_i-\nu_i} \right)
\left( \bm{1}_{\nu_i \geq \mu_{i+1}} \right)
x^{\mu_{i+1}-\nu_{i+1}}
\frac{(-s/x;q)_{\mu_{i+1}-\nu_{i+1}} (-sx;q)_{\nu_i-\mu_{i+1}} (q;q)_{\mu_i-\mu_{i+1}}}
{(q;q)_{\mu_{i+1}-\nu_{i+1}} (q;q)_{\nu_i-\mu_{i+1}} (s^2;q)_{\mu_i-\mu_{i+1}}}.
\end{multline}
Indeed, this clearly holds for $i=M$, and since $j$ is constrained to the value $\mu_i-\nu_i$, we conclude inductively that it holds for all $1 \leq i \leq M$. The indicator functions present in \eqref{branch_induct} ensure that $\mu_i \geq \nu_i \geq \mu_{i+1}$ for all $1 \leq i \leq M$; the total contribution of the $1^{\rm st}$ to $M^{\rm th}$ vertices is therefore
\begin{multline}
\label{1toM}
W_x
\left(
\begin{gathered}
\begin{tikzpicture}[scale=0.5,baseline=(current bounding box.center)]
\draw[lgray,line width=5pt] (0,0) -- (7,0);
\foreach\x in {1,...,6} {\draw[lgray,line width=5pt] (\x,-1) -- (\x,1);}
\node[left] at (0.2,0) {\tiny $\mu_1-\nu_1$};\node[right] at (6.8,0) {\tiny $0$};
\node[below] at (1,-0.8) {\tiny $\nu_1-\nu_2$};\node[above] at (1,0.8) {\tiny $\mu_1-\mu_2$};
\node[below] at (6,-0.8) {\tiny $\nu_M$};\node[above] at (6,0.8) {\tiny $\mu_M$};
\end{tikzpicture}
\end{gathered}
\right)
=
\\
\prod_{i=2}^{M}
\left(
x^{\mu_i-\nu_i}
\frac{(-s/x;q)_{\mu_i-\nu_i}}{(q;q)_{\mu_i-\nu_i}}
\right)
\prod_{i=1}^{M}
\frac{(-sx;q)_{\nu_i-\mu_{i+1}} (q;q)_{\mu_i-\mu_{i+1}}}
{(q;q)_{\nu_i-\mu_{i+1}} (s^2;q)_{\mu_i-\mu_{i+1}}}
\end{multline}
provided that $\mu \succ \nu$, vanishing otherwise. Combining this with the prefactor in \eqref{branch1}, and noting that the summation over $j$ is constrained to the value $j = \mu_1-\nu_1$, we conclude that
\begin{align*}
\mathbb{F}_{\mu/\nu}(x)
=
\prod_{i=1}^{M}
\left(
x^{\mu_i-\nu_i}
\frac{(-s/x;q)_{\mu_i-\nu_i} (-sx;q)_{\nu_i-\mu_{i+1}} (q;q)_{\mu_i-\mu_{i+1}}}
{(q;q)_{\mu_i-\nu_i} (q;q)_{\nu_i-\mu_{i+1}} (s^2;q)_{\mu_i-\mu_{i+1}}}
\right),
\quad
\mu \succ \nu,
\end{align*}
completing the proof of \eqref{one-var1}.

The second formula \eqref{one-var2} follows immediately from \eqref{one-var1} by multiplying it by 
\begin{align*}
\frac{\tilde{\c}_{\mu'}(q,s)}{\tilde{\c}_{\nu'}(q,s)} 
=
\prod_{i \geq 1}
\left(
\frac{(s^2;q)_{\mu_i-\mu_{i+1}}}{(q;q)_{\mu_i-\mu_{i+1}}}
\right)
\left(
\frac{(q;q)_{\nu_i-\nu_{i+1}}}{(s^2;q)_{\nu_i-\nu_{i+1}}}
\right).
\end{align*}

\end{proof}

\begin{cor}{\rm
Let $\mu$ be a positive partition of length $\ell$. Then $\mathbb{F}_{\lambda / \mu}(x_1,\dots,x_m)$ vanishes if the length of $\lambda$ exceeds $\ell+m$.
}
\end{cor}

\begin{proof}
By $m$ iterations of the branching rule, we have
\begin{align}
\label{GT}
\mathbb{F}_{\lambda / \mu}(x_1,\dots,x_m)
=
\sum_{\nu^{(0)} \prec \nu^{(1)} \prec \cdots \prec \nu^{(m)}}
\prod_{i=1}^{m}
\mathbb{F}_{\nu^{(i)}/\nu^{(i-1)}}(x_i),
\end{align}
where we define $\nu^{(0)} = \mu$ and $\nu^{(m)} = \lambda$, and sum over the remaining $m-1$ partitions. Because each partition $\nu^{(i)}$ interlaces $\nu^{(i-1)}$, its length can be at most one greater than its predecessor. It follows that $\ell(\nu^{(m)}) \equiv \ell(\lambda)$ is maximally $\ell(\nu^{(0)})+m \equiv \ell(\mu)+m$.
\end{proof}

\begin{cor}{\rm
The skew \wh polynomials satisfy the stability relation
\begin{align*}
\mathbb{F}_{\lambda / \mu}(x_1,\dots,x_{m-1},-s)
=
\mathbb{F}_{\lambda / \mu}(x_1,\dots,x_{m-1}),
\end{align*}
for all partitions $\lambda,\mu$.
}
\end{cor}

\begin{proof}
Isolating the dependence on $x_m$ in \eqref{GT} and setting $x_m=-s$, we have
\begin{align*}
\mathbb{F}_{\lambda / \mu}(x_1,\dots,x_{m-1},-s)
=
\sum_{\nu^{(0)} \prec \nu^{(1)} \prec \cdots \prec \nu^{(m)}}
\mathbb{F}_{\lambda/\nu^{(m-1)}}(-s)
\prod_{i=1}^{m-1}
\mathbb{F}_{\nu^{(i)}/\nu^{(i-1)}}(x_i),
\end{align*}
where we have defined $\nu^{(0)} = \mu$ and $\nu^{(m)}= \lambda$, as before. Examining equation \eqref{one-var1} for the one-variable skew \wh polynomials, we note that $\mathbb{F}_{\lambda/\nu}(x)$ contains the factor $\prod_{i \geq 1}(-s/x;q)_{\lambda_i-\nu_i}$, which vanishes when $x=-s$ if $\lambda_i > \nu_i$ for any $i$. Furthermore, it is clear that $\mathbb{F}_{\lambda/\nu}(-s)=1$ when $\lambda=\nu$. We conclude that $\mathbb{F}_{\lambda/\nu}(-s) = {\bm 1}_{\lambda = \nu}$, and therefore
\begin{align*}
\mathbb{F}_{\lambda / \mu}(x_1,\dots,x_{m-1},-s)
=
\sum_{\nu^{(0)} \prec \nu^{(1)} \prec \cdots \prec \nu^{(m-1)}}
\prod_{i=1}^{m-1}
\mathbb{F}_{\nu^{(i)}/\nu^{(i-1)}}(x_i),
\end{align*}
where the restriction $\nu^{(m-1)}= \lambda$ is now assumed. The final expression then matches \eqref{GT} in the case of $m-1$ variables.

\end{proof}

\subsection{Reduction to $q$--Whittaker polynomials}
\label{sec:reduction}

At $s=0$, the \wh polynomials reduce to ordinary $q$--Whittaker polynomials. This can be easily deduced from the explicit form \eqref{one-var1}, \eqref{one-var2} of the one-variable skew \wh polynomials:
\begin{align*}
\mathbb{F}_{\mu/\nu}(x)
\Big|_{s=0}
=
\left\{
\begin{array}{ll}
x^{|\mu|-|\nu|}
\displaystyle{
\prod_{i \geq 1}
\frac{(q;q)_{\mu_i-\mu_{i+1}}}
{(q;q)_{\mu_i-\nu_i} (q;q)_{\nu_i-\mu_{i+1}}}
},
\qquad
&
\mu \succ \nu,
\\ \\
0,
\qquad
&
{\rm otherwise,}
\end{array}
\right.
\end{align*}

\begin{align*}
\mathbb{F}^{*}_{\mu/\nu}(x)
\Big|_{s=0}
=
\left\{
\begin{array}{ll}
x^{|\mu|-|\nu|}
\displaystyle{
\prod_{i \geq 1}
\frac{(q;q)_{\nu_i-\nu_{i+1}}}
{(q;q)_{\mu_i-\nu_i} (q;q)_{\nu_i-\mu_{i+1}}}
},
\qquad
&
\mu \succ \nu,
\\ \\
0,
\qquad
&
{\rm otherwise,}
\end{array}
\right.
\end{align*}
which matches precisely with the one-variable skew Macdonald polynomials $P_{\mu/\nu}(x;q,t)$ and $Q_{\mu/\nu}(x;q,t)$ at $t=0$ (see Example 2 (b) in Section 6, Chapter VI of \cite{Macdonald} and set $t=0$). The partition function \eqref{skew-s-wh} thus reduces precisely to Korff's lattice model construction of the $q$--Whittaker polynomials \cite{Korff}, at $s=0$.

\section{Cauchy identities and Pieri rules}
\label{sec:cauchy_pieri}

In this section we derive a series of identities for the \wh polynomials. The first of these is a skew Cauchy identity, which is proved using the exchange relation \eqref{CB} for the fused row operators (in this way, the proof directly mirrors that of \eqref{HL_skew_C}, the skew Cauchy identity for the \hl functions). It reduces to a non-skew Cauchy identity for trivial skew Young diagrams (with an empty bottom partition), and that identity, in turn, can be considered as a multi parameter generalization of the $q$--Gauss summation theorem.

The second is a skew dual Cauchy identity, involving both a \wh polynomial and a stable \hl function, proved using the exchange relation \eqref{cB}. The appearance of a \wh polynomial and a \hl function in the same summation identity is suggestive of the existence of an involution which maps between the two families, much as $q$--Whittaker and Hall--Littlewood polynomials are related under the Macdonald involution \cite{Macdonald}. It would be very interesting to find such an involution, since it would provide some hope for the unification of \wh polynomials and \hl functions as specializations of a single ``spin Macdonald'' function.

Finally, we conclude with Pieri rules for the \wh polynomials. These are derived as simple corollaries of the skew Cauchy and dual skew Cauchy identities.

\subsection{Cauchy identity for \wh polynomials}

\begin{thm}{\rm
Fix two positive integers $m$ and $n$, and let $\mu$ and $\nu$ be two partitions. The \wh polynomials satisfy the following summation identity (assuming all parameters are in the unit disc):
\begin{multline}
\label{sk_wh_cauchy}
\sum_{\lambda}
\mathbb{F}_{\lambda/\mu}(x_1,\dots,x_m)
\mathbb{F}^{*}_{\lambda/\nu}(y_1,\dots,y_n)
=
\\
\prod_{i=1}^{m}
\prod_{j=1}^{n}
\left(
\frac{(-sx_i;q)_{\infty} (-sy_j;q)_{\infty}}
{(s^2;q)_{\infty} (x_i y_j;q)_{\infty}}
\right)
\sum_{\kappa}
\mathbb{F}_{\nu/\kappa}(x_1,\dots,x_m)
\mathbb{F}^{*}_{\mu/\kappa}(y_1,\dots,y_n),
\end{multline}
with the left hand sum taken over all partitions $\lambda$ such that $\lambda' \supset \mu'$ and $\lambda' \supset \nu'$, and the right hand sum taken over all partitions $\kappa$ such that $\kappa' \subset \mu'$ and $\kappa' \subset \nu'$.  
}
\end{thm}

\begin{proof}
This is essentially a repetition of the steps used to prove Theorem \ref{thm:hl_cauchy}. One starts by writing down the expectation value
\begin{align*}
\mathcal{E}_{\mu,\nu}(x_1,\dots,x_m; y_1,\dots,y_n)
:=
\bbra{\mu'} \mathbb{C}(x_1) \dots \mathbb{C}(x_m)
\mathbb{B}^{*}(y_n) \dots \mathbb{B}^{*}(y_1) \kett{\nu'},
\end{align*}
which, by virtue of \eqref{skew-s-wh}--\eqref{dual-skew-s-wh}, clearly expands as
\begin{align}
\label{1}
\mathcal{E}_{\mu,\nu}(x_1,\dots,x_m; y_1,\dots,y_n)
=
\sum_{\lambda}
\mathbb{F}_{\lambda/\mu}(x_1,\dots,x_m)
\mathbb{F}^{*}_{\lambda/\nu}(y_1,\dots,y_n).
\end{align}
On the other hand, by $mn$ iterations of the exchange relation \eqref{CB}, one has
\begin{align}
\nonumber
&
\mathcal{E}_{\mu,\nu}(x_1,\dots,x_m; y_1,\dots,y_n)
\\
\nonumber
&=
\prod_{i=1}^{m}
\prod_{j=1}^{n}
\left(
\frac{(-sx_i;q)_{\infty} (-sy_j;q)_{\infty}}
{(s^2;q)_{\infty} (x_i y_j;q)_{\infty}}
\right)
\bbra{\mu'} \mathbb{B}^{*}(y_n) \dots \mathbb{B}^{*}(y_1)
\mathbb{C}(x_1) \dots \mathbb{C}(x_m) \kett{\nu'}
\\
\label{2}
&=
\prod_{i=1}^{m}
\prod_{j=1}^{n}
\left(
\frac{(-sx_i;q)_{\infty} (-sy_j;q)_{\infty}}
{(s^2;q)_{\infty} (x_i y_j;q)_{\infty}}
\right)
\sum_{\kappa}
\mathbb{F}_{\nu/\kappa}(x_1,\dots,x_m)
\mathbb{F}^{*}_{\mu/\kappa}(y_1,\dots,y_n),
\end{align}
which completes the proof by matching \eqref{1} and \eqref{2}.
\end{proof}

\begin{cor}{\rm
For any two positive integers $m$ and $n$, the \wh polynomials satisfy the following Cauchy identity (assuming all parameters are in the unit disc):
\begin{align}
\label{wh_cauchy}
\sum_{\lambda}
\mathbb{F}_{\lambda}(x_1,\dots,x_m)
\mathbb{F}^{*}_{\lambda}(y_1,\dots,y_n)
=
\prod_{i=1}^{m}
\prod_{j=1}^{n}
\frac{(-sx_i;q)_{\infty} (-sy_j;q)_{\infty}}
{(s^2;q)_{\infty} (x_i y_j;q)_{\infty}}.
\end{align}
Note that this reproduces the Cauchy identity for ordinary $q$--Whittaker polynomials by setting $s=0$.
}
\end{cor}

\begin{proof}
This is immediate from equation \eqref{sk_wh_cauchy}, by choosing $\mu = \nu = \varnothing$. Such a choice trivializes the sum on the right hand side of \eqref{sk_wh_cauchy}: the only term remaining in this sum corresponds with $\kappa = \varnothing$. Since $\mathbb{F}_{\varnothing} = \mathbb{F}^{*}_{\varnothing} = 1$, the result follows.
\end{proof}

\subsection{$q$--Gauss summation identity as special case}

Taking the $m=n=1$ case of \eqref{wh_cauchy}, we recover the well-known $q$--Gauss summation identity
\begin{align}
\label{q-gauss}
\sum_{i=0}^{\infty}
\frac{(-s/x;q)_i (-s/y;q)_i}{(s^2;q)_i (q;q)_i}
(xy)^i
=
\frac{(-sx;q)_\infty (-sy;q)_\infty}{(s^2;q)_\infty (xy;q)_\infty}.
\end{align}
To check that the left hand side of \eqref{q-gauss} is indeed recovered from \eqref{wh_cauchy}, we note that at $m=n=1$ the left hand side of \eqref{wh_cauchy} is summed over all partitions $\lambda$ of at most one part. Using the explicit form of the one-variable \wh polynomials \eqref{one-var1} and \eqref{one-var2}, the left hand side of \eqref{wh_cauchy} becomes
\begin{align*}
1+
\sum_{i=1}^{\infty} 
\mathbb{F}_{i}(x)
\mathbb{F}^{*}_{i}(y)
=
1+
\sum_{i=1}^{\infty}
\left( x^i \frac{(-s/x;q)_i}{(s^2;q)_i} \right)
\left( y^i \frac{(-s/y;q)_i}{(q;q)_i} \right),
\end{align*}
and the formula \eqref{q-gauss} follows immediately.

\subsection{Dual Cauchy identity}\label{sec:dual-cauchy}

\begin{thm}{\rm
Fix two positive integers $m$ and $n$, and let $\mu$ and $\nu$ be two partitions. The stable \hl functions and \wh polynomials satisfy the following summation identity\footnote{We could also write this identity as
\begin{align}
\label{alternative}
\sum_{\lambda}
\widetilde{\F}^{*}_{\lambda/\mu}(u_1,\dots,u_m)
\mathbb{F}_{\lambda'/\nu'}(x_1,\dots,x_n)
=
\prod_{i=1}^{m}
\prod_{j=1}^{n}
\left(
\frac{1+u_i x_j}{1-s u_i}
\right)
\sum_{\kappa}
\widetilde{\F}^{*}_{\nu/\kappa}(u_1,\dots,u_m)
\mathbb{F}_{\mu'/\kappa'}(x_1,\dots,x_n),
\end{align}
simply by multiplying \eqref{skew_dual_cauchy} by $\tilde{\c}_{\nu}(q,s) / \tilde{\c}_{\mu}(q,s)$, and redistributing factors within the summations. The algebraic origin of this alternative identity is, of course, the commutation relation \eqref{Cb}.
}:
\begin{multline}
\label{skew_dual_cauchy}
\sum_{\lambda}
\widetilde{\F}_{\lambda/\mu}(u_1,\dots,u_m)
\mathbb{F}^{*}_{\lambda'/\nu'}(x_1,\dots,x_n)
=
\\
\prod_{i=1}^{m}
\prod_{j=1}^{n}
\left(
\frac{1+u_i x_j}{1-s u_i}
\right)
\sum_{\kappa}
\widetilde{\F}_{\nu/\kappa}(u_1,\dots,u_m)
\mathbb{F}^{*}_{\mu'/\kappa'}(x_1,\dots,x_n).
\end{multline}
}
\end{thm}

\begin{proof}
This is proved by arguments which should by now be familiar to the reader. The starting point is the expectation value
\begin{align*}
\mathcal{E}_{\mu,\nu}(u_1,\dots,u_m; x_1,\dots,x_n)
:=
\bbra{\mu} \widetilde{C}(u_1) \dots \widetilde{C}(u_m)
\mathbb{B}^{*}(x_n) \dots \mathbb{B}^{*}(x_1) \kett{\nu},
\end{align*}
which, using \eqref{def_stable}, \eqref{dual-skew-s-wh} and \eqref{cB}, can be expanded in two different ways in terms of skew stable \hl functions and skew \wh polynomials.
\end{proof}

\begin{cor}{\rm
For any two positive integers $m$ and $n$, the stable \hl functions and \wh polynomials satisfy the following dual Cauchy identity:
\begin{align}
\label{dual_cauchy}
\sum_{\lambda}
\widetilde{\F}_{\lambda}(u_1,\dots,u_m)
\mathbb{F}^{*}_{\lambda'}(x_1,\dots,x_n)
=
\sum_{\lambda}
\widetilde{\F}^{*}_{\lambda}(u_1,\dots,u_m)
\mathbb{F}_{\lambda'}(x_1,\dots,x_n)
=
\prod_{i=1}^{m}
\prod_{j=1}^{n}
\left(
\frac{1+u_i x_j}{1-s u_i}
\right).
\end{align}
Again, we note that this reproduces the correct dual Cauchy identity between Hall--Littlewood and $q$--Whittaker polynomials by setting $s=0$.
}
\end{cor}

\begin{proof}
This follows from equation \eqref{skew_dual_cauchy} by choosing $\mu = \nu = \varnothing$.
\end{proof}

\subsection{Cauchy identity for stable \hl functions}

\begin{thm}{\rm
Fix two positive integers $m$ and $n$, and let $\mu$ and $\nu$ be two partitions. The stable \hl functions satisfy the summation identity
\begin{align}
\label{stable_hl_cauchy}
\sum_{\lambda}
\widetilde{\F}_{\lambda/\mu}(u_1,\dots,u_m)
\widetilde{\F}^{*}_{\lambda/\nu}(v_1,\dots,v_n)
=
\prod_{i=1}^{m}
\prod_{j=1}^{n}
\left(
\frac{1-q u_i v_j}{1- u_i v_j}
\right)
\sum_{\kappa}
\widetilde{\F}_{\nu/\kappa}(u_1,\dots,u_m)
\widetilde{\F}^{*}_{\mu/\kappa}(v_1,\dots,v_n),
\end{align}
assuming that $|(u_i-s)(v_j-s)| < |(1-su_i)(1-sv_j)|$ for all $i,j$.
}
\end{thm}

\begin{proof}
This can be established in the same way as in the proof of the Cauchy identity \eqref{HL_skew_C} for (ordinary) \hl functions, but now using the commutation relation \eqref{cb}.
\end{proof}

\begin{cor}{\rm
For any two positive integers $m$ and $n$, the stable \hl functions satisfy the Cauchy identity
\begin{align*}
\sum_{\lambda}
\widetilde{\F}_{\lambda}(u_1,\dots,u_m)
\widetilde{\F}^{*}_{\lambda}(v_1,\dots,v_n)
=
\prod_{i=1}^{m}
\prod_{j=1}^{n}
\left(
\frac{1-q u_i v_j}{1- u_i v_j}
\right),
\end{align*}
assuming that $|(u_i-s)(v_j-s)| < |(1-su_i)(1-sv_j)|$ for all $i,j$.
}
\end{cor}

\begin{proof}
This is the $\mu = \nu = \varnothing$ case of \eqref{stable_hl_cauchy}.
\end{proof}

\subsection{Pieri rules}

In symmetric function theory, the {\it Pieri rules} constitute the simplest types of product formulae, \ie\ they are rules for multiplying a certain symmetric function by (typically) a more elementary one. Here we list two such formulae, which arise as special cases of the Cauchy and dual Cauchy identities for skew polynomials.

\underline{\it First Pieri rule for $\mathbb{F}_{\lambda}$.} We take the specialization $n=1$, $\mu=\varnothing$ of the identity \eqref{sk_wh_cauchy}. This trivializes the summation on the right hand side of \eqref{sk_wh_cauchy}, forcing $\kappa = \varnothing$. We then read
\begin{align*}
\sum_{\lambda}
\mathbb{F}_{\lambda}(x_1,\dots,x_m)
\mathbb{F}^{*}_{\lambda/\nu}(y)
=
\prod_{i=1}^{m}
\frac{(-sx_i;q)_{\infty} (-sy;q)_{\infty}}{(s^2;q)_{\infty} (x_i y;q)_{\infty}}\,
\mathbb{F}_{\nu}(x_1,\dots,x_m).
\end{align*}
This can be converted into a formula more closely resembling a product rule, by expanding the multiplicative factor appearing on the right hand side in terms of \wh polynomials. We are easily able to do that, using the Cauchy identity \eqref{wh_cauchy} itself at $n=1$, and the explicit expression \eqref{one-var2} for the one-variable dual \wh polynomial. We obtain the equation
\begin{align*}
\sum_{\lambda}
\mathbb{F}_{\lambda}(x_1,\dots,x_m)
\mathbb{F}^{*}_{\lambda/\nu}(y)
=
\left(1+
\sum_{i=1}^{\infty}
y^i \frac{(-s/y;q)_i}{(q;q)_i}
\mathbb{F}_i(x_1,\dots,x_m)
\right)
\mathbb{F}_{\nu}(x_1,\dots,x_m).
\end{align*}
This identity is the natural $s$-generalization of the ``horizontal'' Pieri rule for $q$--Whittaker polynomials: indeed, at $s=0$ the $y$ variable becomes a generating parameter that can be dropped from both sides of the equation, and we obtain
\begin{align*}
\sum_{\lambda \succ \nu: |\lambda|-|\nu|=i}
P_{\lambda}(x_1,\dots,x_m)
\varphi'_{\lambda/\nu}(q)
&=
\frac{1}{(q;q)_i}
P_i(x_1,\dots,x_m)
P_{\nu}(x_1,\dots,x_m),
\\
\text{where}\quad
\varphi'_{\lambda/\nu}(q)
:&=
\prod_{j \geq 1}
\frac{(q;q)_{\nu_j-\nu_{j+1}}}
{(q;q)_{\lambda_j-\nu_j} (q;q)_{\nu_j-\lambda_{j+1}}},
\end{align*}
expressing the product of a one-row $q$--Whittaker polynomial $P_i$ and a general $q$--Whittaker polynomial $P_{\nu}$ as a sum over the ways of adding weight $i$ horizontal strips to the starting Young diagram $\nu$.

\underline{\it Second Pieri rule for $\mathbb{F}_{\lambda}$.} We take the $m=1$, $\nu = \varnothing$ specialization of \eqref{alternative}. This reduces the summation on its right hand side to the single term $\kappa = \varnothing$, and we find that
\begin{align*}
\sum_{\lambda}
\mathbb{F}_{\lambda}(x_1,\dots,x_n)
\widetilde{\F}^{*}_{\lambda'/\mu'}(u)
=
\prod_{j=1}^{n}
\left(
\frac{1+u x_j}{1-s u}
\right)
\mathbb{F}_{\mu}(x_1,\dots,x_n),
\end{align*}
where we have also conjugated all partitions appearing in the identity. Expanding the multiplicative factor on the right hand side using the dual Cauchy identity \eqref{dual_cauchy}, we then have
\begin{multline*}
\sum_{\lambda}
\mathbb{F}_{\lambda}(x_1,\dots,x_n)
\widetilde{\F}^{*}_{\lambda'/\mu'}(u)
=
\\
\left( 1+\frac{u(1-s^2)}{1-su} \sum_{i=1}^{n} \left( \frac{u-s}{1-su} \right)^{i-1} 
\mathbb{F}_{1^i}(x_1,\dots,x_n) \right)
\mathbb{F}_{\mu}(x_1,\dots,x_n).
\end{multline*}
This identity, in turn, plays the role of an $s$-generalization of the ``vertical'' Pieri rule for $q$--Whittaker polynomials. At $s=0$, it becomes
\begin{align*}
\sum_{\lambda' \succ \mu': |\lambda|-|\mu|=i}
P_{\lambda}(x_1,\dots,x_n)
\psi_{\lambda'/\mu'}(q)
&=
P_{1^i}(x_1,\dots,x_n)
P_{\mu}(x_1,\dots,x_n),
\\
\text{where}\quad
\psi_{\lambda/\mu}(q)
:&=
\prod_{j \geq 1: m_j(\mu) = m_j(\lambda)+1}
\left(1-q^{m_j(\mu)} \right),
\end{align*}
expressing the product of a one-column $q$--Whittaker polynomial $P_{1^i}$ and a general $q$--Whittaker polynomial $P_{\mu}$ as a sum over the ways of adding weight $i$ vertical strips to the starting Young diagram $\mu$.

\section{Integral representation of \wh polynomials}
\label{sec:integral}

We conclude with an elegant integral formula for the \wh polynomials. The derivation of this formula is based on a known integral expression for the \hl function $\G_{\lambda}$, found in \cite{Borodin}. Indeed, since the \wh polynomials are obtained in a systematic way by the fusion/analytic continuation procedure of Section \ref{sec:wh_poly}, we need only apply these steps to the existing integral formula for $\G_{\lambda}$ and observe what we obtain at the end of the calculation.

\subsection{Multiple integral formula for $\G_{\lambda}(v_1,\dots,v_N)$}

Let $n \geq k$ be two positive integers, and begin by fixing a partition $\lambda \in {\rm Part}_n$, whose first $n-k$ parts are positive and whose final $k$ parts are equal to zero; \ie\ we have $\lambda_i \geq 1$ for all $1 \leq i \leq n-k$, and $\lambda_{n-k+1} = \cdots = \lambda_{n} = 0$. Choose another integer $N$ such that $N \geq n-k$. Quoting equation (7.8) from \cite{Borodin}, the following integral formula for \hl functions holds:
\begin{multline*}
\G_{\lambda}(v_1,\dots,v_N)
=
(s^2;q)_n \times
\\
\oint_{\mathcal{C}} \frac{du_1}{2\pi\i}
\cdots
\oint_{\mathcal{C}} \frac{du_n}{2\pi\i}
\prod_{1 \leq i<j \leq n}
\left(
\frac{u_i-u_j}{u_i-q u_j}
\right)
\prod_{i=1}^{n}
\left(
\frac{1}{(1-s u_i)(u_i-s)}
\left( 
\frac{1-su_i}{u_i-s}
\right)^{\lambda_i}
\prod_{j=1}^{N}
\frac{1-qu_i v_j}{1-u_i v_j}
\right),
\end{multline*}
where all integrations take place along the same positively-oriented contour $\mathcal{C}$. This contour is chosen such that {\bf 1.} The points $s^{-1}$ and $\{v_1^{-1},\dots,v_N^{-1}\}$ lie outside $\mathcal{C}$; {\bf 2.} The points $\{s,sq,\dots,sq^{n-1}\}$ lie inside $\mathcal{C}$; {\bf 3.} The image of $\mathcal{C}$ under multiplication by $q$, denoted $q\mathcal{C}$, lies completely inside $\mathcal{C}$. An example of a suitable contour is shown in Figure \ref{fig:contour}.

\begin{figure}
\begin{tikzpicture}[>=stealth]
\draw[gray,thick,->] (-2.5,0) -- (9,0);
\draw[gray,thick,->] (0,-3) -- (0,3);
\draw[thick,<-] (5,0) arc (0:-360:3 and 2);
\draw[densely dotted] (1.7,0) ellipse [x radius=2.5cm, y radius=1.5cm];
\node at (-0.2,-0.2) {\scriptsize $0$};
\node at (1.65,1.65) {\scriptsize$q\mathcal{C}$};
\node at (0.75,2) {\scriptsize$\mathcal{C}$};
\foreach\x in {0.5,0.9,1.5,2.3,3.3}{
\filldraw[fill=lgray] (\x,0) arc (0:-360:0.05);
}
\node[below,text centered] at (3.3,0) {\tiny$s$};
\node[below,text centered] at (2.3,0) {\tiny$sq$};
\node[below,text centered] at (1.5,0) {\tiny$\cdots$};
\node[below,text centered] at (0.5,0) {\tiny$sq^{n-1}$};
\foreach\x in {6.45,6.6,6.75,6.9,8}{
\filldraw[fill=lgray] (\x,0) arc (0:-360:0.05);
}
\node[below,text centered] at (6.6,0) {\tiny$\{v_j^{-1}\}$};
\node[below,text centered] at (8,0) {\tiny$s^{-1}$};
\end{tikzpicture}
\caption{A possible arrangement of $\mathcal{C}$ and poles in the plane.}
\label{fig:contour}
\end{figure}
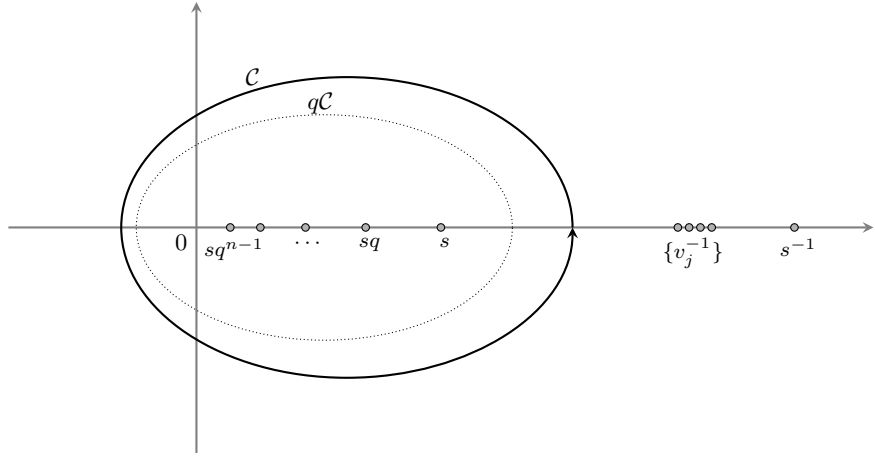
Following Section 7 of \cite{Borodin}, we integrate over the $u_n,\dots,u_{n-k+1}$ contours explicitly, starting with $u_n$ and working backwards sequentially. This is a straightforward calculation, in view of the assumption $\lambda_{n-k+1} = \cdots = \lambda_{n} = 0$. For each $j=n,\dots,n-k+1$, one readily sees that the $u_j$ contour only surrounds a simple pole at $u_j=sq^{n-j}$, whose residue can be taken immediately. After performing these integrations, the formula then reads
\begin{multline}
\label{eq7.9}
\G_{\lambda}(v_1,\dots,v_N)
=
\frac{(s^2;q)_n}{(s^2;q)_k}
\prod_{j=1}^{N}
\left(
\frac{1-sq^k v_j}{1-sv_j} 
\right)
\times
\oint_{\mathcal{C}} \frac{du_1}{2\pi\i}
\cdots
\oint_{\mathcal{C}} \frac{du_{n-k}}{2\pi\i}
\\
\prod_{1 \leq i<j \leq n-k}
\left(
\frac{u_i-u_j}{u_i-q u_j}
\right)
\prod_{i=1}^{n-k}
\left(
\frac{1}{(1-s u_i)(u_i-sq^k)}
\left( 
\frac{1-su_i}{u_i-s}
\right)^{\lambda_i}
\prod_{j=1}^{N}
\frac{1-qu_i v_j}{1-u_i v_j}
\right),
\end{multline}
\cf\ equation (7.9) in \cite{Borodin}.

\subsection{Fusion combined with analytic continuation}

We now apply the following steps to the integral \eqref{eq7.9}: {\bf 1.} We send $n,k \rightarrow \infty$ while keeping $n-k$ fixed and finite. For simplicity, we write $n-k \equiv \ell$; {\bf 2.} We take a $\{K_1,\dots,K_m;s,\dots,s\}$-specialization of the variables $(v_1,\dots,v_N)$, where it is assumed that $N = K_1+ \cdots + K_m$; {\bf 3.} The resulting expression depends on $K_1,\dots,K_m$ only via $q^{K_1},\dots,q^{K_m}$, allowing us to analytically continue in these variables, letting $q^{K_i} \mapsto -x_i/s$ for all $1 \leq i \leq m$; {\bf 4.} We normalize by dividing by $\prod_{i=1}^{m} (-sx_i;q)_{\infty}/(s^2;q)_{\infty}$, which must be a common factor of the final expression.

As we explained in Section \ref{sec:G_fusion}, performing these steps to the \hl function $\G_{\lambda}(v_1,\dots,v_N)$ transforms it exactly into the \wh polynomial $\mathbb{F}_{\lambda'}(x_1,\dots,x_m)$, where $\lambda'$ is the conjugate partition of $(\lambda_1,\dots,\lambda_{\ell})$. The result of these calculations will be, therefore, a multiple integral formula for $\mathbb{F}_{\lambda'}(x_1,\dots,x_m)$.

Applying the first step, we see immediately that
\begin{multline*}
\G_{\lambda}(v_1,\dots,v_N)
\Big|_{\rm Step\ {\bf 1}}
=
\prod_{j=1}^{N}
\left(
\frac{1}{1-sv_j} 
\right)
\times
\\
\oint_{\mathcal{C}} \frac{du_1}{2\pi\i u_1}
\cdots
\oint_{\mathcal{C}} \frac{du_{\ell}}{2\pi\i u_{\ell}}
\prod_{1 \leq i<j \leq \ell}
\left(
\frac{u_i-u_j}{u_i-q u_j}
\right)
\prod_{i=1}^{\ell}
\left(
\frac{1}{1-s u_i}
\left( 
\frac{1-su_i}{u_i-s}
\right)^{\lambda_i}
\prod_{j=1}^{N}
\frac{1-qu_i v_j}{1-u_i v_j}
\right),
\end{multline*}
where the contour of integration $\mathcal{C}$ is as before: all points $s q^i$, $i \in \mathbb{Z}_{\geq 0}$, and the point $0$, are contained within it. The geometric specialization\footnote{Note that setting each $v$ variable to $s q^i$ for some $i \in \mathbb{Z}_{\geq 0}$ is consistent with the assumption that the points $\{v_j^{-1}\}$ lie outside of $\mathcal{C}$.} of the second step yields
\begin{multline*}
\G_{\lambda}(v_1,\dots,v_N)
\Big|_{\rm Steps\ {\bf 1},{\bf 2}}
=
\prod_{i=1}^{m}
\frac{(s^2 q^{K_i};q)_{\infty}}{(s^2;q)_{\infty}} 
\times
\\
\oint_{\mathcal{C}} \frac{du_1}{2\pi\i u_1}
\cdots
\oint_{\mathcal{C}} \frac{du_{\ell}}{2\pi\i u_{\ell}}
\prod_{1 \leq i<j \leq \ell}
\left(
\frac{u_i-u_j}{u_i-q u_j}
\right)
\prod_{i=1}^{\ell}
\left(
\frac{1}{1-s u_i}
\left( 
\frac{1-su_i}{u_i-s}
\right)^{\lambda_i}
\prod_{j=1}^{m}
\frac{1-s u_i q^{K_j}}{1-s u_i}
\right),
\end{multline*}
and finally, the analytic continuation of the third step gives
\begin{multline*}
\G_{\lambda}(v_1,\dots,v_N)
\Big|_{\rm Steps\ {\bf 1},{\bf 2},{\bf 3}}
=
\prod_{i=1}^{m}
\frac{(-sx_i;q)_{\infty}}{(s^2;q)_{\infty}} 
\times
\\
\oint_{\mathcal{C}} \frac{du_1}{2\pi\i u_1}
\cdots
\oint_{\mathcal{C}} \frac{du_{\ell}}{2\pi\i u_{\ell}}
\prod_{1 \leq i<j \leq \ell}
\left(
\frac{u_i-u_j}{u_i-q u_j}
\right)
\prod_{i=1}^{\ell}
\left(
\frac{1}{1-s u_i}
\left( 
\frac{1-su_i}{u_i-s}
\right)^{\lambda_i}
\prod_{j=1}^{m}
\frac{1+u_i x_j}{1-s u_i}
\right).
\end{multline*}
After the normalization required by the fourth step (noting that the correct overall factor does indeed emerge from the calculation) we arrive at our desired  formula, which we now quote as a theorem:

\begin{thm}{\rm
Let $m$ be a positive integer, and $\lambda,\lambda'$ denote a partition and its conjugate, chosen such that $\ell(\lambda) = \lambda'_1 \leq m$. The \wh polynomial $\mathbb{F}_{\lambda}(x_1,\dots,x_m)$ is given by the integral expression
\begin{multline*}
\mathbb{F}_{\lambda}(x_1,\dots,x_m)
=
\\
\oint_{\mathcal{C}} \frac{du_1}{2\pi\i u_1}
\cdots
\oint_{\mathcal{C}} \frac{du_L}{2\pi\i u_L}
\prod_{1 \leq i<j \leq L}
\left(
\frac{u_i-u_j}{u_i-q u_j}
\right)
\prod_{i=1}^{L}
\left( 
\frac{1-su_i}{u_i-s}
\right)^{\lambda'_i}
\left(
\frac{\prod_{j=1}^{m}(1+u_i x_j)}{(1-s u_i)^{m+1}}
\right),
\end{multline*}
where $L=\lambda_1$ denotes the largest part of $\lambda$ and the contour $\mathcal{C}$ is as specified above; see Figure \ref{fig:contour}. An alternative way to arrive at this formula would be to use the orthogonality of $\F_{\lambda}$'s (\cf\ \cite{Borodin,BorodinP1,BorodinP2}) and the dual Cauchy identity \eqref{dual_cauchy}, above.
}
\end{thm}

\bibliographystyle{alpha}
\bibliography{references}

\begin{thebibliography}{GdGW17}

\bibitem[Bax07]{Baxter}
R.~J. Baxter.
\newblock {\em Exactly solved models in statistical mechanics}.
\newblock Courier Corporation, 2007.

\bibitem[BBW16]{BorodinBW}
A.~Borodin, A.~Bufetov, and M.~Wheeler.
\newblock Between the stochastic six vertex model and {H}all--{L}ittlewood
  processes.
\newblock {\em arXiv:1611.09486}, 2016.

\bibitem[BC14]{BorodinC}
A.~Borodin and I.~Corwin.
\newblock Macdonald processes.
\newblock {\em Probability Theory and Related Fields}, 158(1-2):225--400, 2014.

\bibitem[BM16]{BosnjakM}
G.~Bosnjak and V.~Mangazeev.
\newblock Construction of ${R}$-matrices for symmetric tensor representations
  related to ${U}_q(\widehat{sl_n})$.
\newblock {\em arXiv:1607.07968}, 2016.

\bibitem[Bor16]{Borodin2}
A.~Borodin.
\newblock Stochastic higher spin six vertex model and {M}acdonald measures.
\newblock {\em arXiv:1608.01553}, 2016.

\bibitem[Bor17]{Borodin}
A.~Borodin.
\newblock On a family of symmetric rational functions.
\newblock {\em Adv. Math.}, 306C:973--1018, 2017.

\bibitem[BP14]{BPlectures}
A.~Borodin and L.~Petrov.
\newblock Integrable probability: From representation theory to {M}acdonald
  processes.
\newblock {\em Probability Surveys}, 11:1--58, 2014.

\bibitem[BP16a]{BorodinP1}
A.~Borodin and L.~Petrov.
\newblock Higher spin six vertex model and symmetric rational functions.
\newblock {\em Selecta Math. New Ser.}, doi:10.1007/s00029-016-0301-7, 2016.

\bibitem[BP16b]{BorodinP2}
A.~Borodin and L.~Petrov.
\newblock Lectures on {I}ntegrable probability: {S}tochastic vertex models and
  symmetric functions.
\newblock {\em arXiv:1605.01349}, 2016.

\bibitem[CP16]{CorwinP}
I.~Corwin and L.~Petrov.
\newblock Stochastic higher spin vertex models on the line.
\newblock {\em Communications in Mathematical Physics}, 343(2):651--700, 2016.

\bibitem[DP15]{DuvalP}
A.~Duval and V.~Pasquier.
\newblock Pieri rules, vertex operators and {B}axter {Q}-matrix.
\newblock {\em arXiv:1510.08709}, 2015.

\bibitem[Eti99]{Etingof}
P.~Etingof.
\newblock Whittaker functions on quantum groups and $q$-deformed {T}oda
  operators.
\newblock {\em Translations of the American Mathematical Society-Series 2},
  194:9--26, 1999.

\bibitem[GdGW17]{GarbaliGW}
A.~Garbali, J.~de~Gier, and M.~Wheeler.
\newblock A new generalisation of {M}acdonald polynomials.
\newblock {\em Communications in Mathematical Physics},
  doi:10.1007/s00220-016-2818-1, 2017.

\bibitem[GLO10]{GLO1}
A.~Gerasimov, D.~Lebedev, and S.~Oblezin.
\newblock On $q$-deformed $\mathfrak{gl}_{\ell+1}$-{W}hittaker functions
  {I,II,III}.
\newblock {\em Comm. Math. Phys.}, 294:97--119, 121--143, 2010.

\bibitem[GLO12]{GLO2}
A.~Gerasimov, D.~Lebedev, and S.~Oblezin.
\newblock On a classical limit of $q$-deformed {W}hittaker functions.
\newblock {\em Lett. Math. Phys.}, 100:279--290, 2012.

\bibitem[KBI93]{KorepinBI}
V.~E. Korepin, N.~M. Bogoliubov, and A.~G. Izergin.
\newblock {\em Quantum inverse scattering method and correlation functions}.
\newblock Cambridge Monographs on Mathematical Physics. Cambridge University
  Press, Cambridge, 1993.

\bibitem[Kor13]{Korff}
C.~Korff.
\newblock Cylindric versions of specialised {M}acdonald functions and a
  deformed {V}erlinde algebra.
\newblock {\em Communications in Mathematical Physics}, 318(1):173--246, 2013.

\bibitem[Kor16]{Korff2}
C.~Korff.
\newblock From quantum {B}{\"a}cklund transforms to topological quantum field
  theory.
\newblock {\em Journal of Physics A: Mathematical and Theoretical},
  49(10):104001, 2016.

\bibitem[KR87]{KirillovR}
A.~Kirillov and N.~Reshetikhin.
\newblock Exact solution of the integrable {XXZ} {H}eisenberg model with
  arbitrary spin. {I}. {T}he ground state and the excitation spectrum.
\newblock {\em Journal of Physics A: Mathematical and General}, 20(6):1565,
  1987.

\bibitem[KRS81]{KulishRS}
P.~Kulish, N.~Reshetikhin, and E.~Sklyanin.
\newblock Yang--{B}axter equation and representation theory: {I}.
\newblock {\em Letters in Mathematical Physics}, 5(5):393--403, 1981.

\bibitem[Mac95]{Macdonald}
I.~G. Macdonald.
\newblock {\em Symmetric functions and {H}all polynomials}.
\newblock Oxford Mathematical Monographs. The Clarendon Press Oxford University
  Press, New York, second edition, 1995.
\newblock With contributions by A. Zelevinsky, Oxford Science Publications.

\bibitem[Man14]{Mangazeev}
V.~V. Mangazeev.
\newblock On the {Y}ang--{B}axter equation for the six-vertex model.
\newblock {\em Nuclear Physics B}, 882:70--96, 2014.

\bibitem[O'C12]{O'Connell}
N.~O'Connell.
\newblock Directed polymers and the quantum {T}oda lattice.
\newblock {\em The Annals of Probability}, 40(2):437--458, 2012.

\bibitem[OP16]{OrrPetrov}
D.~Orr and L.~Petrov.
\newblock Stochastic higher spin six vertex model and $q$-{TASEP}s.
\newblock {\em arXiv:1610.10080}, 2016.

\bibitem[Rui90]{Ruijsenaars}
S.~Ruijsenaars.
\newblock Relativistic {T}oda systems.
\newblock {\em Communications in Mathematical Physics}, 133(2):217--247, 1990.

\end{thebibliography}

\end{document}